\documentclass{article}

\IfFileExists{MinionPro.sty}
{\usepackage[lf]{MinionPro}

\usepackage{amsmath,amsthm,amsbsy}}
{\usepackage{times}
\usepackage{savesym} 
\usepackage{amsmath,amsthm,amsbsy}
\savesymbol{iint}
\savesymbol{openbox}
\usepackage{txfonts}

\DeclareSymbolFont{symbols}{OMS}{cmsy}{m}{n}
\restoresymbol{TXF}{iint}
\restoresymbol{TXF}{openbox}}

\usepackage{geometry}
\usepackage{color}
\definecolor{labelkey}{rgb}{0,0,.75}
\definecolor{MyGreen}{rgb}{0,.6,.2}
\definecolor{MyDarkBlue}{rgb}{.1,.1,.75}
\definecolor{violet}{rgb}{0.56, 0.0, 1.0}
\usepackage[colorlinks=true, citecolor=MyGreen, linkcolor=MyDarkBlue]{hyperref}
\usepackage{tensor}
\usepackage[all,cmtip]{xy}
\usepackage{graphicx}
\usepackage{enumitem}

\makeatletter
\newcommand{\optionaldesc}[2]{%
  \phantomsection
  #1:\protected@edef\@currentlabel{#1}\label{#2}%
}
\makeatother

\makeatletter
\newtheorem*{rep@theorem}{\rep@title}
\newcommand{\newreptheorem}[2]{%
\newenvironment{rep#1}[1]{%
 \def\rep@title{#2 \ref{##1}}%
 \begin{rep@theorem}}%
 {\end{rep@theorem}}}
\makeatother

\newreptheorem{theorem}{Theorem}

\numberwithin{equation}{section}

\newcommand{\Reals}{\mathbb R}

\newcommand{\Ints}{\mathbb Z}
\newcommand{\Nats}{\mathbb N}

\def\ip<#1,#2>{\left<#1,#2\right>}
\def\bignorm||#1||{\left|\left|#1\right|\right|}
\def\rsc#1{{\{#1\}}}
\def\rscr{\rsc{r}}
\def\rscq#1{{[#1]}}
\def\rscrq{\rscq{r}}
\def\Bv[#1,#2,#3]{B^{#1,#2}_{#3}}
\def\tildeBv[#1,#2,#3]{\widetilde B^{#1,#2}_{#3}}
\def\Bvloc[#1,#2,#3]{B^{#1,#2}_{#3,\rm{loc}}}
\def\Bvz[#1,#2,#3]{B^{#1,#2}_{#3,0}}
\theoremstyle{plain}
\newtheorem{theorem}{Theorem}
\newtheorem{lemma}[theorem]{Lemma}
\newtheorem{lemma-tp}[theorem]{Lemma (To Be Proved)}
\newtheorem{proposition}[theorem]{Proposition}
\newtheorem{corollary}[theorem]{Corollary}
\newtheorem{defn}[theorem]{Definition}

\newtheorem{assumption}[theorem]{Assumption}

\numberwithin{theorem}{section}
\numberwithin{equation}{section}

\newcommand\blfootnote[1]{%
  \begingroup
  \renewcommand\thefootnote{}\footnote{#1}%
  \addtocounter{footnote}{-1}%
  \endgroup
}

\newcounter{mnotecount}[section]
\let\oldmarginpar\marginpar
\renewcommand\marginpar[1]{\-\oldmarginpar[\raggedleft\footnotesize #1]%
{\raggedright\footnotesize #1}}

\title{A Scaling Approach to Elliptic Theory for Geometrically-Natural Differential Operators with Sobolev-Type Coefficients}
\author{Michael Holst, David Maxwell and  Gantumur Tsogtgerel}
\date{June 28, 2023}
\begin{document}
\maketitle

\begin{abstract}
We develop local elliptic regularity for operators having coefficients
in a range of Sobolev-type function spaces (Bessel potential, Sobolev-Slobodeckij,
Triebel-Lizorkin, Besov) where the coefficients have a regularity structure typical
of operators in geometric analysis.  The proofs rely on a nonstandard technique using
rescaling estimates and apply to operators having coefficients with low regularity.
For each class of function space for an operator's coefficients, we exhibit a natural associated
range of function spaces of the same type for the domain of the operator and we provide
regularity inference along with interior estimates.  Additionally,
we present a unified set of multiplication results for the function spaces we 
consider.
\end{abstract}

\section{Introduction}

Elliptic differential operators associated with Riemannian metrics 
having limited regularity arise naturally in the construction of
initial data in general relativity.  Moreover, because 
of connections with the associated evolution problem, it is natural to work with metrics
having regularity measured in Sobolev-type scales, and with a non-integral
number of derivatives \cite{klainerman_rough_2005}\cite{smith_sharp_2005}\cite{maxwell_rough_2006}.
In this paper we develop a largely self-contained account
of the mapping properties and local elliptic regularity theory for differential
operators having coefficients in any one of a broad category of 
Sobolev-type spaces, including spaces with non-integral levels of differentiability, 
where the coefficients also admit a regularity structure 
typical of geometric differential operators. 
For each category of function spaces considered,
we allow for coefficients with low regularity, and our approach to 
local elliptic theory is apparently novel in this context,
relying on rescaling estimates for Sobolev-type spaces to reduce the problem to 
that of constant-coefficient operators.

Local elliptic regularity is a well-established subject, with a wealth of results available in 
a number of contexts, even in low-regularity settings.  For second-order scalar elliptic operators
in divergence form \cite{trudinger_linear_1973} treats a form of elliptic regularity assuming 
only that the coefficients are measurable, although only for a very limited set of (operator-dependent) function spaces.  
A related theory appears in the text \cite{gilbarg_elliptic_2001} that 
applies to a range of integer-based Sobolev spaces 
under progressively stronger hypotheses on the coefficients of the elliptic operator. See also
\cite{giaquinta_introduction_1993}, which contains analogous results that apply to systems of equations.
Spaces with a non-integral number of derivatives include the $L^2$-based 
spaces $H^s$ which appear naturally in hyperbolic problems as well as 
their generalizations: Bessel potential spaces $H^{s,p}$, Sobolev-Slobodeckij 
spaces $W^{s,p}$, Triebel-Lizorkin spaces $F^{s,p}_q$, and
Besov spaces $\Bv[s,p,q]$.
So long as the differential operators involved have smooth coefficients, elliptic theory for these spaces
can be found in \cite{triebel_theory_2010}.  For less regular coefficients, one is led to 
the theory of pseudodifferential operators with non-smooth symbols and paradifferential calculus.
See, e.g., \cite{taylor_pseudodifferential_1991} and \cite{marschall_pseudodifferential_1988}.  Nevertheless,
this theory is somewhat technical, and it can be difficult for non-practitioners to apply it immediately
to the specific class of questions addressed in the current work.

Within the  mathematical relativity literature one finds 
instead a sequence of custom-made regularity theorems for second-order operators associated with 
with a metric $g$ on a domain of dimension $n\ge 3$:
\begin{itemize}
	\item \cite{choquet-bruhat_elliptic_1981}: $g\in W^{k,2}$ with $k\in\Nats$, $k>n/2+1$ and hence possessing H\"older continuous derivatives,
	\item \cite{choquet-bruhat_einstein_2004}\cite{maxwell_solutions_2005}: $g\in W^{k,p}$ with $k\in\Nats$, $k\ge 2$ and $k>n/p$ and hence H\"older continuous,
	\item \cite{maxwell_rough_2006}: $g\in H^s$ with $s\in\Reals$, $s>n/2$ and hence H\"older continuous,
	\item \cite{holst_rough_2009} $g\in W^{s,p}$ with $s\in\Reals$, $s\ge 1$ and $s>n/p$ and hence H\"older continuous.
\end{itemize}
Although \cite{maxwell_rough_2006} was the first work in this context to treat spaces
with a fractional number of derivatives, its limited focus on the
$L^2$ setting meant that it did not recover the full set of earlier results.
By contrast, \cite{holst_rough_2009} recovers prior results fully, but its main
regularity result, Lemma 32, contains an error that is not straightforward to correct.
Moreover, although
Sobolev-Slobodeckij spaces $W^{s,p}$ are a reasonable choice for interpolating
between integer-based Sobolev spaces, Bessel potential spaces $H^{s,p}$
enjoy better interpolation and embedding properties and 
are a compelling alternative. Hence it would be desirable to 
extend the results above to other classes of function spaces, and indeed
our work here concerning Bessel potential spaces 
provides the elliptic theory used by the recent preprint \cite{allen_sobolevclass_2022},
which treats geometric operators on asymptotically hyperbolic manifolds.

Our main results concern
local elliptic regularity for differential
operators having coefficients in any of the Sobolev-type spaces $H^{s,p}$, 
$W^{s,p}$, $F^{s,p}_q$ and $\Bv[s,p,q]$ mentioned above.  
Although we use basic techniques from the
theory of paraproducts in the proofs of some of our work, 
we do so with a minimum of theoretical 
overhead, and Section \ref{secsec:TL-rescaling} contains a short survey
of the few tools needed.  Moreover, although Bessel potential spaces are
a special case of Triebel-Lizorkin spaces and could have been dealt with
as a consequence of the general theory, in Section \ref{s:Hsp} we present
a simplified approach in the Bessel potential context 
that is free from paraproduct methods.  This approach comes
at the expense of establishing a less-than-sharp intermediate result 
on rescaling (Proposition \ref{prop:poor-mans-bp} vs. Proposition \ref{prop:rescale-Fsp}),
but this has no impact on the final regularity theory.  Readers who
are only interested in the Bessel potential case can stop reading
at the end of Section \ref{s:Hsp}
without needing to move on to the relative complexities of the 
more general function spaces.

In addition to extending the scope of
\cite{maxwell_rough_2006} and \cite{holst_rough_2009} to a broader class of function spaces,
the results of this paper
strengthen our earlier work.  Rather than simply obtaining a-priori estimates 
for functions with a known level of regularity, we obtain full
regularity inference in the spirit of, e.g., \cite{gilbarg_elliptic_2001}
Theorem 8.8.  Additionally, we have extended the range
of parameters of the function spaces treated. This extension is only
marginal in the generic case, but substantially extends the range of
parameters whenever the operators
involved omit low-order terms; see the discussion following
Definition \ref{def:L-Hsp}.
Although we restrict our attention to interior regularity, 
the tools developed here are also sufficient to address boundary value problems.
We have omitted these considerations, in part for simplicity of exposition: 
boundary traces
generally lie in Besov spaces, which are among the most technical of the spaces we consider, 
and which we treat last.
We will address boundary regularity in followup work.

Principal applications of elliptic regularity only apply
to the range of Lebesgue exponents $1<p<\infty$. Motivated by this, and again
for the sake of simple exposition, we have avoided the edge 
cases of $F^{s,p}_q$ and $\Bv[s,p,q]$ where $p,q=1,\infty$, much less 
the quasi-normed spaces where $p,q<1$. 
We observe, however, that the multiplication rules of Theorems \ref{thm:mult-Fsp} and \ref{thm:mult-Besov}
and the rescaling estimates of Propositions \ref{prop:rescale-Fsp} and \ref{prop:rescale-Bsp}
are candidates that could benefit from extending beyond the parameter
ranges treated here.

\subsection{Coefficient regularity structure}\label{secsec:intro-structure}
Differential operators in 
geometric analysis admit a representation in local coordinates in terms 
of coefficients that are universal expressions involving the
values and derivatives of the coordinate representation $g_{ab}$ of 
a metric $g$.  The prototypical example is
the Laplacian $\Delta_g$ associated with $g$, which can be written
in terms of the inverse metric $g^{ab}$ and the determinant $\sqrt{g}$ as
\[
\Delta_g = g^{ab}\partial_a\partial_b + \sum_{a} \sqrt{g}\left(\partial_a\left(\frac1{\sqrt g} g^{ab}\right)\right)\partial_b.
\]
The leading order coefficients have the regularity of $g_{ab}$, whereas
the next order coefficients involve first derivatives of $g_{ab}$.
More generally, consider 
the conformal Laplacian $\mathcal C_g$ of $g$,
\[
\mathcal C_g=-c_n \Delta_g + R_g
\]
where $c_n=-4(n-1)/(n-2)$ and where $R_g$
is the scalar curvature of the metric. If 
the coefficients of the metric lie in $W^{1,p}_{\rm loc}$ with $p>n$, 
a computation using H\"older's inequality and Sobolev embedding shows that
$\mathcal C_g$ has the form
\[
\mathcal C_g = -c_n g^{ab}\partial_a\partial_b + \beta^a \partial_a +\gamma
\]
where
\begin{align*}
g^{ab}\in W^{1,p}_{\rm loc},\qquad
\beta^a\in W^{0,p}_{\rm loc}, \qquad
\gamma\in W^{-1,p}_{\rm loc}.
\end{align*}	
In particular, the leading order coefficients have the regularity
of the metric, and there is a loss of one derivative as we descend from
one order of coefficient to the next.
This leads us to consider elliptic $d^{\rm th}$-order operators of the form
\begin{equation}\label{eq:op-type}
L = \sum_{|\alpha|\le d} a^\alpha \partial_\alpha
\end{equation}
where the top-most coefficients lie in a space $J^s$ with $s$
derivatives and more generally where each $a^\alpha \in J^{s-d+|\alpha|}$.
While this category of operator is not the most general possible
in geometric analysis \cite{stredder_natural_1975},
it is sufficiently broad to include many operators of interest,
including Hodge Laplacians,
the Lichnerowicz Laplacian \cite{besse_einstein_1987},
the vector Laplacian \cite{isenberg_constant_1995}
and the conformal Laplacian, so long as $s\ge 1$ and 
so long as the underlying metric
lies in a space sufficiently regular so as to ensure H\"older continuity.
It also includes the class of geometric operators satisfying 
the hypotheses of Assumption P
of \cite{allen_sobolevclass_2022}.

Given an operator of the form \eqref{eq:op-type} with
leading order coefficients in some space $J^s$ with $s$ derivatives, one wants to
find compatible spaces $K^{\sigma}$ with $\sigma$
derivatives such that that $L: K^{\sigma}\to K^{\sigma-d}$ 
and such that local elliptic regularity holds: roughly that
if $u$ is regular enough that $L$ can act on it, and 
if $L u \in K^{\sigma-d}$, then locally $u\in K^\sigma$ along
with associated estimates.
We establish this theory for elliptic operators of the form
\eqref{eq:op-type} where the topmost coefficients come
from a space $J^s$ of one of the following types:
\begin{itemize}
	\item a Bessel potential space $H^{s,p}$, in which case
	$K^\sigma$ is another Bessel potential space $H^{\sigma,q}$ (Section \ref{s:Hsp}),
	\item a Triebel-Lizorkin space $F^{s,p}_q$, in which case
	$K^\sigma$ is another Triebel-Lizorkin space $F^{\sigma,a}_b$ (Section \ref{s:Fsp}),
	\item a Sobolev-Slobodeckij space $W^{s,p}$, in which case
	$K^\sigma$ is another Sobolev-Slobodeckij space $W^{\sigma,q}$ (Section \ref{s:Wsp}),
	\item a Besov space $\Bv[s,p,q]$, in which case
	$K^\sigma$ is another Besov space $\Bv[\sigma,a,b]$ (Section \ref{s:Bsp}).
\end{itemize}
In all these cases, the space $J^s$ is restricted to be suitably regular so that
its elements are H\"older continuous, and we give a careful description of the parameters
determining the allowable spaces $K^\sigma$.

\subsection{Rescaling estimates}
Our general approach is the same for all the function spaces considered,
and in the specific case of Bessel potential spaces $H^{s,p}$ the core ingredients are:
\begin{enumerate}
	\item Multiplication properties for $H^{s,p}$ spaces, Theorem \ref{thm:mult}.
	\item Mapping properties: given an operator $L$ of the form \eqref{eq:op-type} with leading coefficients
in $H^{s,p}$, for which spaces $H^{\sigma,q}$ does $L$ map $H^{\sigma,q}\to H^{\sigma-d,q}$?
This is the content of Proposition \ref{prop:mapping-Hsp}.
\item A rescaling estimate, described below.
\item\label{step:zoom}  A coefficient freezing/blowup argument which uses the rescaling estimate to establish ``regularity at a point'', 
Proposition \ref{prop:elliptic-zoom}.
\item A partition of unity decomposition and bootstrap to obtain the main regularity result, Theorem \ref{thm:interior-reg}.
\end{enumerate}

Coefficient freezing as used in step \ref{step:zoom} above is classical, but we 
use a nonstandard rescaling technique to manage the perturbations from the constant coefficient operator.
To motivate this technique, consider the simplest case of integer-based Sobolev spaces on
the unit ball $B_1\subset \Reals^n$, and let $u\in W^{k,p}(B_1)$, where $k\in\Nats$, $1<p<\infty$.
For $0<r\le 1$ we
define $u_\rscr(x)=u(rx)$, so
$u_\rscr$ rescales $u$ up from the ball of radius $r$ to the unit ball.  
Derivatives are damped under this rescaling operation, but the singularities 
permitted by $L^p$ spaces are enhanced, and a computation using 
Sobolev embedding and H\"older's inequality shows
\begin{equation}\label{eq:rescale-est}
||u_\rscr||_{W^{k,p}(B_1)} \lesssim r^{\alpha} ||u||_{W^{k,p}(B_1)}
\end{equation}
where $\alpha = \min(k-\frac{n}{p},0)$, except in the marginal case $k=n/p$,
in which case we can take $\alpha$ to be any negative number. The cap $\alpha\le0$
appears in this estimate because of the constants, which are invariant under rescaling.
However, if $k>n/p$ so that elements of $W^{k,p}(B_1)$ are H\"older continuous, 
and if $f(0)=0$, one can do better.  Now estimate \eqref{eq:rescale-est} holds
with $\alpha=\min(k-n/p,1)$, except in the marginal case $k=n/p+1$, in which case 
we can use any $\alpha<1$.  Regardless, if $k>n/p$ and if $f(0)=0$, 
estimate \eqref{eq:rescale-est} holds for some $\alpha>0$.

Now consider a differential 
operator $L=\sum_{|\beta|\le d} a^\beta \partial_\beta$ 
with coefficients $a^\beta \in W^{k-d+|\beta|,p}(B_1)$ and with $k>n/p$.
The leading order coefficients lie in $W^{k,p}(B_1)$ and are therefore H\"older continuous.
Hence we can define the principal part of $L$ at $0$, 
\[
	L_0 = \sum_{|\beta|=d} a^\beta(0)\partial_\beta.
\]
If $u$ is a distribution that is regular enough that $L$ can act on it, 
a computation shows that 
\begin{align*}
	r^d (Lu)_{\rscr} &=  L_0 u_\rscr + \underbrace{\sum_{|\beta|=d} (a^\beta-a^\beta(0))_\rscr \partial_\beta u_\rscr}_{:=B_{[r]}u_\rscr} 
	+ \underbrace{\sum_{|\beta|<d} r^{d-|\beta|} a^\beta_{\rscr}  \partial_\beta u_\rscr}_{:=C_{[r]}u_\rscr}.
\end{align*}
The aim at this point is to show that 
by taking $r$ sufficiently small, 
the coefficients of the perturbations $B_{[r]}$ and $C_{[r]}$ can be made as small as desired 
so that a parametrix for $L_0$ can be employed to deduce regularity properties
of $u_\rscr$, and this is where the rescaling estimate \eqref{eq:rescale-est} is needed.  
Using the structural hypothesis $a_\beta\in W^{k-d+|\beta|,p}(B_1)$
along with the rescaling estimate \eqref{eq:rescale-est} we find, except in marginal cases
where an unimportant adjustment is needed,
that the coefficients of $C_{[r]}$ satisfy
\[
||r^{d-|\beta|}(a_\beta)_\rscr||_{W^{k-d+|\beta|,p}(B_1)} \lesssim 
r^{d-|\beta|} r^{\min(k-d+|\beta|-n/p,0)} ||a_\beta||_{W^{k-d+|\beta|,p}(B_1)}.
\]
Since $|\beta|< d$ for each of these coefficients, 
and since $k>n/p$, we obtain
\[
r^{d-|\beta|} r^{\min(k-d+\beta-n/p,0)} = r^{\min(k-n/p,d-|\beta|)} = r^\epsilon
\]
for some $\epsilon>0$. Hence the coefficients of $C_{[r]}$ scale away as $r\to 0$.
On the other hand, the high order perturbation 
coefficients $a^{\beta}-a^{\beta}(0)$ lie in $W^{k,p}(B_1)$ 
with $k>n/p$ and vanish at $0$, so the improved variation of the scaling estimate 
\eqref{eq:rescale-est} again shows that the coefficients of $B_{[r]}$ vanish as 
$r\to 0$.

Propositions \ref{prop:rescale-Fsp} and
\ref{prop:rescale-Bsp} show that estimate \eqref{eq:rescale-est}
generalizes to Triebel-Lizorkin and Besov spaces respectively.  
For Triebel-Lizorkin spaces, the proof
requires elementary techniques from Littlewood-Paley theory and paramultiplication, 
and the necessary background is recalled in Section \ref{secsec:TL-rescaling}
prior to the proof of Proposition \ref{prop:rescale-Fsp}.  The analogous results
for Besov spaces follow from the Triebel-Lizorkin result and interpolation.
As mentioned above, in the interest of
approachability, for Bessel potential spaces we use an alternative approach with
a rescaling estimate, Proposition \ref{prop:poor-mans-bp}, that is not sharp, but which admits 
an elementary proof that is independent of Littlewood-Paley theory.

\subsection{Multiplication}
Mapping properties of differential operators with coefficients in 
Sobolev-type spaces depend on pointwise multiplication rules
that determine when a product of factors from two given function
spaces lies in a third space.  There is an extensive literature on this subject,
including \cite{palais_foundations_1968} \cite{zolesio_multiplication_1977}
 \cite{amann_multiplication_1991} \cite{sickel_holder_1995}
 \cite{runst_sobolev_1996} \cite{johnsen_pointwise_1995} \cite{behzadan_multiplication_2021} that contains individual pieces
of the theory we require. Chapter 4 of  \cite{runst_sobolev_1996} is especially comprehensive. 
Where these works overlap,
there is generally agreement on the hypotheses, but certain edge cases are treated,
or not, by different authors.  Rather than attempt to assemble these disparate 
pieces into a coherent whole, we include a self-contained proof of 
multiplication rules
for Triebel-Lizorkin spaces, Theorem \ref{thm:mult-Fsp}, and for Besov spaces, 
Theorem \ref{thm:mult-Besov}, in the appendices.  Corresponding rules for 
Bessel potential spaces and Sobolev-Slobodeckij spaces, Theorems \ref{thm:mult}
and \ref{thm:mult-Wsp} respectively, follow as corollaries.  The
proofs rely on the same elementary Littlewood-Paley/paramultiplication
techniques that we use
to obtain the rescaling estimates of Section \ref{secsec:TL-rescaling}.
Although we have limited our analysis to the region 
$1<p,q<\infty$ for the spaces
$F^{s,p}_q$ and $\Bv[s,p,q]$, in this restricted setting
we obtain a consistent set of hypotheses over all ranges of $s$ 
that are simpler in character, and that are at least as sharp, as what appears
currently in the literature.

\section{Coefficients in Bessel Potential Spaces}\label{s:Hsp}
In this section we prove interior elliptic estimates
for operators having coefficients in Bessel potential
spaces, with a goal of presenting the result using a minimum of technology.
The primary background requirements are:
\begin{itemize}
	\item standard facts about Sobolev spaces with integer orders of differentiability,
	\item embedding, interpolation and duality theory for Bessel potential spaces, 
	\item multiplication rules for Bessel potential spaces, which we recall below, and
	\item elementary tools from harmonic analysis needed to construct parametrices for elliptic operators with constant coefficients.
\end{itemize}
In particular, the approach is otherwise independent of Littlewood-Paley theory or 
the general theory of pseudodifferential operators, beyond what is required to define
the spaces themselves.

Let $\mathcal F$ denote the Fourier transform and 
for $s\in \Reals$ let $D^s$ be the pseudodifferential
operator given by
\[
\mathcal F[ D^{s} u](\xi) = (1+|\xi|^2)^{s/2} \mathcal F[u](\xi).
\]
Given
$1<p<\infty$ and $s\in\Reals$ a tempered distribution 
$u$ on $\Reals^n$ belongs to the Bessel potential
space $H^{s,p}(\Reals^n)$ if $D^{-s} u\in L^p(\Reals^n)$,
in which case
\[
||u||_{H^{s,p}(\Reals^n)} = || D^{-s} u||_{L^p(\Reals^n)}.
\]
We use the same notation for for distributions taking on values 
in a real vector space 
(e.g., $\Reals^k$ or $\Reals^{k\times k}$).
When $k\in \Ints$, then $H^{k,p}(\Reals^n)$ with this definition
coincides with standard Sobolev spaces of distributions having
derivatives laying in Lebesgue spaces.

Given an open set $\Omega\subseteq\Reals^n$ the space $H^{s,p}(\Omega)$
consists of restrictions of distributions in $H^{s,p}(\Reals^n)$
to $\Omega$ and is given the quotient norm.  That is,
\[
||u||_{H^{s,p}(\Omega)} = \inf \{ ||\hat u||_{H^{s,p}(\Reals^n)}:\hat u\in H^{s,p}(\Reals^n), \hat u|_{\Omega} = u\}.
\]
We say an open set $\Omega$ is a \textbf{$C^\infty$ domain} if each point in the boundary
admits an open ball centered at it and a diffeomorphism from it to an open subset of $\Reals^n$ such that
the image of the intersection of $\Omega$ with the ball is a simply connected subset of the upper half space
$\Reals^{n,+}$.  If $\Omega$ is a bounded $C^\infty$ domain and if $k\in\Ints$, then $H^{k,p}(\Omega)$
coincides with the usual integer-based Sobolev spaces.

We have the following embedding, interpolation, and duality properties of Bessel
potential spaces, which are special cases of the same results cited in Section 
\ref{s:Fsp} for the more general Triebel-Lizorkin spaces.
\begin{proposition}\label{prop:embedding-Hsp}
Assume $1<p,p_1,p_2<\infty$
and $s,s_1,s_2\in \Reals$, and suppose $\Omega$ is a bounded open set in $\Reals^n$.
\begin{enumerate}
	\item
	If $s_1>s_2$ then $H^{s_1,p}(\Reals^n)\hookrightarrow H^{s_2,p}(\Reals^n)$ and 
	$H^{s_1,p}(\Omega)\hookrightarrow H^{s_2,p}(\Omega)$.
	\item If $p_1 \ge p_2$ then $H^{s,p_1}(\Omega)\hookrightarrow H^{s,p_2}(\Omega)$.
	\item
	If $s_1>s_2$ and
	$\frac{1}{p_1}-\frac{s_1}{n} = \frac{1}{p_2}-\frac{s_2}{n}$ then
	$H^{s_1,p_1}(\Reals^n)\hookrightarrow H^{s_2,p_2}(\Reals^n)$.
	\item\ If $s_1>s_2$ and 
	$\frac{1}{p_1}-\frac{s_1}{n}\le \frac{1}{p_2}-\frac{s_2}{n}$
	then $H^{s_1,p_1}(\Omega)\hookrightarrow H^{s_2,p_2}(\Omega)$.
	\item If $0<\alpha<1$ then $H^{\frac{n}p+\alpha,p}(\Reals^n)\hookrightarrow C^{0,\alpha}(\Reals^n)$
	and $H^{\frac{n}p+\alpha,p}(\Omega)\hookrightarrow C^{0,\alpha}(\Omega)$.
\end{enumerate}
\end{proposition}
In the final embedding above, $C^{0,\alpha}(\Reals^n)$ denotes the H\"older space with
norm $||u||_{C^{0,\alpha}(\Reals^n)}=||u||_{L^\infty(\Reals^n)}+ \sup_{x\neq y}\left|\frac{f(x)-f(y)}{|x-y|^\alpha}\right|$,
with an analogous norm for functions defined on $\Omega$.
\begin{proposition}\label{prop:interp-Hsp}
	Assume $1<p_1,p_2<\infty$
	and $s_1,s_2\in \Reals$, and suppose $\Omega$ is either $\Reals^n$ or is
	a bounded $C^\infty$ domain in $\Reals^n$.
	For $0<\theta<1$,
	\[
		[H^{s_1,p_1}(\Omega),H^{s_2,p_2}(\Omega)]_\theta = H^{s,p}(\Omega)
	\]
	where
	\[
		s=(1-\theta)s_1+\theta s_2,\quad \frac{1}{p} = (1-\theta)\frac{1}{p_1}+\theta \frac{1}{p_2}
	\]
\end{proposition}
\begin{proposition}\label{prop:dual-Hsp}
	Assume $1<p<\infty$ and $s\in \Reals$.
	The bilinear map $C^\infty(\Reals^n)\times C^\infty(\Reals^n)\to \Reals$
	given by $\left<f,g\right> := \int_\Omega fg$ extends to a continuous bilinear map
	$F^{s,p}_q(\Reals^n)\times F^{-s,p^*}_{q^*}(\Reals^n)\to \Reals$ 
	where $1/p^* = 1-1/p$ and $1/q^* =1-1/q$.  Moreover, 
	$f\mapsto \left<f,\cdot\right>$ is a continuous identification of $F^{s,p}_q(\Reals^n)$
	with $(F^{-s,p^*}_{q^*}(\Reals^n))^*$.
\end{proposition}

\subsection{Mapping properties}

The following definition encodes the regularity structure of the coefficients of differential
operators appearing frequently in geometric analysis.

\begin{defn}\label{def:L-Hsp}
	Consider a
	$d^{\rm th}$ order differential operator on an open set $\Omega\subseteq \Reals^n$,
	\[
	L = \sum_{|\alpha|\le d} a^\alpha \partial_\alpha
	\]
	where the coefficients are $\Reals^{k\times k}$-valued for a system of $k$ variables.
	We say that $L$ is of class $\mathcal L^d(H^{s,p};\Omega)$ 
	for some $s\in\Reals$ and $1<p<\infty$ if each
	\[
	a^\alpha \in H^{s+|\alpha|-d,p}(\Omega ).
	\]
	
	If $L$ omits terms of order lower than $d_0$ for some $0\le d_0 \le d$,
	i.e.,
	\[
	L = \sum_{d_0\le |\alpha|\le d } a^\alpha \partial_\alpha,
	\]
	we say $L\in \mathcal L_{d_0}^d(H^{s,p};\Omega)$.
	\end{defn}
	
	To motivate the roles of the pair of indices $d$ and $d_0$ in the previous definition,
	recall from the introduction the conformal Laplacian $\mathcal C_g = -c_n\Delta_g + R[g]$ 
	of a metric $g$ on a bounded $C^\infty$ domain $\Omega$. An elementary
	computation using Sobolev embedding shows
	that if $g\in H^{1,p}(\Omega)$ with $p>n$ then 
	$\mathcal C_g$ is of class $\mathcal L^2(H^{1,p};\Omega)$. The low-order term is the 
	scalar curvature $R[g]\in H^{-1,p}(\Omega)$ and its presence 
	restricts the set of spaces that $\mathcal C_g$ can act on: functions in these spaces
	must possess at least one derivative.  By contrast, the ordinary Laplacian for the same metric
	has no zero-order term and consequently is an operator of class $\mathcal L^2_1(H^{1,p};\Omega)$.
	Moreover, one can show that it acts on a broader class of spaces and, for example, 
	defines a map $L^p(\Omega)\to H^{-2,p}(\Omega)$.  Hence, in addition to the order $d$ of 
	the differential operator we also track the order $d_0$ of the term with the 
	lowest number of derivatives
	appearing in the operator. 	

Consider an operator $L$ of class $\mathcal L_{d_0}^d (H^{s,p};\Omega)$.  It
defines a map from $C^\infty(\overline \Omega)$, the set of smooth functions on $\Omega$
admitting smooth extensions to $\Reals^n$, to the set $D'(\Omega)$ of distributions on $\Omega$,
and we wish to establish finer-grained mapping properties.  Specifically,
we would like to determine the indices $(\sigma,q)\in \Reals\times (1,\infty)$ such that 
$L$ is continuous $H^{\sigma,q}(\Omega)\to H^{\sigma-d,q}(\Omega)$.

The following result on multiplication of Bessel potential spaces is the 
tool needed to establish these mapping properties. It can be readily proved for integral orders
of differentiability using only Sobolev embedding and duality arguments, and a slightly less sharp version
that would, in fact, be sufficient for our purposes can be proved with interpolation techniques (\cite{behzadan_multiplication_2021} Theorem 5.1).
See also \cite{palais_foundations_1968}, the original reference for multiplication of Bessel potential
spaces, which also considers the case of more than two factors.
The statement below is a special case of the
multiplication rules for the broader class of Triebel-Lizorkin spaces
proved in Appendix \ref{app:mult-Fsp}.
\begin{theorem}\label{thm:mult}
Let $\Omega$ be a bounded open subset of $\Reals^n$.
	Suppose $1<p_1,p_2,q<\infty$ and $s_1,s_2,\sigma\in\Reals$.  Let $r_1,r_2$ and $r$ be defined by
\[
\frac{1}{r_1} = \frac{1}{p_1} - \frac{s_1}{n},\qquad
\frac{1}{r_2} = \frac{1}{p_2} - \frac{s_2}{n},\quad\text{\rm and}\quad
\frac{1}{r} = \frac{1}{q} - \frac{\sigma}{n}.
\]
Pointwise multiplication of $C^\infty(\overline \Omega)$ 
functions extends to a continuous bilinear map 
$H^{s_1,p_1}(\Omega)\times H^{s_2,p_2}(\Omega)
\rightarrow H^{\sigma,q}(\Omega)$ so long as
\begin{align}
 s_1+s_2&\ge 0\label{eq:Hsp-mult-a}\\
 \min({s_1,s_2}) &\ge \sigma \label{eq:Hsp-mult-b}\\
\max\left(\frac{1}{r_1},\frac{1}{r_2}\right) & \le  \frac{1}{r}\label{eq:Hsp-mult-c}\\
\frac{1}{r_1} + \frac{1}{r_2} & \le 1\label{eq:Hsp-mult-d}\\
\frac{1}{r_1} + \frac{1}{r_2} & \le \frac{1}{r}\label{eq:Hsp-mult-lebesgue}
\end{align}
with inequality \eqref{eq:Hsp-mult-lebesgue} strict if $\min(1/r_1,1/r_2,1-1/r)=0$.
\end{theorem}
There are admittedly a large number of conditions in the previous result, but
they all have an easy interpretation. 
For a pair $(\sigma,q)\in \Reals\times(1,\infty)$ we call
\[
\frac{1}{r}=\frac{1}{q} -\frac{\sigma}{n}
\]
the \textbf{Lebesgue regularity} of the pair, the terminology
motivated by the observation that if $0<1/r<1$ then
Sobolev embedding implies $H^{\sigma,q}_{\rm loc}(\Reals^n)$ 
embeds in $L^r_{\rm loc}(\Reals^n)$. Using this vocabulary, 
the conditions \eqref{eq:Hsp-mult-a}--\eqref{eq:Hsp-mult-lebesgue} are, loosely:
\begin{itemize}
	\item[\eqref{eq:Hsp-mult-a}] If one factor has a negative number of derivatives, the remaining factor 
	must have at as many positive derivatives.
	\item[\eqref{eq:Hsp-mult-b}] The product will not have, in general, more derivatives than either factor.
	\item[\eqref{eq:Hsp-mult-c}] The product will not have, in general, better Lebesgue regularity than either factor.
	\item[\eqref{eq:Hsp-mult-d}] $L^1$ is a least-regular barrier for multiplication of $L^p$ spaces. Note that in light of inequality \eqref{eq:Hsp-mult-lebesgue} this condition only plays a role if $s<0$.
	\item[\eqref{eq:Hsp-mult-lebesgue}] The Lebesgue regularity of the product is consistent with multiplication of $L^p$ spaces. Strictness of this inequality is needed in edge cases because $L^\infty$ is not the right target space for borderline Sobolev embeddings.
\end{itemize}

Repeated applications of Theorem \ref{thm:mult} 
imply the following mapping result, where we emphasize the
new hypothesis $s>n/p$ which ensures, among other consequences, 
that the highest order coefficients of the differential operator are continuous.

\begin{proposition}\label{prop:mapping-Hsp} Let $\Omega$ be a 
	bounded open subset of $\Reals^n$.
Suppose $1<p,q<\infty$, $s>n/p$, $\sigma\in\Reals$ and 
$d,d_0\in\Ints_{\ge 0}$ with $d\ge d_0$.  
An operator of class $\mathcal L_{d_0}^d(H^{s,p};\Omega)$ 
extends from a map
$C^\infty(\overline \Omega)\mapsto \mathcal D'(\Omega)$ to 
a continuous linear map 
$H^{\sigma,q}(\Omega)\mapsto H^{\sigma-d,q}(\Omega)$
so long as
\begin{equation}\label{eq:S-conds}
\begin{aligned}
\sigma&\in [d-s,s+d_0]\\
\frac{1}{q} - \frac{\sigma}{n} &\in \left[
\frac{1}{p} - \frac{s+d_0}{n}, \frac{1}{p^*}-\frac{d-s}{n}
\right]
\end{aligned}
\end{equation}
where $1/p^* = 1-1/p$ is the conjugate Lebesgue exponent of $p$.
Moreover, the map $A$ between these spaces depends continuously
on its coefficients $a^\alpha \in H^{s-d+\alpha,p}(\Omega)$.
\end{proposition}
Conditions \eqref{eq:S-conds} on $\sigma$ and $q$
describe the natural Sobolev indices of spaces
for an operator of class
$\mathcal L_{d_0}^d(H^{s,p}; \Omega)$ to act on, 
and it is convenient to have notation for this set.
\begin{defn}\label{def:scriptS-Hsp}
Suppose $1<p<\infty$, $s\in\Reals$ and $d,d_0\in\Ints_{\ge 0}$
with $d\ge d_0$.
The \textbf{compatible Sobolev indices} 
for an operator of class $\mathcal L_{d_0}^d(H^{s,p};\Omega)$ acting
on a bounded open set $\Omega$ is the set
\[
\mathcal S_{d_0}^d(H^{s,p}) \subseteq \Reals\times (1,\infty)
\]
of tuples $(\sigma,q)\in \Reals\times(1,\infty)$ 
satisfying \eqref{eq:S-conds}.
\end{defn}
\begin{figure}
\begin{center}
\hskip -2.5cm\includegraphics{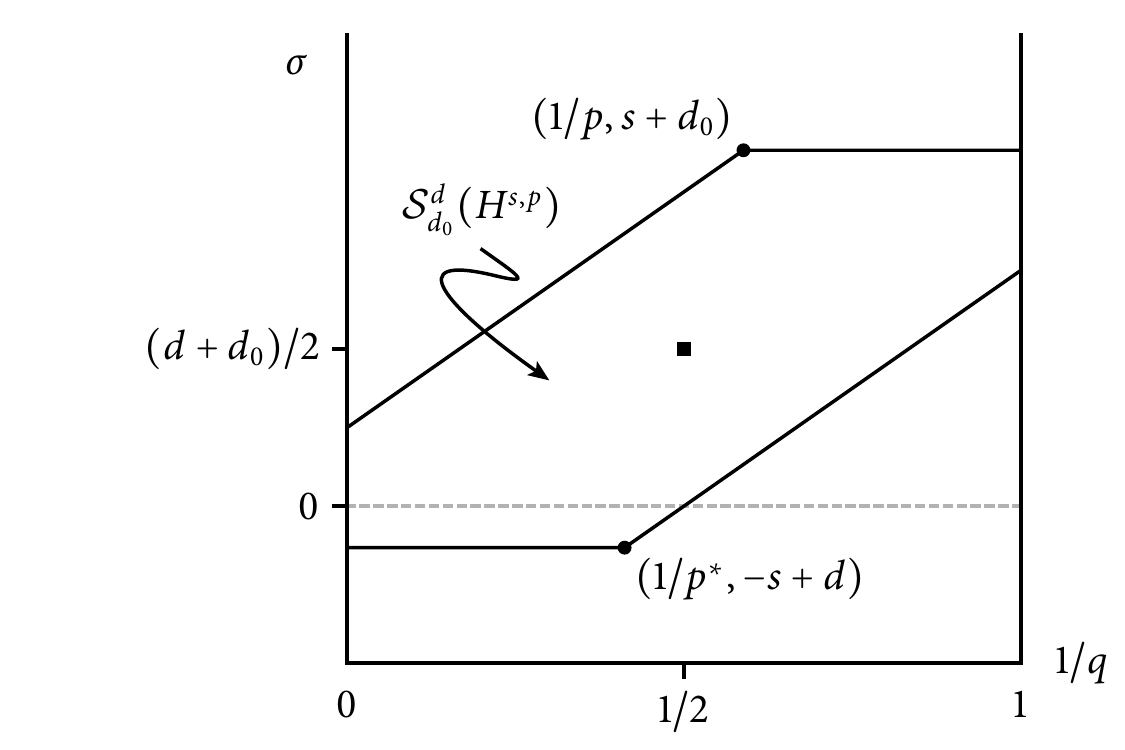}
\end{center}
\caption{
\label{fig:S}
The region $\mathcal S_{d_0}^d(H^{s,p})$ for $n=3$, $d=2$, $d_0=0$, $p=1.7$
and $s=\frac n p + \frac 12$.}
\end{figure}

The second condition of \eqref{eq:S-conds} is a restriction on the Lebesgue regularity
of $(\sigma,q)$ and it is helpful to visualize $\mathcal S^d_{d_0}(H^{s,p})$
after making the transformation $q\mapsto 1/q$
since sets of constant Lebesgue regularity
then appear as straight lines with slope $1/n$ as in Figure \ref{fig:S}.

For any particular collection of parameters, it may be
that $\mathcal S^{d}_{d_0}(H^{s,p})$ is empty.  The following result
establishes when the set is nonempty and hence when an operator of
class $\mathcal L_{d_0}^d(H^{s,p};\Omega)$ has a suitable collection of
Bessel potential spaces to act on.
\begin{lemma}\label{lem:S-members}
Suppose $1<p<\infty$, $s\in\Reals$ and $d,d_0\in \Ints_{\ge 0}$
with $d\ge d_0$.  Then $\mathcal S_{d_0}^d(H^{s,p})$ 
is nonempty if and only if
\begin{align}
\label{eq:S-p-1}
s &\ge (d-d_0)/2\text{, and}\\
\label{eq:S-p-2}
\frac{1}{p} - \frac{s}{n} &\le \frac{1}{2}-\frac{(d-d_0)/2}{n}.
\end{align}
If $S_{d_0}^d(H^{s,p})$ is non-empty then it contains $(s+d_0,p)$, $(d-s,p^*)$,
and $((d+d_0)/2,2)$.  Moreover, if $(\sigma,q)\in \mathcal S_{d_0}^d(H^{s,p})$, then
we have the continuous inclusions of Fr\'echet spaces
\begin{equation}\label{eq:S-include}
H^{s+d_0,p}_{\rm loc}(\Reals^n) \subseteq H^{\sigma,q}_{\rm loc}(\Reals^n)
\subseteq H^{d-s,p^*}_{\rm loc}(\Reals^n).
\end{equation}
\end{lemma}
\begin{proof}
The intervals in \eqref{eq:S-conds} defining $\mathcal S_{d_0}^d(H^{s,p})$
are non-empty if and only if $s\ge (d-d_0)/2$ and
\[
\frac{1}{p}-\frac{s+d_0}{n} \le 1 - \frac{1}{p}-\frac{d-s}{n},
\]
which is equivalent to inequality \eqref{eq:S-p-2}.  Moreover,
if these intervals are nonempty, then $(s+d_0,p)$ and $(d-s,p^*)$
evidently belong to $\mathcal S_{d_0}^d(H^{s,p})$. 
To show that $(\sigma,q)=((d+d_0)/2,2)\in \mathcal S_{d_0}^d(H^{s,p})$
whenever this set is nonempty, define
\[
\frac{1}{r_1} = \frac{1}{p}-\frac{s+d_0}{n},\qquad
\frac{1}{r_2} = \frac{1}{p^*}-\frac{d-s}{n},\quad
\text{and}\quad\frac{1}{r} = \frac{1}{2} -\frac{(d+d_0)/2}{n}.
\]
Then $((d+d_0)/2,2)\in  \mathcal S_{d_0}^d(H^{s,p})$ so long as
\begin{equation}\label{eq:s-compare}
d-s\le \frac{d+d_0}2\le s+d_0
\end{equation}
and
\begin{equation}\label{eq:r-compare}
\frac{1}{r_1} \le \frac{1}{r} \le \frac{1}{r_2}.
\end{equation}
Since $\mathcal S_{d_0}^d(H^{s,p})$ is nonempty,
$d-s\le s+d_0$ and $1/r_1\le 1/r_2$.  
Inequalities \eqref{eq:s-compare} and \eqref{eq:r-compare}
are then consequences of the observations
\[
\frac{d+d_0}{2}=\frac{1}{2}\left[ (d-s) + (s+d_0)\right]
\]
and
\[
\frac1r = \frac{1}{2}\left(\frac1r_1+\frac1r_2\right).
\]

Finally, recall
that $H^{s_1,p_1}_{\rm loc}(\Reals^n)$ embeds continuously
in  $H^{s_2,p_2}_{\rm loc}(\Reals^n)$ if $s_1\ge s_2$ and if
\[
\frac{1}{p_1}-\frac{s_1}{n} \le \frac{1}{p_2}-\frac{s_2}{n}.
\]
So conditions \eqref{eq:S-conds} ensure that
if $(\sigma,q)\in \mathcal S_{d_0}^d(H^{s,p})$ then we have
the continuous inclusions \eqref{eq:S-include}.
\end{proof}

The indices $(s+d_0,p)$ and $(d-s,p^*)$ correspond to the 
most regular and least regular spaces an operator in 
$\mathcal L_{d_0}^d(H^{s,p};\Omega)$ can naturally act on;
indices $(\sigma,q)\in \mathcal S_{d_0}^d(H^{s,p})$ yield spaces
$H^{\sigma,q}_{\rm loc}(\Reals^n)$ that lie intermediate between the
two extreme spaces $H^{s+d_0,p}_{\rm loc}(\Reals^n)$ and $H^{d-s,p^*}_{\rm loc}(\Reals^n)$.
Indeed, one interpretation of Lemma \ref{lem:S-members} is that 
$\mathcal S^d_{d_0}(H^{s,p})$ is nonempty exactly when
$H^{s+d_0,p}_{\rm loc}(\Reals^n)$ includes $H^{d-s,p^*}_{\rm loc}(\Reals^n)$.
Alternatively, inequalities \eqref{eq:S-p-1}--\eqref{eq:S-p-2} are exactly
the condition that $H^{s,p}_{\rm loc}$ includes
$H^{(d+d_0)/2,2}_{\rm loc}$, which highlights the importance
of this $L^2$-based space.  In effect, Lemma \ref{lem:S-members}
implies $\mathcal S_{d_0}^d(H^{s,p})$
is nonempty exactly when it contains $((d+d_0)/2,2)$.   
For example, consider the case of a
general second-order operator, so $d=2$ and $d_0=0$.  Then
$\mathcal S_0^2(H^{s,p})$ is non-empty
when $H^{s,p}_{\rm loc}(\Reals^n)$ contains $H^{1,2}_{\rm loc}(\Reals^n)$,
the natural $L^2$-based space for a weak existence theory.  See also Figure 
\ref{fig:S}, where the key $L^2$-based space appears as a small square.

In addition to the mapping properties of $L$ described in Proposition
\ref{prop:mapping-Hsp} we require an analogous 
result for the commutator of $L$ with a smooth cutoff function $\phi$.
\begin{lemma} \label{lem:commutator-Hsp}
Let $\Omega$ be a bounded open subset of $\Reals^n$.
	Suppose $1<p,q<\infty$, $s>n/p$, $\sigma\in\Reals$ and 
$d,d_0\in\Ints_{\ge 0}$ with $d\ge d_0$. Let
$L$ be an operator of class $\mathcal L_{d_0}^d(H^{s,p};\Omega)$
and let $\phi\in\mathcal D(\Omega)$.  
Then $[L,\phi]$ 
extends from a map 
$C^\infty(\overline \Omega)\mapsto \mathcal D'(\Omega)$ to 
a continuous linear map 
$H^{\sigma,q}(\Omega)\mapsto H^{\sigma-d+1,q}(\Omega)$
so long as  $(\sigma + 1,q) \in \mathcal S^{d}_{d_0}(H^{s,p})$. Moreover,
if $d_0=0$, the same result holds if $(\sigma,q)\in \mathcal S^{d}_{0}(H^{s,p})$.
\end{lemma}
\begin{proof}
Suppose $u\in H^{\sigma,q}(\Omega)$.  A term of $[L,\phi]u$
has the form
\[
a^\alpha \partial_\beta \phi \partial_\gamma u	
\]
where $a^\alpha\in H^{s-d+|\alpha|,p}$ and where
$\max(1,d_0)\le |\alpha|\le d$ and $|\gamma|\le |\alpha|-1$.
The result in the case of general $d_0$ follows from a direct computation using
these facts along with Theorem \ref{thm:mult}.  

If $d_0=0$, we can improve the result as follows.  Let $\hat L$ be $L$ with its zero-order term
removed, so $\hat L\in \mathcal L_{1}^{d}(H^{s,p};\Omega)$.  Then $[L,\phi]=[\hat L,\phi]$ and using
the result just proved we find that 
the commutator $[L,\phi]$ maps $H^{\sigma,q}(\Omega) \to H^{\sigma-d+1,a}(\Omega)$ so long as
$(\sigma+1,q)\in \mathcal S^{d}_{1}(H^{s,p})$.  But a routine computation shows that this condition is
equivalent to $(\sigma,q) \in \mathcal S^{d-1}_{0}(H^{s,p})$ and the claimed improvement follows
since $S^{d-1}_{0}(H^{s,p})\supset S^{d}_{0}(H^{s,p})$.
\end{proof}

\subsection{Rescaling estimates}\label{secsec:rescale-Hsp}

For a Schwartz function $u$ let
$u_\rscr(x) = u(rx)$.
This rescaling operation extends to general tempered distributions
by continuity and duality arguments, and we use the same notation
when $u$ is a distribution.  When $m$ is a non-negative integer,
it is easy to see that $u\mapsto u_\rscr$ is a 
continuous automorphism of $H^{m,p}(\Reals^n)$ for all $1<p<\infty$.
The same holds for $H^{s,p}(\Reals^n)$ for $s>0$ by interpolation,
and for $s<0$ by an elementary duality argument. 

In this section, we prove two principal estimates for the norms of rescaling
operators, with bounds depending on $r$ along with the parameters of the 
function space being acted on.
First, we have 
Proposition \ref{prop:rescale-int}, 
which establishes
the desired estimates for functions with integer-order differentiability.
Then, using interpolation, we generalize the estimates to non-integer
differentiability at the penalty of a mild loss of sharpness; this
is the content of Proposition \ref{prop:poor-mans-bp}, which is
the main result of this section.  In fact, one can recover the sharpness
using more sophisticated tools from Littlewood-Paley theory, and indeed
we do for the broader class of Triebel-Lizorkin spaces in Section 
\ref{secsec:TL-rescaling}.  Nevertheless, the less-than optimal estimates
of Proposition \ref{prop:poor-mans-bp} are sufficient to establish
local elliptic regularity in the context of Bessel potential function spaces, 
and it permits a proof using only a minimal set of tools.
\begin{proposition}\label{prop:rescale-int}
Suppose $1<p<\infty$, $m\in\Ints$ 
and that $\chi$ is a Schwartz function on $\Reals^n$.
There exists a constant $\alpha\in\Reals$ such that
for all $0<r\le 1$ and all $u\in H^{m,p}(\Reals^n)$
\begin{equation}\label{eq:Hsp-scale}
||\chi u_\rscr||_{H^{m,p}(\Reals^n)} \lesssim r^{\alpha}||u||_{H^{m,p}(\Reals^n)}.
\end{equation}
Specifically:
\begin{enumerate}
	\item\label{part:rescale-int-generic}
Inequality \eqref{eq:Hsp-scale} holds with
\[
\alpha = \min\left(m-\frac{n}{p},0\right)
\]
unless $m-n/p=0$, in which case it holds for any choice
of $\alpha<0$, with implicit constant depending on 
$\alpha$.
	\item\label{part:rescale-int-zero-center}
	If $m>n/p$ (in which case functions in $H^{m,p}(\Reals^n)$
	are H\"older continuous) and if $u(0)=0$, then 
	inequality \eqref{eq:Hsp-scale} holds with
	\[
\alpha = \min\left(m-\frac{n}{p},1\right)
	\]
	unless $m-n/p=1$, in which case it holds for any choice of
	$\alpha<1$, with implicit constant depending on $\alpha$.
\end{enumerate}
\end{proposition}

Proposition \ref{prop:rescale-int} is established in the following sequence
of elementary results which treat specific subcases and supporting lemmas.
Specifically, it is an immediate consequence of Corollary \ref{cor:scale-sigma-neg},
Lemma \ref{lem:int-case-general} and Corollary \ref{cor:holder-rescale} below.
In applications, $\chi$ in Proposition \ref{prop:rescale-int} will be
compactly supported and we are effectively interested in rescaling $u$ from 
a ball of radius $r<1$ up to a ball of fixed radius.  Derivatives are dampened
under this operation, and the role of $\chi$ is to control the 
zero-frequency terms; without the cutoff function, for all
$m>0$ the optimal scaling would be $\alpha=-n/p$ instead. 

We begin in the easier setting $m\le 0$, in which case it turns out that the
the cutoff function plays no role. When $m=0$, we have 
the following consequence of the change of variables formula
that the $L^p$ norm has straightforward scaling 
behavior for all $r>0$.
\begin{lemma}\label{lem:Lpscale}
Suppose $1<p<\infty$, $r>0$, and $u\in L^p(\Reals^n)$.
Then
\[
||u_\rscr||_{L^p(\Reals^n)} = ||u||_{L^p(\Reals^n)} r^{-n/p}.
\]
\end{lemma}

Turning to the case $m<0$ in Proposition \ref{prop:rescale-int},
the proof proceeds via a duality argument, for
which we need the following result
concerning rescaling $H^{k,p}(\Reals^n)$ for $k>0$ with $r>1$ (rather than $r\le 1$).  

\begin{lemma}\label{lem:scale-in}
Suppose $1<p<\infty$ and $k\in\Ints_{\ge 0}$.  
For all $r\ge 1$ and all $u\in H^{k,p}(\Reals^n)$,
\begin{equation}
||u_\rscr||_{H^{k,p}(\Reals^n)} \lesssim ||u||_{H^{k,p}(\Reals^n)} r^{m -\frac{n}{p}}.
\end{equation}
\end{lemma}
\begin{proof}
Lemma \ref{lem:Lpscale}, the Gagliardo-Nirenberg-Sobolev inequality 
and the fact that $r\ge 1$ imply
\[
\begin{aligned}
||u_\rscr||_{H^{k,p}(\Reals^n)} 
&\lesssim ||u_\rscr||_{L^p(\Reals^n)} + \sum_{|\alpha|=k} ||\nabla^\alpha u_\rscr||_{L^p(\Reals^n)}\\
&= r^{-\frac{n}{p}} ||u||_{L^p(\Reals^n)} + r^{k} 
\sum_{|\alpha|=k} ||(\nabla^\alpha u)_\rscr||_{L^p(\Reals^n)}\\
&=  r^{-\frac{n}{p}}  ||u||_{L^p(\Reals^n)} +  r^{k-\frac{n}{p}} 
\sum_{|\alpha|=k} ||(\nabla^\alpha u)||_{L^p(\Reals^n)})\\
&\lesssim r^{k-\frac{n}{p}} ||u||_{H^{k,p}(\Reals^n)}.
\end{aligned}
\]
\end{proof}
The following corollary follows from duality from Lemma \ref{lem:scale-in}.
\begin{corollary} 
\label{cor:scale-sigma-neg}
Suppose $1<p<\infty$ and $m\in\Ints_{\le 0}$. 
For all $0<r\le 1$ and all  $u\in H^{m,p}(\Reals^n)$
\[
||u_{\{r\}}||_{H^{m,p}(\Reals^n)} \lesssim  r^{m-\frac{n}{p}}||u||_{H^{m,p}(\Reals^n)}.
\]
\end{corollary}
\begin{proof}
Using a density argument it is enough to establish the 
result when $u$ is smooth and compactly supported.
For any test function $\phi$,
\begin{equation}
\begin{aligned}
|\ip<u_\rscr,\phi>| &= \left|\int_{\Reals^n} u(rx) \phi(x) \;dx\right|\\ 
&= r^{-n} \left|\int_{\Reals^n} u(y) \phi(y/r)\; dy\right| \\
&\le r^{-n} ||u||_{H^{m,p}(\Reals^n)} ||\phi_{\rsc{1/r}}||_{H^{-m,p^*}(\Reals^n)}.
\end{aligned}
\end{equation}
Since $-m\ge 0$ and since $1/r\ge 1$, Lemma \ref{lem:scale-in} implies 
\[
||\phi_{\rsc{1/r}}||_{H^{-m,p^*}(\Reals^n)} \lesssim r^{m +\frac{n}{p^*}}||\phi||_{H^{-m,p^*}(\Reals^n)} 
\]
where the implicit constant is independent of $r$ and $u$. Hence
\[
|\ip<u_\rscr,\phi>| \lesssim ||u||_{H^{m,p}(\Reals^n)} ||\phi||_{H^{-m,p^*}(\Reals^n)} 
r^{-n +m +\frac{n}{p^*}} = ||u||_{H^{m,p}(\Reals^n)} ||\phi||_{H^{-m ,p^*}(\Reals^n)}
r^{m -\frac{n}{p}},
\]
which concludes the proof, noting that $H^{m,p}(\Reals^n)$ 
is the dual space of $H^{-m,p^*}(\Reals^n)$.
\end{proof}

Corollary \ref{cor:scale-sigma-neg} implies 
Proposition \ref{prop:rescale-int}
in the case $m\le 0$, for if $\chi$ is a Schwartz function
$||\chi u||_{H^{m,p}(\Reals^n)} \lesssim ||u||_{H^{m,p}(\Reals^n)}.
$
We now turn to the more involved case of Proposition \ref{prop:rescale-int},
$m>0$.

\begin{lemma}\label{lem:int-case-general} 
Suppose $1<p<\infty$, $m\in\Nats$ and that $\chi$ is a Schwartz function on $\Reals^n$.
There exists $\alpha\in\Reals$ such that
for all $1<r\le1$ and all  $u\in H^{m,p}(\Reals^n)$
\begin{equation}\label{eq:rescale-m-pos}
||u_\rscr||_{H^{m,p}(\Reals^n)} \lesssim  r^{\alpha}||u||_{H^{m,p}(\Reals^n)}.
\end{equation}
Specifically, we can take $\alpha=\min(0,m-n/p)$ 
in inequality \eqref{eq:rescale-m-pos} unless $m-n/p=0$,
in which case inequality \eqref{eq:rescale-m-pos} holds for any 
choice of $\alpha<0$
and the implicit constant depends on $\alpha$.
\end{lemma}
\begin{proof}
Repeated applications of the 
Gagliardo-Nirenberg-Sobolev inequality imply
\begin{equation}\label{eq:hard-case-base}
||\chi u_\rscr||_{H^{m,p}(\Reals^n)}
\lesssim ||\chi u_\rscr||_{L^{p}(\Reals^n)} + 
||\nabla^m u_\rscr||_{L^p(\Reals^n)}
\end{equation}
where the implicit constant depends on $\chi$; since $\chi$
is fixed the explicit dependence is unimportant.

The second term on the right-hand side of equation 
\eqref{eq:hard-case-base} is easy to estimate. 
Using the identity 
$\nabla^m u_\rscr = r^m (\nabla^m u)_\rscr$ Lemma \ref{lem:Lpscale}
we find
\[
||\nabla^m u_\rscr||_{L^p(\Reals^n)} \lesssim r^{m-n/p}||\nabla^m u||_{L^p(\Reals^m)} \le 
r^{m-n/p} ||u||_{H^{m,p}(\Reals^n)}.
\]

Turning to the low order term, first consider the case $m>n/p$.
Sobolev imbedding implies $u\in L^\infty(\Reals^n)$ and 
\[
\begin{aligned}
||\chi u_\rscr||_{L^p(\Reals^n)} \le
||\chi ||_{L^p(\Reals^n)} ||u_\rscr||_{L^\infty(\Reals^n)} 
&= ||\chi ||_{L^p(\Reals^n)} ||u||_{L^\infty(\Reals^n)}\\
&\lesssim  ||u||_{H^{m,p}(\Reals^n)}
= r^{\min(0,m-n/p)}||u||_{H^{m,p}(\Reals^n)}
\end{aligned}
\]
Now suppose $m<n/p$.  Then Sobolev embedding implies
$u\in L^q(\Reals^n)$ where
\[
\frac{1}{q}= \frac{1}{p} - \frac{m}{n}.
\]
H\"older's inequality, the fact that $\chi$ lies in every Lebesgue space,
and Lemma \ref{lem:Lpscale} imply
\[
\begin{aligned}
||\chi u_\rscr||_{L^p(\Reals^n))} \lesssim
||u_\rscr||_{L^q(\Reals^n)} &\le r^{-n/q} ||u||_{L^q(\Reals^n)}\\
&= r^{m-n/p} ||u||_{L^q(\Reals^n)} \lesssim 
r^{m-n/p}||u||_{H^{m,p}(\Reals^n)}
=r^{\min(m-n/p,0)}||u||_{H^{m,p}(\Reals^n)}.
\end{aligned}
\]
When
 $m=n/p$ we use the marginal case of Sobolev embedding and an argument
similar to the above to conclude
\[
||\chi u_\rscr||_{L^p(\Reals^n)} \lesssim r^{-n/q} ||u||_{H^{m,p}(\Reals^n)}
\]
for any $q\ge p$. Taking $q$ sufficiently large we find
inequality \eqref{eq:rescale-m-pos} holds for any choice of $\alpha<0$.
\end{proof}

Lemma \ref{lem:int-case-general} completes the proof
of part \eqref{part:rescale-int-generic} of Proposition \ref{prop:rescale-int}, and the following two results establish
part \eqref{part:rescale-int-zero-center}.
\begin{lemma}\label{lem:Lppart} 
Let $\chi$ be a Schwartz function on $\Reals^n$
and let $\alpha\in [0,1]$.
For all $0<r\le 1$
and all $u\in C^{0,\alpha}(\Reals^n)$ with $u(0)=0$,
\[
||\chi u_\rscr||_{L^p(\Reals^n)} \lesssim ||u||_{C^{0,\alpha}(\Reals^n)} r^\alpha.
\]
\end{lemma}
\begin{proof}
We divide $\Reals^n$ into three regions: the ball $B_1(0)$, the annulus 
$A=B_{1/r}(0)\setminus B_1(0)$ and the exterior region $E=\Reals^n\setminus B_{1/r}(0)$.
On the unit ball, since $u(0)=0$, 
\[
||\chi u_\rscr||_{L^p(B_1(0))} \lesssim ||\chi||_{L^\infty(\Reals^n)}||u||_{C^{0,\alpha}(B_1(0))} r^{\alpha}.
\]
To obtain the remainder of the estimate, pick a constant $B>0$ 
such that $|\chi(x)|\le B |x|^{-(n+1)/p-\alpha}$
for all $x\in\Reals^n$ with $|x|\ge 1$. Then, letting
$C_n$ be the volume of the unit $n-1$ sphere, we find on the annulus $A$,
\begin{equation}
\begin{aligned}
\int_{A} \chi^p u_\rscr^p &\le C_n \int_1^{1/r} B^p s^{-n-1-\alpha p} ||u||_{C^{0,\alpha}(\Reals^n)}^p (rs)^{\alpha p} s^{n-1}\; ds \\
&\le C_n B^p ||u||_{C^{0,\alpha}(\Reals^n)}^p r^{p\alpha} \int_{1}^{1/r} s^{-2}\; ds \le 
C_n B^p||u||_{C^{0,\alpha}(\Reals^n)}^p r^{p\alpha}.
\end{aligned}
\end{equation}
Taking $p^{\rm th}$
roots establishes the desired estimate on $A$.

Finally, for the exterior region we have
\[
\int_{E} |\chi u_\rscr|^p \le C_n B^p ||u||_{C_{0,\alpha}(\Reals^n)}^p 
\int_{1/r}^\infty s^{-n-1-\alpha p }s^{n-1} \; ds = 
C_n B^p ||u||_{C^{0,\alpha}(\Reals^n)}^p \frac{1}{1+\alpha p} r^{1 +\alpha p}
\]
which completes the proof.
\end{proof}

\begin{corollary}\label{cor:holder-rescale}
Let $\chi$ be a Schwartz function on $\Reals^n$. Suppose $1<p<\infty$ and $m\in\Nats$ with $m>n/p$.  
There exists $\alpha\in\Reals$ such that
for all $0<r\le 1$ and all  $u\in H^{m,p}(\Reals^n)$
with $u(0)=0$
\begin{equation}\label{eq:rescale-m-pos-zeromid}
||\chi u_\rscr||_{H^{m,p}(\Reals^n)} \lesssim  r^{\alpha}||u||_{H^{m,p}(\Reals^n)}.
\end{equation}
Specifically, we can take $\alpha=\min(1,m-n/p)$ 
in inequality \eqref{eq:rescale-m-pos-zeromid} unless $m-n/p=1$,
in which case inequality \eqref{eq:rescale-m-pos-zeromid} holds for any 
choice of $\alpha<1$
and the implicit constant depends on $\alpha$.
\end{corollary}
\begin{proof}
Following the argument of the beginning of the proof of Lemma \ref{lem:int-case-general}
we know that
\[
||\chi u_\rscr||_{H^{m,p}(\Reals^n)}
\lesssim ||\chi u_\rscr||_{L^{p}(\Reals^n)} + 
||\nabla^m u_\rscr||_{L^p(\Reals^n)}
\]
and that $||\nabla^m u_\rscr||_{L^p(\Reals^n)}\lesssim r^{m-n/p}||u||_{H^{s,p}(\Reals^n)}$.
Hence it suffices to show
\[
||\chi u_\rscr||_{L^p(\Reals^n)} \lesssim r^{\min(1,m-n/p)} ||u||_{H^{m,p}(\Reals^n)} .
\]

Suppose first that $0<m-n/p<1$.  Then $u\in C^{0,\alpha}(\Reals^n)$ with
$\alpha = m-n/p$.
Since $u(0)=0$, Lemma \ref{lem:Lppart} implies
\[
||\chi u_\rscr||_{L^p(\Reals^n)} \lesssim ||u||_{C^{0,\alpha}(\Reals^n)} r^\alpha
\lesssim ||u||_{H^{m,p}(\Reals^n)} r^{m-n/p}
= ||u||_{H^{m,p}(\Reals^n)} r^{\min(1,m-n/p)}.
\]
On the other hand, if $m-n/p>1$ then $u$ lies in $C^{0,1}(\Reals^n)$
and the same argument shows
\[
||\chi u_\rscr||_{L^p(\Reals^n)} \lesssim ||u||_{H^{m,p}(\Reals^n)} r^{1}
=||u||_{H^{m,p}(\Reals^n)} r^{\min(1,m-n/p)}
\]

The result in the marginal case $m-n/p=1$ follows from
a similar argument using the fact that $H^{m,p}(\Reals^n)$
embeds into $C^{0,\alpha}(\Reals^n)$ for any $\alpha\in(0,1)$.
\end{proof}

\begin{figure}
	\begin{center}
	\includegraphics{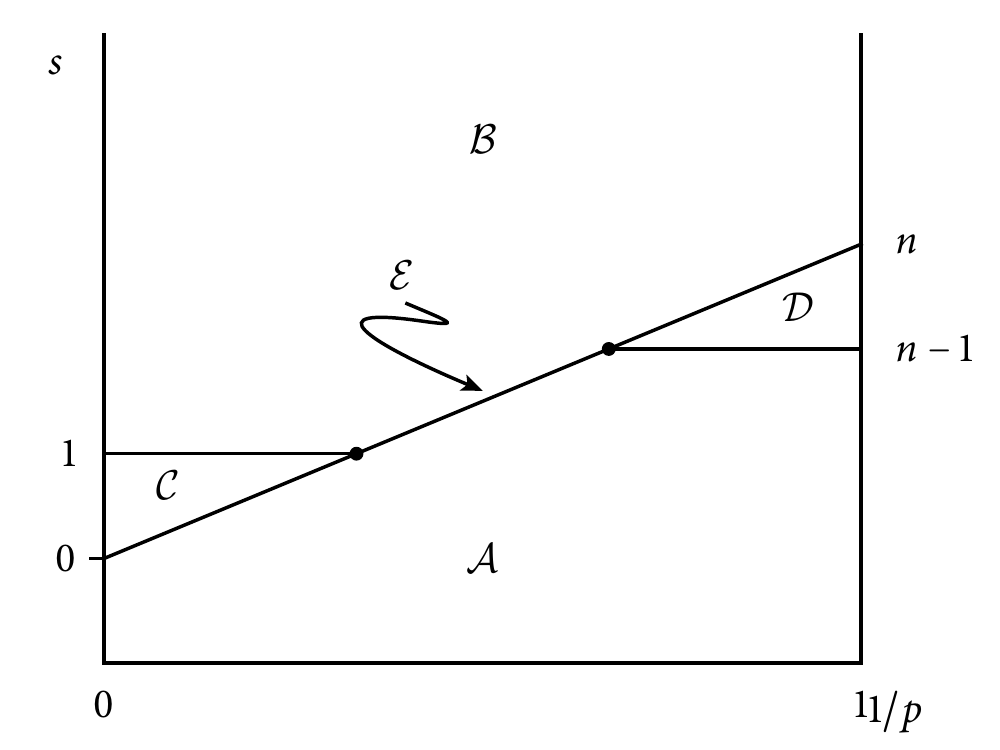}
	\end{center}
\caption{
\label{fig:interp}
Interpolation regions in Proposition \ref{prop:poor-mans-bp}.}
\end{figure}

Having now established Proposition \ref{prop:rescale-int}, we 
turn to its generalization to non-integer orders of differentiability.
In fact, one can show that the statement of
Proposition \ref{prop:rescale-int} generalizes without change, other 
than replacing $m\in\Ints$ with $s\in\Reals$; see Proposition
\ref{prop:rescale-Fsp} which establishes an extension of 
this fact to the broader
class of Triebel-Lizorkin spaces.  The following result is
mildly weaker, but is easier to prove and is sufficient for
our application establishing local elliptic regularity.  The
key difference is that
the equal sign in the definition of the exponent $\alpha$ 
in parts \eqref{part:rescale-int-generic}
and \eqref{part:rescale-int-zero-center}
 of Proposition \ref{prop:rescale-int}
 has been replaced with
a strict inequality.

\begin{proposition}\label{prop:poor-mans-bp}
Suppose $1<p<\infty$, $s\in\Reals$ and $\chi$ is a Schwartz
function on $\Reals^n$.  There exists
$\alpha\in \Reals$ such that for all $0<r\le 1$
and all $u\in H^{s,p}(\Reals^n)$,
\begin{equation}\label{eq:Hsp-scale-poor-bp}
||\chi u_{\{r\}}||_{H^{s,p}(\Reals^n)} \lesssim r^{\alpha}||u||_{H^{s,p}(\Reals^n)}.
\end{equation}
In particular:
\begin{enumerate}
	\item\label{part:rescale-poor-bp-generic}
Inequality \eqref{eq:Hsp-scale-poor-bp} holds with any
\[
\alpha < \min\left(s-\frac{n}{p},0\right)
\]
with implicit constant depending on  $\alpha$.
	\item\label{part:rescale-poor-bp-zero-center}
If $s>n/p$ and $u(0)=0$ then 
inequality \eqref{eq:Hsp-scale-poor-bp} holds with any
\[
\alpha < \min\left(s-\frac{n}{p},1\right)
\]
with implicit constant depending on  $\alpha$.
\end{enumerate}
\end{proposition}
\begin{proof}
We divide the tuples $(s,1/p)$ in $\Reals\times(0,1)$
into the following regions (see Figure \ref{fig:interp}):
\begin{enumerate}
	\item[$\mathcal A$:] $s\le n-1$, $0<1/p<1$, $s-n/p<0$,
	\item[$\mathcal B$:] $s\ge 1$, $0<1/p<1$, $s-n/p>0$,
	\item[$\mathcal C$:] $s < 1$, $s-n/p\ge 0$,
	\item[$\mathcal D$:] $s > n-1$, $s-n/p\le 0$,
	\item[$\mathcal E$:] $1\le s\le n-1$, $s-n/p=0$.
\end{enumerate}

Suppose $(s,1/p)\in\mathcal A$ and for the moment assume $s\ge 0$.  
If $s$ is an integer,
the result follows from Proposition \ref{prop:rescale-int}
so we can assume $0 <s < n-1$.  There exist
$1/p_0,1/p_1\in(0,1)$  such that $(s,1/p)$
lies on the line joining $(0,1/p_0)$ and $(n-1,1/p_1)$.  That is,
\[
(s,1/p) = (1-\theta) (0,1/p_0) + \theta (n-1,1/p_1) 
\]
for some $\theta\in(0,1)$.  Since 
$H^{s,p}(\Reals^n) = [H^{0,p_0}(\Reals^n),H^{n-1,p_1}(\Reals^n)]_\theta$ we conclude from Proposition \ref{prop:rescale-int} and interpolation applied to the map
$u\mapsto \chi u_\rscr$ that
\[
||\chi u_\rscr||_{H^{s,p}(\Reals^n)} \lesssim 
(r^{-n/p_0})^{1-\theta} (r^{n-1-n/p_1})^\theta||u||_{H^{s,p}(\Reals^n)}
= r^{s-n/p}||u||_{H^{s,p}(\Reals^n)}.
\]
Replacing Proposition \ref{prop:rescale-int} with Corollary \ref{cor:scale-sigma-neg},
the same technique works in region $\mathcal A$ if $s<0$ and indeed
the argument is simpler since we can select $p_0=p_1=p$.

If $(s,1/p)\in \mathcal B$ and $s$ is an integer then 
Proposition \ref{prop:rescale-int} implies 
inequality \eqref{eq:Hsp-scale} holds with $\alpha =0$.
When $s$ is not an integer we can interpolate between
$(\lfloor s\rfloor,1/p_1)$ and $(\lceil s\rceil,1/p_2)$ for appropriate
choices of $p_1$ and $p_2$ to obtain the same result.

Next, suppose $(s,1/p)$ lies in the triangular region $\mathcal C$,
so $n/p<s<1$.  Consider any $\sigma$ with $0<\sigma<n/p$, so $\sigma<s<1$
as well.
Then 
\[
H^{s,p}(\Reals^n) = [H^{\sigma,p}(\Reals^n),H^{1,p}(\Reals^n)]_\theta
\]
with $\theta=(s-\sigma)/(1-\sigma)\in (0,1)$.
From interpolation we find
\[
||\chi u_\rscr||_{H^{s,p}(\Reals^n)} \lesssim (r^{\sigma-n/p})^{1-\theta} ||u||_{H^{s,p}(\Reals^n)}.
\]
Noting that $1-\theta\to (1-s)/(1-n/p)<\infty$ as $\sigma\to n/p$ we
can take $\sigma$ as close to $n/p$ from below as we please to conclude
that inequality \eqref{eq:Hsp-scale-poor-bp} 
holds with any fixed choice of $\alpha<0$. Note that this interpolation,
and the one to follow for region $\mathcal D$, is the source of
the loss of sharpness of the current proposition.

In the region $\mathcal D$ the argument is similar to the argument
for region $\mathcal C$; we now interpolate between a point with $s=n-1$
and a point with $s>n/p$ in region $\mathcal A$ taken arbitrarily close
to the line $s=n/p$. On the line segment
$\mathcal E$ the proof follows by interpolating between a point in region $\mathcal A$ and a point region $\mathcal B$ taken arbitrarily
close to the line segment. This completes the proof of part \eqref{part:rescale-poor-bp-generic}.

The proof of part \eqref{part:rescale-poor-bp-zero-center} is proved in a completely analogous way with the 
main division now occurring on the line $s=n/p+1$. Given the careful proof of part \eqref{part:rescale-poor-bp-generic} we omit the details.
\end{proof}

\subsection{Interior elliptic estimates}

This section contains our principal elliptic regularity results, which
are established in two steps.  First, 
Proposition \ref{prop:elliptic-zoom} shows that if an operator is elliptic
at a single point, then elliptic regularity can be established for functions that are supported in a sufficiently
small neighborhood near the point.  The rescaling estimates of the previous section, along with
a parametrix construction, are the key tools
needed at this first stage. Theorem \ref{thm:interior-reg-Fsp} 
then builds on Proposition \ref{prop:elliptic-zoom}
to obtain full interior regularity for elliptic operators using a partition of unity argument 
along with a bootstrap. The commutator estimates of Lemma \ref{lem:commutator-Hsp} are the key
technical used at this second stage.

\begin{defn}\label{def:L-elliptic-Hsp}
Let $\Omega$ be an open subset of $\Reals^n$. Suppose $1<p<\infty$ and $s>n/p$.  An operator
\[
\sum_{|\alpha|\le d} a^\alpha \partial_\alpha
\]
of class $\mathcal L_{d_0}^d(H^{s,p};\Omega)$
is \textbf{elliptic} at $x_0\in\Omega$ if for every $\xi\in \Reals^n\setminus\{0\}$
\[
\sum_{|\alpha|=d} a^\alpha(x_0) \xi^\alpha \in \Reals^{k\times k}
\]
is non-singular, where $\xi^{\alpha}=\xi^{\alpha_1}\cdots\xi^{\alpha_d}$.
\end{defn}

We have the following standard parametrix construction for homogeneous, constant coefficient 
elliptic operators.
\begin{lemma}\label{lem:parametrix} Suppose $L=\sum_{|\alpha|=d} a^\alpha\partial_\alpha$ is a constant
coefficient elliptic differential operator.  
There exists maps $Q$ and $T$ acting on tempered distributions supported on $B_R(0)$ such that 
\begin{itemize}
	\item $Q:H^{s-d,p}(\Reals^n)\to H^{s,p}(\Reals^n)$ is 
is continuous for all $s\in\Reals$ and $1<p<\infty$,
	\item $T:H^{s_1,p}(\Reals^n)\to H^{s_2,p}(\Reals^n)$ is continuous for all $s_1,s_2\in\Reals$ and $1<p<\infty$,
	\item $QLu = u + Tu$ for all tempered distributions $u$.
\end{itemize}
\end{lemma}
\begin{proof}
Let $\chi$ be a smooth, compactly supported cutoff function that equals 1 in a neighborhood of zero.
Define the parametrix $Q$ on tempered distributions 
by $Qu =  \mathcal F^{-1} (1-\chi(\xi))(a^\alpha \xi^\alpha)^{-1}\mathcal F$, where $\mathcal F$ is
the Fourier transform. Similarly, let $T$ be the smoothing map
$Tu = \mathcal F^{-1} \chi(\xi) \mathcal F$.  A computation shows $Q Lu = u + T u$ for all tempered distributions.

Setting $\left<\xi\right>=\sqrt{1+|\xi|^2}$, 
the claimed continuity properties of $Q$ follows from factoring the multiplier as
\[
\left[(1-\chi(\xi))(a^\alpha \xi^\alpha)^{-1} \left<\xi\right>^{d}\right]\left<\xi\right>^{-d} 
\]
The continuity of the multiplier operator determined by the the first factor
follows from the Mikhlin multiplier theorem whereas the second factor is handled by the 
definition of Bessel potential spaces.

The smoothing map $T$ has a compactly supported multiplier
and its continuity properties follow from the same arguments as above, without restriction on the gain
in derivatives.
\end{proof}

We now establish Proposition \ref{prop:elliptic-zoom}, the regularity result
for functions supported in a sufficiently small region near a point where an operator is elliptic.  The 
statement of this proposition requires notation for function spaces associated with compactly supported
functions on a bounded open set $\Omega$, and there are two
natural classes of spaces one can work with.  The first is 
the closure of $\mathcal D(\Omega)$ in $H^{s,p}(\Omega)$.
We find it more convenient to take the closure of
$\mathcal D(\Omega)$ in $H^{s,p}(\Reals^n)$ instead;
following the notation of \cite{mclean_strongly_2000} we denote this 
latter space by  $\widetilde H^{s,p}(\Omega)$.
An element in $\widetilde H^{s,p}(\Omega)$ is, by definition, an element of $H^{s,p}(\Reals^n)$
and it is easy to see that it has support in $\overline{\Omega}$.   Moreover, if 
$u\in H^{s,p}(\Omega)$ has support contained in a compact 
set $V\subseteq \Omega$, an easy argument
using a cutoff function that equals 1 on $V$ and vanishes outside $\Omega$ shows that there exists 
a unique 
$\tilde u \in \widetilde H^{s,p}(\Omega)$ with $\tilde u|_{\Omega} = u$, and indeed one 
has the estimate $||\tilde u||_{\tilde H^{s,p}(\Omega)} \lesssim ||u||_{H^{s,p}(\Omega)}$
with implicit constant depending on $V$. Following standard practice we informally identify
$u$ with its zero extension $\tilde u$.

\begin{proposition}\label{prop:elliptic-zoom} 
Let $\Omega\subset \Reals^n$ be a bounded open set.  
Suppose $s\in\Reals$, $1<p<\infty$, $d,d_0\in\Ints_{\ge 0}$ with $d_0\le d$, that $s>n/p$,
and that these parameters satisfy inequalities \eqref{eq:S-p-1}--\eqref{eq:S-p-2}
of Lemma \ref{lem:S-members} and hence 
$\mathcal S_{d_0}^d(H^{s,p})\neq \emptyset$.
Suppose additionally that 
$L=\sum_{|\alpha|\le d} a^\alpha \partial_\alpha$ 
is a differential operator of class $\mathcal L_{d_0}^d(H^{s,p};\Omega)$
and that for some $x\in \Omega$ that
\[
L_0 = \sum_{|\alpha|=m} a^\alpha(x)\partial_\alpha
\]
is elliptic.  
Given $(\sigma,q)\in \mathcal S_{d_0}^d(H^{s,p})$
there exists $r>0$ such that $\overline{B_r(x)}\subset \Omega$
and such that if
\begin{align*}
u&\in \widetilde H^{d-s,p^*}( B_{r}(x) )\quad \text{and}\\
Lu&\in H^{\sigma-d,q}(\Omega)
\end{align*}
then $u\in H^{\sigma,q}(\Omega)$ and
\begin{equation}\label{eq:pointwise-est-Hsp}
||u||_{H^{\sigma,q}(\Omega)} \lesssim
||L u||_{H^{\sigma-d,q}(\Omega)}
+ ||u||_{H^{d-s-1,p^*}(\Omega)}		
\end{equation}
with implicit constant independent of $u$ but depending on all other parameters.
\end{proposition}
\begin{proof}
It suffices to prove the result assuming $x=0$ and that $B_1(0)\subset\Omega$.
From the definition of Bessel potential spaces on bounded domains, we can further assume that 
the coefficients of $L$ have been extended to all of $\Reals^n$.

For each $r\in(0,1]$ define
\[
L_\rscrq = L_0 + B_\rscrq + C_\rscrq
\]
where
\begin{equation}
\begin{aligned}
B_\rscrq &= \sum_{|\alpha|=d} (a^\alpha-a^\alpha(0))_\rscr\partial_\alpha \\
C_\rscrq &= \sum_{|\alpha|<d} r^{d-|\alpha|} (a^{\alpha})_\rscr  \partial_\alpha.
\end{aligned}
\end{equation}

Suppose $u\in \widetilde H^{d-s,p^*}(B_r(0))$ 
for some $r< 1$ and that $Lu\in H^{\sigma-d,q}(\Omega)$. 
Recall that by definition $u\in H^{d-s,p^*}(\Reals^n)$.  Moreover, 
$Lu$ is compactly supported in $\Omega$ and hence defines an element
of $\widetilde H^{\sigma-d,q}(\Omega)\subset  H^{\sigma-d,q}(\Reals^n)\subset H^{-s,p^*}(\Reals^n)$.
In particular, we can treat $u$ and $Lu$ as distributions on $\Reals^n$ and 
a short computation shows $u_\rscr$ satisfies
\[
L_0 u_\rscr + B_\rscrq u_\rscr + C_\rscrq u_\rscr = r^{d} (Lu)_\rscr
\]
as an equation in $H^{-s,p^*}(\Reals^n)$.

Pick a cutoff function $\chi$ that equals $1$ 
on a neighborhood of $\overline{B_1(0)}$.  
Since $u_\rscr$ is supported on $\overline{B_1(0)}$,
\begin{equation}\label{eq:pre-parametrix}
L_0 u_\rscr + \chi B_\rscrq u_\rscr + \chi C_\rscrq u_\rscr = r^{d} (Lu)_\rscr	
\end{equation}
as well.

Let $Q$ and $T$ be the parametrix and smoothing operator for $L_0$ from Lemma \ref{lem:parametrix}.
Applying $\chi Q$ to equation \eqref{eq:pre-parametrix} we have
\begin{equation}\label{eq:post-parametrix}
u_\rscr + \chi T u_\rscr + \chi Q\circ(\chi (B_\rscrq+C_\rscrq)) u_\rscr = r^{d} \chi Q (Lu)_\rscr.
\end{equation}
It will be convenient to define $Q'=\chi Q$ and $T'=\chi T$, in which case $Q'$ has the
same continuity properties as in Lemma \ref{lem:parametrix} and, using the compact support of $\chi$,
$T'$ is a continuous map 
$H^{s_1,p_1}(\Reals^n)\to H^{s_2,p_2}(\Reals^n)$ for all choices of $s_1,s_2,p_1,p_2$.

Consider a coefficient of $C_\rscrq$, 
$c^\alpha = r^{d-|\alpha|} (a^\alpha)_\rscr\in H^{s-d+|\alpha|,p}(\Reals^n)$.
From Proposition \ref{prop:poor-mans-bp}, for any $\epsilon>0$
\begin{equation}
	\label{eq:c-coeef-est-Hsp}
	||\chi c^\alpha ||_{H^{s-d+|\alpha|,p}(\Reals^n)} \lesssim 
r^{d-|\alpha|} r^{\min\left(s-d+|\alpha|-\frac{n}{p},0\right)-\epsilon} ||a^\alpha||_{H^{s+d-|\alpha|,p}(\Reals^n)}
= r^{\min\left(s-\frac{n}{p},d-|\alpha|\right)-\epsilon} ||a^\alpha||_{H^{s+d-|\alpha|,p}(\Reals^n)}.
\end{equation}
Similarly, consider
a coefficient $b^\alpha = (a^{\alpha}-a^{\alpha}(0))_\rscr$ of $B_\rscrq$.
Since $a^\alpha\in H^{s,p}(\Reals^n)$ when $|\alpha|=d$, and since
$b^\alpha(0)=0$, Proposition \ref{prop:poor-mans-bp} implies
\begin{equation}\label{eq:b-coeff-est-Hsp}
||\chi b^\alpha||_{H^{s,p}(\Reals^n)} \lesssim r^{\min\left(s-\frac{n}{p},1\right)-\epsilon} ||b^\alpha||_{H^{s,p}(\Reals^n)}.
\end{equation}
Pick $\epsilon>0$ with  $\epsilon < \min\left(s-\frac{n}{p},1\right)$.
Since $d-|\alpha|\ge 1$ in estimate \eqref{eq:c-coeef-est-Hsp} it follows
that the coefficients of
$\chi B_\rscrq$ and $\chi C_{\rscrq}$ converge to zero in the norms indicated on the
left-hand sides of inequalities \eqref{eq:c-coeef-est-Hsp} and \eqref{eq:b-coeff-est-Hsp}
as $r\to 0$.
Using the fact that these coefficients are compactly supported in a common bounded open set,
a computation using 
Theorem \ref{thm:mult} shows that $\chi (B_\rscrq+C_\rscrq)$ converges
to zero as an operator $H^{\sigma',q'}(\Reals^n)\to H^{\sigma'-d,q'}(\Reals^n)$
as $r\to0$ for any choice of $(\sigma',q')\in \mathcal S^{d}_{d_0}(H^{s,p})$.
Hence we may take $r$
sufficiently small so that $I + Q'\circ (\chi(B_\rscrq+C_\rscrq))$ has a 
continuous
inverse $U_\rscrq:H^{s-d,p^*}(\Reals^n)\rightarrow 
H^{s-d,p^*}(\Reals^n)$ that also maps
$H^{s+d_0,p}(\Reals^n)\rightarrow 
H^{s+d_0,p}(\Reals^n)$ and 
$H^{\sigma,q}(\Reals^n)\rightarrow 
H^{\sigma,q}(\Reals^n)$.
Applying $U_\rscrq$ to equation \eqref{eq:post-parametrix}
we conclude
\[
u_\rscr  + U_\rscrq T' u_\rscr = r^{d} U_\rscrq Q' (Lu)_\rscr
\]
and consequently
\begin{equation}
	\begin{aligned}
||u_\rscr||_{H^{\sigma,q}(\Reals^n)} &\lesssim ||U_\rscrq T' u_\rscr||_{H^{\sigma,q}(\Reals^n)} + 
 || U_\rscrq Q' (Lu)_\rscr ||_{H^{\sigma,q}(\Reals^n)}\\
 & \lesssim ||U_\rscrq T' u_\rscr||_{H^{s+d_0,p}(\Reals^n)} + 
|| Q' (Lu)_\rscr ||_{H^{\sigma,q}(\Reals^n)}\\
& \lesssim ||u_\rscr||_{H^{d-s-1,p^*}(\Reals^n)} + 
||(Lu)_\rscr||_{H^{\sigma-d,q}(\Reals^n)}.
\end{aligned}
\end{equation}
Note that the implicit constants above depend on $r$, but this dependence is unimportant
since the smallness of $r$ has already be chosen.  Since rescaling $v\mapsto v_\rscr$ with fixed $r$
is a continuous automorphism of any Bessel potential space on $\Reals^n$ we conclude
\begin{equation}\label{eq:pointwise-nearly-done-Hsp}
||u||_{H^{\sigma,q}(\Reals^n)} 
\lesssim 
||Lu||_{H^{\sigma-d,q}(\Reals^n)}
+||u||_{H^{d-s-1,p^*}(\Reals^n)}. 
\end{equation}
Estimate \eqref{eq:pointwise-est-Hsp} follows from inequality
\eqref{eq:pointwise-nearly-done-Hsp} along with the fact that 
if $v\in H^{\sigma',q'}(\Reals^n)$ for some $(\sigma',q')$
has support on a fixed $\overline{B_r(0)}\subset \Omega$,
then $||v||_{H^{\sigma',q'}(\Omega)} \sim ||v||_{H^{\sigma',q'}(\Reals^n)}$
with implicit constants depending on $r$.
\end{proof}
We now arrive at our main regularity result.
\begin{theorem}\label{thm:interior-reg} 
Let $\Omega$ be a bounded open set in $\Reals^n$ and suppose $s,p, d_0$ and $d$ are parameters
as in Lemma \ref{lem:S-members} such that
$s>n/p$ and such that inequalities \eqref{eq:S-p-1}--\eqref{eq:S-p-2}
are satisfied so $\mathcal S_{d_0}^d(H^{s,p})\neq \emptyset$.  
Suppose $L$ is of class $\mathcal L_{d_0}^d(H^{s,p};\Omega)$
and is elliptic on $\Omega$.  If $u\in H^{d-s,p^*}(\Omega)$
and $Lu\in H^{\sigma-d,q}(\Omega)$ for some 
$(\sigma,q)\in \mathcal S_{d_0}^d(H^{s,p})$
then for any open set $U$ with $\overline U\subseteq \Omega$, $u\in H^{\sigma,q}(U)$ and 
\begin{equation}\label{eq:int-reg}
||u||_{H^{\sigma,q}(U)} \lesssim 
||Lu||_{H^{\sigma-d,q}(\Omega)}+
||u||_{H^{d-s-1,p^*}(\Omega)}.
\end{equation}
\end{theorem}

\begin{proof}
The proof is a bootstrap that relies on the following main step.  We have initially assumed
that $u\in H^{d-s,p^*}(\Omega)$ so that $L$ can be applied to it,
and that $Lu\in H^{\sigma-d,q}(\Omega)$ for some 
$(\sigma,q)\in \mathcal S^{d}_{d_0}(H^{s,p})$.  Suppose we know additionally
that on some open set $\Omega_A$ with $\overline U\subset \Omega_A\subset\Omega$
that
$u\in H^{\sigma_A,q_A}(\Omega_A)$ for some pair $(\sigma_A,q_A)$ such
that the commutator result Lemma \ref{lem:commutator-Hsp} applies.
Now consider
a target level of regularity $(\sigma_B,q_B)$ satisfying the following:
\begin{enumerate}
	\item[\optionaldesc{H1}{cond:hard-deriv}] $\sigma_B \le \sigma$,
	\item[\optionaldesc{H2}{cond:hard-lebesgue}] $\displaystyle \frac{1}{q}-\frac{\sigma}{n} \le \frac{1}{q_B}-\frac{\sigma_B}{n}$,
	\item[\optionaldesc{H3}{cond:soft-deriv}]$\sigma_B\le \sigma_A + 1$,
	\item[\optionaldesc{H4}{cond:soft-lebesgue}] $\displaystyle \frac{1}{q_A}-\frac{\sigma_A+1}{n}\le \frac{1}{q_B}-\frac{\sigma_B}{n}$.
\end{enumerate}
The first two conditions ensure via Sobolev embedding 
that $H^{\sigma,q}(\Omega)$ is contained in $H^{\sigma_B,q_B}(\Omega)$
and form a hard limit on the target regularity.  The second two conditions ensure
$H^{\sigma_B,q_B}(\Omega_A) \supset H^{\sigma_A+1,q_A}(\Omega_A)$ and limit
the improvement in regularity that can be achieved on a single step of the bootstrap.  We
claim that under these hypotheses that $u\in H^{\sigma_B,q_B}(\Omega_B)$
for some open set $\Omega_B\subset \Omega_A$
such that $\overline U\subset \Omega_B$ and that we have the estimate
\begin{equation}\label{eq:bootstrap-main}
	||u||_{H^{\sigma_B,q_B}(\Omega_{B})} \lesssim 
	||Lu||_{H^{\sigma-d,q}(\Omega)} +
	||u||_{H^{\sigma_A,q_{A}}(\Omega_{A})}
\end{equation}
with implicit constant independent of $u$.

To establish inequality \eqref{eq:bootstrap-main} we first select
an open set $\Omega_B$ with 
$\overline U\subseteq \Omega_{B}\subseteq\overline \Omega_{B}\subseteq \Omega_{A}$.
Since $\overline{\Omega_B}$ is compact we can select finitely many balls 
$B_i=B_{r_i}(x_i)\subset \Omega_0$
that cover $\Omega_B$ and such that the conclusion of Proposition
\ref{prop:elliptic-zoom} holds for the pair $(\sigma_{B},q_{B})$.
Using a partition 
of unity subordinate to the balls $B_{i}$ and 
$\Omega_{A}\setminus \overline \Omega_{B}$ we can find non-negative
smooth functions $\phi_i$ compactly supported in $B_i$ such that
$\sum \phi_i = 1$ on $\Omega_{B}$.

Consider
\begin{equation}\label{eq:L-split}
	L (\phi_i u) = \phi_i L u + [L,\phi_i] u.	
\end{equation}
From conditions \eqref{cond:hard-deriv}--\eqref{cond:hard-lebesgue} and Sobolev embedding
we know
\begin{equation}\label{eq:L-split-first}
	||\phi_i L u||_{H^{\sigma_B-d,q_B}(\Omega_B)} \lesssim ||L u||_{H^{\sigma_B-d,q_B}(\Omega_B)}
	\lesssim ||L u||_{H^{\sigma-d,q}(\Omega)}.
\end{equation}
Conditions \eqref{cond:soft-deriv}--\eqref{cond:soft-lebesgue} allow us to apply Sobolev embedding
to the commutator term from equation \eqref{eq:L-split} and we have
\begin{equation}\label{eq:L-split-second-a}
	||[L,\phi_i] u||_{H^{\sigma_{B}-d,q_{B}}(B_i)} \lesssim 
	||[L,\phi_i] u||_{H^{\sigma_{A}-d+1,q_{A}}(B_i)}.
\end{equation}
Since we have assumed that $(\sigma_A,q_A)$ satisfies the conditions of 
Lemma \ref{lem:commutator-Hsp}
(i.e., either $(\sigma_A+1,q_A)\in \mathcal S^{d}_{d_0}(H^{s,p})$ or $d_0=0$ and 
$(\sigma_A,q_A)\in \mathcal S^{d}_{d_0}(H^{s,p})$), we  have
\begin{equation}\label{eq:L-split-second-b}
	||[L,\phi_i] u||_{H^{\sigma_{A}-d+1,q_A}(B_i)} \lesssim 
	||u||_{H^{\sigma_{A},q_A}(B_i)} \lesssim 
	||u||_{H^{\sigma_{A},q_A}(\Omega_A)}.
\end{equation}
Combining inequalities equalities \eqref{eq:L-split-first}, 
\eqref{eq:L-split-second-a} and \eqref{eq:L-split-second-a} 
we find $L  (\phi_i u) \in H^{\sigma_{B}-d,q_{B}}(B_i)$
and we conclude from 
Proposition \ref{prop:elliptic-zoom} that
$\phi_i u\in H^{\sigma_{B},q_{B}}(\Reals^n)$
and additionally
\[
||\phi_i u||_{H^{\sigma_{B},q_{B}}(\Omega_A)} \lesssim
||Lu||_{H^{\sigma,q}(\Omega)} +
||u||_{H^{\sigma_{A},q_A}(\Omega_A)}.
\]
Inequality \eqref{eq:bootstrap-main} now follows from 
the observation 
$||\phi_i u||_{H^{\sigma_{B},q_{B}}(\Omega_B)}\lesssim ||\phi_i u||_{H^{\sigma_{B},q_{B}}(\Omega_A)}$
and summing on $i$.

We now describe the bootstrap in the easier case when $d_0=0$, where 
Lemma \ref{lem:commutator-Hsp} has the fewest restrictions.
The argument begins
with $(\sigma_0,q_0)=(d-s-1,p^*)$ and $(\sigma_1,q_1)=(d-s,p^*)$.  Conditions
\eqref{cond:hard-deriv}--\eqref{cond:hard-lebesgue} are an immediate consequence of
the definition of the region $\mathcal S^{d}_0(H^{s,p})$ and conditions
\eqref{cond:soft-deriv}--\eqref{cond:soft-lebesgue} are obvious.  Moreover,
since $(\sigma_0+1,q_0)=(d-s,p^*)\in \mathcal S^{d}_0(H^{s,p})$, Lemma \ref{lem:commutator-Hsp}
applies.  Hence all the conditions of the bootstrap step are met and we conclude
there is an open set $\Omega_1$ with $\overline U \subset \Omega_1\subset \Omega$
such that
\begin{equation}\label{eq:bootstrap-start}
	||u||_{H^{d-s,p^*}(\Omega_1)} \lesssim
	||Lu||_{H^{\sigma,q}(\Omega)} +
	||u||_{H^{d-s-1,p^*}(\Omega)}.
\end{equation}
We now iteratively apply the bootstrap step through a finite sequence 
$(\sigma_k,q_k)$ in $S^{d}_0(H^{s,p})$ described below that starts at $(s-d,p^*)$ 
and terminates at $(\sigma,q)$. At each step we ensure conditions
\eqref{cond:hard-deriv}--\eqref{cond:soft-lebesgue} hold
and obtain inequalities
\begin{equation}\label{eq:bootstrap-middle}
	||u||_{H^{\sigma_{k+1},q_{k+1}}(\Omega_{k+1})} \lesssim
	||Lu||_{H^{\sigma,q}(\Omega)} +
	||u||_{H^{\sigma_k,q_k}(\Omega_k)}
\end{equation}
for a sequence nested open sets
$\overline U \subset \Omega_{k+1}\subset\Omega_k \subset \Omega$.  Because
we have assumed that $d_0=0$, and since each $(\sigma_k,q_k)\in \mathcal S^{d}_0(H^{s,p})$,
we are assured that at each step we can use the commutator estimate from Lemma \eqref{lem:commutator-Hsp}.
Inequality \eqref{eq:int-reg} follows from chaining together the initial estimate
\eqref{eq:bootstrap-start} with the estimates \eqref{eq:bootstrap-middle} obtained
along the bootstrap.

The specific sequence $(\sigma_k,q_k)$ can be achieved as follows, starting from 
$(\sigma_1,q_1)=(d-s,p^*)$, in two cases depending on whether $1/q\le 1/p^*$ or not
as depicted in Figure \ref{fig:bootstrap}.

If $1/q \le 1/p^*$ we first lower $1/q_k$ by steps of at most
$1/n$ until it has been lowered to $1/q$.  At this point we raise $\sigma_k$ by
steps of at most $1$ until it has been raised to $q$.  Note that at each step $k$,
\[
\frac{1}{p}-\frac{s}{n} \le 
\frac{1}{q}-\frac{\sigma}{n} \le \frac{1}{q_k}-\frac{\sigma_k}{n} \le \frac{1}{q_1} -\frac{\sigma_1}{n}
=\frac{1}{p^*}-\frac{d-s}{n}
\]
and hence the sequence remains in $\mathcal S^{d}_0(H^{s,p})$ as is required to apply
Lemma \ref{lem:commutator-Hsp}.  These same
inequalities also show that conditions \eqref{cond:hard-deriv}--\eqref{cond:hard-lebesgue}
are maintained at each step.  Moreover,  
conditions \eqref{cond:soft-deriv}--\eqref{cond:soft-lebesgue} hold because
we either fix $\sigma_k$ and lower $q_k$ by at most $1/n$ or we fix $q_k$ and raise $\sigma_k$
by at most $1$.

Now suppose $1/q>1/p^*$.  Since $1/q-\sigma/n\le 1/p^*-(d-s)/n$, we can lower $\sigma$ to
a value $\sigma'$ such that the inequality becomes an equality.  We now start the bootstrap
by raising $\sigma_k$ to $\min(\sigma_k+1,\sigma')$ 
at each step while simultaneously raising 
$1/q_k$ so that the value $1/q_k-\sigma_k/n=1/p^*-(d-s)/n$ remains invariant. This
stage ends when $(\sigma_k,q_k)=(\sigma',q)$.  We then increase $\sigma_k$ while leaving $q_k$
fixed as in the earlier argument when $1/q<1/p^*$.  The first stage of the sequence 
has $d-s\le \sigma_k \le \sigma'\le \sigma\le s$ and 
\[
\frac{1}{p}-\frac{s}{n}\le \frac{1}{q}-\frac{\sigma}{n} \le \frac{1}{p^*}-\frac{d-s}{n} 
= \frac{1}{q_k}-\frac{\sigma_k}{n}.
\]
Hence for this first part of the sequence the terms $(\sigma_k,q_k)$ lie 
in $\mathcal S^{d}_0(H^{s,p})$ and additionally
conditions \eqref{cond:hard-deriv}--\eqref{cond:hard-lebesgue} are maintained.
Moreover, condition \eqref{cond:soft-deriv} is enforced because we raise $\sigma_k$
by at most $1$, and condition \eqref{cond:soft-lebesgue} is satisfied because
it is an equality during this first stage where we lower $1/q_k$.  The same argument
as in the case when $1/q\le 1/p^*$ then shows that conditions 
\eqref{cond:hard-deriv}--\eqref{cond:soft-lebesgue}
hold in the second stage when we raise $\sigma_k$ while leaving $q_k$ fixed.
The proof is now complete in event that $d_0=0$.

\begin{figure}
	\begin{center}
		\includegraphics{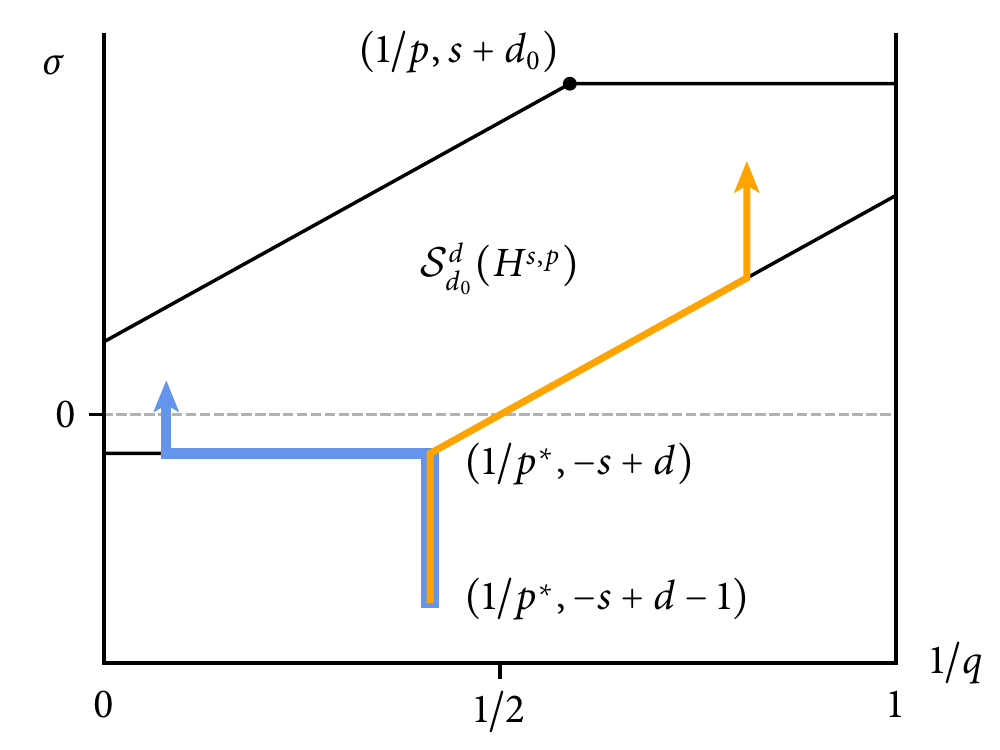}
	\end{center}
	\caption{
	\label{fig:bootstrap}
Two possible bootstrap paths.}
\end{figure}
	
We now turn to the case $d_0>0$ where the bootstrap requires more care 
because the hypotheses of Lemma \ref{lem:commutator-Hsp} are more
restrictive.
Consider the case $d_0=1$.  If $(\sigma,q)\in \mathcal S^{d}_0(H^{s,p})$, we can
simply apply the earlier result, so it suffices to assume  
$(\sigma,q)\in\mathcal S^{d}_1(H^{s,p}))$ but $(\sigma,q)\not \in\mathcal S^{d}_0(H^{s,p})$.
Because $d_0=1$ we can define a point $(\sigma',q')$
in $\mathcal S^{d}_0(H^{s,p})$  determined by $(\sigma,q)$ and the following rules:
\begin{enumerate}
	\item If $\sigma<s$, leave $\sigma'=\sigma$ fixed but raise $1/q$ by at most $1/n$ to $1/q'$
	such that $1/q'-\sigma/n=1/p-s/n$.
	\item If $\sigma\ge s$ and $1/p-s/n \le 1/q-\sigma/n$, lower $\sigma$ by at most 1 to $s$ while simultaneously 
	lowering $1/q$ by at most $1/n$ so that the Lebesgue regularity $1/q-\sigma/n=1/q'-\sigma'/n$ is unchanged.
	\item Otherwise, $(\sigma,q)$ satisfies $s\le \sigma\le \sigma+1$ and 
	\[
		\frac{1}{p}-\frac{s+1}{n} \le \frac{1}{q}-\frac{\sigma}{n} < \frac{1}{p}-\frac{s}{n}
	\]
	and we set $(\sigma',q')=(s,p)$.
\end{enumerate}
In each of these cases $H^{s,q}(\Omega)$ is contained in $H^{s',q'}(\Omega)$ by Sobolev embedding,
so we can apply the original bootstrap for $d_0=0$ to get to $(\sigma',q')$. 
Since $(\sigma',q')\in \mathcal S^{d}_0(H^{s,p})$ a computation shows $(\sigma'+1,q')\in \mathcal S^{d_0}_1(H^{s,p})$
and hence the commutator result Lemma \ref{lem:commutator-Hsp} can be applied at $(\sigma',q')$.
One verifies that in all three cases listed above,
the target regularity $(\sigma,q)$ satisfies conditions \eqref{cond:hard-deriv}-\eqref{cond:soft-lebesgue}
when starting from $(\sigma',q')$ and 
we can apply the bootstrap step exactly once to arrive at $(\sigma,q)$.  This proves
the result when $d_0=1$, and iterating this argument obtains the proof for any value of $d_0$.
\end{proof}
\section{Coefficients in Triebel-Lizorkin Spaces}
\label{s:Fsp}

This section generalizes the results of Section \ref{s:Hsp} to operators having
coefficients in Triebel-Lizorkin spaces, 
which are defined defined in terms of Littlewood-Paley projection operators.
Let $\phi$ be a smooth radial bump function that equals
$1$ on $B_1$ and vanishes outside of $B_2$ and
let $\psi(\xi) = \phi(\xi)-\phi(2\xi)$, so
$\psi(\xi)=0$ if $|\xi|>2$ or $|\xi|<1/2$.
For $k\in \Ints$ we define the Littlewood-Paley
projections $P_k$ and $P_{\le k}$
\[
\begin{aligned}
\mathcal F[ P_k u](\xi) &= \psi(2^{-k}\xi) \mathcal F[u](\xi) \\
\mathcal F[ P_{\le k} u](\xi) &= \phi(2^{-k}\xi) \mathcal F[u](\xi).
\end{aligned}
\]
We also use the notation $P_{a\le \cdot\le b}$ for $\sum_{j=a}^b P_j$.

Let $s\in\Reals$ and $1<p,q<\infty$.
Given a tempered distribution $u$, each $P_k u$ is an 
analytic function on $\Reals^n$ and 
a tempered distribution belongs to the Triebel-Lizorkin
space $F^{s,p}_q(\Reals)$ if 
\[
||u||_{F^{s,p}_q(\Reals^n)} = ||P_{\le 0}u||_{L^p(\Reals^n)}+\left|\left|\,\,
\left[ \sum_{k=1}^\infty ( 2^{sk}|P_k u|)^q\,\,\right]^{1/q}\right|\right|_{L^p(\Reals^n)}
<\infty.
\]
We call $p$ the \textbf{Lebesgue parameter}, whereas $q$ 
is the \textbf{fine parameter}.  
Bessel potential spaces are 
Triebel-Lizorkin spaces with fine parameter $q=2$ and, as recalled in the
following section, Sobolev-Slobodeckij
spaces $W^{s,p}(\Reals^n)$ are 
also special cases of Triebel-Lizorkin spaces with either $q=2$
or $q=p$.
Just as for Bessel potential spaces, on an open set 
$\Omega\subseteq \Reals^n$, $F^{s,p}_q(\Omega)$
consists of the restrictions to $\Omega$ of distributions in
$F^{s,p}_q(\Reals^n)$ to $\Omega$ and is given the quotient norm.
The text \cite{triebel_theory_2010} contains a 
comprehensive description of Triebel-Lizorkin spaces, and indeed considers a
wider set of parameters than those we employ here.

Embedding properties for Triebel-Lizorkin spaces follow those
for Bessel potential spaces with the following rule 
of thumb: when a loss of derivatives is involved,
the fine parameter has no role. Specifically, we 
recall the following results summarizing \cite{triebel_theory_2010} Proposition 2.3.2/2 
and Theorems 2.7.1 and 3.3.1 along with \cite{triebel_interpolation_1978} Theorems 2.8.1 and 4.6.1; the notation $A\hookrightarrow B$ denotes a continuous inclusion of
space $A$ into space $B$.  Note that here and elsewhere when we quote established results,
we restrict the range
of the parameters to $1<p,q<\infty$ even when they apply with greater generality.

\begin{proposition}\label{prop:embedding-Fsp}
Assume $1<p,p_1,p_2,q,q_1,q_2< \infty$
and $s,s_1,s_2\in \Reals$, and suppose $\Omega$ is a bounded
open set in $\Reals^n$.
\begin{enumerate}
	\item
	If $s_1>s_2$ then $F^{s_1,p}_{q_1}(\Reals^n)\hookrightarrow F^{s_2,p}_{q_2}(\Reals^n)$ and 
	$F^{s_1,p}_{q_1}(\Omega)\hookrightarrow F^{s_2,p}_{q_2}(\Omega)$.
	\item If $q_1\le q_2$ then $F^{s,p}_{q_1}(\Reals^n)\hookrightarrow F^{s,p}_{q_2}(\Reals^n)$ and 
	$F^{s_1,p}_{q_1}(\Omega)\hookrightarrow F^{s_2,p}_{q_2}(\Omega)$.
	\item If $p_1 \ge p_2$ then $F^{s,p_1}_{q}(\Omega)\hookrightarrow F^{s,p_2}_q(\Omega)$.
	\item
	If $s_1>s_2$ and
	$\frac{1}{p_1}-\frac{s_1}{n} = \frac{1}{p_2}-\frac{s_2}{n}$ then
	$F^{s_1,p_1}_{q_1}(\Reals^n)\hookrightarrow F^{s_2,p_2}_{q_2}(\Reals^n)$.
	\item\ If $s_1>s_2$ and 
	$\frac{1}{p_1}-\frac{s_1}{n}\le \frac{1}{p_2}-\frac{s_2}{n}$
	then $F^{s_1,p_1}_{q_1}(\Omega)\hookrightarrow F^{s_2,p_2}_{q_2}(\Omega)$.
	\item If $0<\alpha<1$ then $F^{\frac{n}p+\alpha,p}_q(\Reals^n)\hookrightarrow C^{0,\alpha}(\Reals^n)$
	and $F^{\frac{n}p+\alpha,p}_q(\Omega)\hookrightarrow C^{0,\alpha}(\Omega)$.
\end{enumerate}
\end{proposition}
Note that for bounded open sets $\Omega$, \cite{triebel_theory_2010} and \cite{triebel_interpolation_1978}
proves these 
embedding properties under the the addition hypothesis that $\Omega$ is a $C^\infty$ domain.
The results for arbitrary bounded open sets follow as a corollary using the quotient
space definition of the relevant spaces.

Complex interpolation of Triebel-Lizorkin spaces
is described in \cite{triebel_theory_2010} Theorems 2.4.7 and 3.3.6.
\begin{proposition}
	Assume $1<p_1,p_2,q_1,q_2<\infty$
	and $s_1,s_2\in \Reals$, and suppose $\Omega$ is either $\Reals^n$ or is
	a bounded $C^\infty$ domain in $\Reals^n$.
	For $0<\theta<1$,
	\[
		[F^{s_1,p_1}_{q_1}(\Omega),F^{s_2,p_2}_{q_2}(\Omega)]_\theta = F^{s,p}_q(\Omega)
	\]
	where
	\[
		s=(1-\theta)s_1+\theta s_2,\quad \frac{1}{p} = (1-\theta)\frac{1}{p_1}+\theta \frac{1}{p_2},\quad
		(1-\theta)\frac{1}{q_1}+\theta \frac{1}{q_2}.
	\]
\end{proposition}
Duality of spaces of functions on $\Reals^n$ follows from \cite{triebel_theory_2010} Theorem 2.11.2.  Duality 
for Lipschitz bounded domains can be found in \cite{triebel_function_2002}, but the theory is more complex and
we have avoided its use in our arguments; see, e.g., Proposition \ref{prop:mult-Fsp-s2-neg}.
\begin{proposition}\label{prop:dual-Fsp}
	Assume $1<p,q<\infty$
	and $s\in \Reals$.
	The bilinear map $\mathcal D(\Reals^n)\times \mathcal D(\Reals^n)\to \Reals$
	given by $\left<f,g\right> := \int_\Omega fg$ extends to a continuous bilinear map
	$F^{s,p}_q(\Reals^n)\times F^{-s,p^*}_{q^*}(\Reals^n)\to \Reals$ 
	where $1/p^* = 1-1/p$ and $1/q^*=1-1/q$.  Moreover, 
	$f\mapsto \left<f,\cdot\right>$ is a continuous identification of $F^{s,p}_q(\Reals^n)$
	with $(F^{-s,p^*}_{q^*}(\Reals^n))^*$.
\end{proposition}

\subsection{Mapping properties}

Operators with coefficients in Triebel-Lizorkin spaces
are defined analogously to those of Definition \ref{def:L-Hsp}.
\begin{defn}\label{def:L-Fsp}
Suppose $d_0,d\in \Ints_{\ge 0}$ with $d_0\le d$.
A differential operator on 
an open set $\Omega\subseteq \Reals^n$ of the form
\[
\sum_{d_0\le |\alpha|\le d} a^\alpha \partial_\alpha
\]
is of class $\mathcal L_{d_0}^d(F^{s,p}_q;\Omega)$ 
for some $s\in\Reals$, and $1<p,q<\infty$ if each
\[
a^\alpha \in F^{s+|\alpha|-d,p}_q(\Omega ).
\]
\end{defn}

We have the following multiplication rules for Triebel-Lizorkin
spaces that generalize the rules found in Theorem \ref{thm:mult}.
A self-contained proof is given in Appendix \ref{app:mult-Fsp}.

\begin{theorem}\label{thm:mult-Fsp}
Let $\Omega$ be a bounded open subset of $\Reals^n$.
Suppose $1<p_1,p_2,p,q_1,q_2,q<\infty$ and $s_1,s_2,s\in\Reals$.  Let $r_1,r_2$ and $r$ be defined by
\[
\frac{1}{r_1} = \frac{1}{p_1} - \frac{s_1}{n},\qquad
\frac{1}{r_2} = \frac{1}{p_2} - \frac{s_2}{n},\quad\text{\rm and}\quad
\frac{1}{r} = \frac{1}{p} - \frac{s}{n}.
\]
Pointwise multiplication of $C^\infty(\overline \Omega)$ 
functions extends to a continuous bilinear map 
$F^{s_1,p_1}_{q_1}(\Omega)\times F^{s_2,p_2}_{q_2}(\Omega)
\rightarrow F^{s,p}_{q}(\Omega)$ so long as
\begin{align}
s_1+s_2&\ge 0\\
\min({s_1,s_2})&\ge s\\
\max\left(\frac{1}{r_1},\frac{1}{r_2}\right) & \le  \frac{1}{r}\\
\frac{1}{r_1} + \frac{1}{r_2} & \le 1\\
\label{eq:Fsp-lebesgue-mult}
\frac{1}{r_1} + \frac{1}{r_2} & \le \frac{1}{r}
\end{align}
with the the following caveats:
\begin{itemize}
	\item Inequality \eqref{eq:Fsp-lebesgue-mult} is 
	strict if $\min(1/r_1,1/r_2,1-1/r)=0$.
	\item If $s_i=s$ for some $i$ then $1/q\le 1/q_i$.
	\item If $s_1+s_2=0$ then $\frac{1}{q_1}+\frac{1}{q_2}\ge 1$.
	\item If $s_1=s_2=s=0$ then $\displaystyle \frac1q\le \frac{1}{2}\le \frac1{q_i}$ for $i=1,2$.
\end{itemize}
\end{theorem}

Repeated applications of Theorem \ref{thm:mult-Fsp}
lead to the following analogue of Proposition \ref{prop:mapping-Hsp}.
\begin{proposition}\label{prop:mapping-Fsp}
Let $\Omega$ be a bounded open set in $\Reals^n$.
Suppose $1<p,q,a,b<\infty$, 
$s>n/p$, $\sigma\in\Reals$ and 
$d,d_0\in\Ints_{\ge 0}$ with $d\ge d_0$.  
An operator of class $\mathcal L_{d_0}^d(F^{s,p}_q;\Omega)$
extends from a map 
$C^\infty(\overline \Omega)\mapsto \mathcal D'(\Omega)$ to 
a continuous linear map 
$F^{\sigma,a}_b(\Omega)\mapsto F^{\sigma-d,a}_b(\Omega)$
so long as
\begin{equation}\label{eq:S-conds-Fsp}
\begin{aligned}
\sigma&\in [d-s,s+d_0]\\
\frac{1}{a} - \frac{\sigma}{n} &\in \left[
\frac{1}{p} - \frac{s+d_0}{n}, \frac{1}{p^*}-\frac{d-s}{n}
\right]
\end{aligned}
\end{equation}
and so long as:
\begin{itemize}
	\item If $\sigma=s+d_0$ then $\frac{1}{q} \ge \frac{1}{b}$.
	\item If $\sigma=d-s$ then $\frac{1}{b} \ge \frac{1}{q^*} $. 
\end{itemize}
Moreover, operators in $\mathcal L_{d_0}^d(F^{s,p}_q;\Omega)$ 
depend continuously
on their coefficients $a^\alpha \in F^{s-d+\alpha,p}_q(\Omega)$.
\end{proposition}

Note that in the Bessel potential case, $q=b=2$ in Proposition \ref{prop:mapping-Fsp} and the supplemental
conditions at $\sigma=s+d_0$ and $\sigma=d-s$ are always satisfied.
We have the following
generalization of Definition \ref{def:scriptS-Hsp}.
\begin{defn}\label{def:scriptS-Fsp}
Suppose $1<p,q<\infty$, $s\in\Reals$ and $d,d_0\in\Ints_{\ge 0}$ with $d\ge d_0$.
The \textbf{compatible Sobolev indices} 
for an operator of class $\mathcal L_{d_0}^d(F^{s,p}_q;\Omega)$ is the set
\[
\mathcal S_{d_0}^d(F^{s,p}_q) \subseteq \Reals\times (1,\infty)\times(1,\infty)
\]
of tuples $(\sigma,a,b)$ 
satisfying \eqref{eq:S-conds-Fsp} along with the additional 
conditions at the end of Proposition \ref{prop:mapping-Fsp}
when $\sigma=s+d_0$ or $\sigma=d-s$.
\end{defn}

The proof of Lemma \ref{lem:S-members} generalizes
to the Triebel-Lizorkin context with minimal changes, using
Proposition \ref{prop:embedding-Fsp} for facts about
Sobolev embedding.  The only interesting difference is that
the marginal conditions at the end of Proposition
\ref{prop:mapping-Fsp} adds an additional condition
when $s$ is the smallest possible value such that
$\mathcal S_{d_0}^d(F^{s,p}_q)$  is nonempty.
\begin{lemma}\label{lem:S-members-Fsp}
Suppose $1<p,q<\infty$, $s\in\Reals$ and $d,d_0\in \Ints_{\ge 0}$
with $d\ge d_0$.  Then $\mathcal S_{d_0}^d(F^{s,p}_q)$ 
is nonempty if and only if
\begin{align}
\label{eq:S-p-1-Fsp}
s &\ge (d-d_0)/2\text{, and}\\
\label{eq:S-p-2-Fsp}
\frac{1}{p} - \frac{s}{n} &\le \frac{1}{2}-\frac{(d-d_0)/2}{n}
\end{align}
with the additional condition $q\le 2$ in the marginal case
$s=(d-d_0)/2$.
If $S_{d_0}^d(F^{s,p}_q)$ is non-empty then it contains $(s+d_0,p,q)$, $(d-s,p^*,q^*)$,
and $((d+d_0)/2,2,2)$.  Moreover, if $(\sigma,a,b)\in \mathcal S_{d_0}^d(F^{s,p}_q)$, then
we have the continuous inclusions of Fr\'echet spaces
\begin{equation}\label{eq:S-include-Fsp}
F^{s+d_0,p}_{q,\rm loc}(\Reals^n) \subseteq F^{\sigma,a}_{b,\rm loc}(\Reals^n)
\subseteq F^{d-s,p^*}_{q^*,\rm loc}(\Reals^n).
\end{equation}
\end{lemma}
Just as in the Bessel potential case, if an operator
of class $\mathcal L_{d_0}^d(F^{s,p}_q;\Omega)$ is compatible
with any indices at all, it acts on a compatible $L^2$-based 
\textit{Bessel potential} space, $F^{(d-d_0)/2, 2}_{2}(\Omega)=H^{(d-d_0)/2, 2}(\Omega)$.

The following commutator result generalizes Lemma \ref{lem:commutator-Hsp}.
\begin{lemma} \label{lem:commutator-Fsp}
Suppose $1<p,q,a,b<\infty$, $s>n/p$, $\sigma\in\Reals$ and 
$d,d_0\in\Ints_{\ge 0}$ with $d\ge d_0$. Let
$L$ be an operator of class $\mathcal L_{d_0}^d(F^{s,p}_q;\Omega)$
and let $\phi\in\mathcal D(\Omega)$.  
Then $[L,\phi]$ 
extends from a map 
$C^\infty(\overline \Omega)\mapsto \mathcal D'(\Omega)$ to 
a continuous linear map 
$F^{\sigma,a}_b(\Omega)\mapsto F^{\sigma-d+1,a}_b(\Omega)$
so long as  $(\sigma + 1,a,b) \in \mathcal S^{d}_{d_0}(F^{s,p}_q)$. Moreover,
if $d_0=0$, the same result holds if $(\sigma,a,b)\in \mathcal S^{d}_{0}(H^{s,p}_q)$.
\end{lemma}
\begin{proof}
The proof in the case of general $d_0$ is a computation that follows
the analogous part of the proof of Lemma \ref{lem:commutator-Hsp}.  The 
only difference is that there are fine parameter restrictions that need 
to be checked. In the notation of Theorem \ref{thm:mult-Fsp},
nontrivial restrictions could happen when $s_i=s$, when $s_1+s_2=0$ and when
$s_1=s_2=s=0$.  One readily verifies that under the given hypotheses
that $s_1=s_2=s=0$ cannot happen and that in the remaining cases 
the fine parameter restrictions are always met.

If $d_0=0$ the proof again follows the strategy of Lemma \ref{lem:commutator-Hsp}.
First we observe that $[L,\phi]=[\hat L,\phi]$ where $\hat L$ is $L$ with its
constant term eliminated.  Using the case $d_0=1$ already proved we have continuity 
so long as $(\sigma+1,a,b)\in \mathcal S^d_1(F^{s,p}_q)$, and computation shows that
if $(\sigma,a,b)\in \mathcal S^{d}_0(F^{s,p}_q)$ then $(\sigma+1,a,b)\in \mathcal S^d_1(F^{s,p}_q)$.
\end{proof}

\subsection{Rescaling estimates}\label{secsec:TL-rescaling}

In this section we show that the rescaling estimates of Proposition \ref{prop:rescale-int}
carry over to Triebel-Lizorkin spaces.

\begin{proposition}\label{prop:rescale-Fsp}
Suppose $1<p,q<\infty$,
	$s\in\Reals$
and that $\chi$ is a Schwartz function on $\Reals^n$.
There exists a constant $\alpha\in\Reals$ such that
for all $0<r\le 1$ and all $u\in F^{s,p}_q(\Reals^n)$
\begin{equation}\label{eq:Fsp-scale}
||\chi u_{\{r\}}||_{F^{s,p}_q(\Reals^n)} \lesssim r^{\alpha}||u||_{F^{s,p}_q(\Reals^n)}.
\end{equation}
Specifically:
\begin{enumerate}
	\item\label{part:rescale-Fsp-generic}
Inequality \eqref{eq:Fsp-scale} holds with
\[
\alpha = \min\left(s-\frac{n}{p},0\right)
\]
unless $s-n/p=0$, in which case it holds for any choice
of $\alpha<0$, with implicit constant depending on 
$\alpha$.
	\item\label{part:rescale-Fsp-zero-center}
	If $s>n/p$ (in which case functions in $F^{s,p}_q(\Reals^n)$
	are H\"older continuous) and if $u\in F^{s,p}_q(\Reals^n)$
	with $u(0)=0$, then inequality holds with
	\[
\alpha = \min\left(s-\frac{n}{p},1\right)
	\]
	unless $s-n/p=1$, in which case it holds for any choice of
	$\alpha<1$, with implicit constant depending on $\alpha$.
\end{enumerate}
\end{proposition}

The remainder of this section is an extended proof of this result, and relies
on the following elementary facts from Littlewood-Paley theory.

\begin{proposition}\label{prop:littlewood-payley-review}
Let $1<p,q,t<\infty$.
\begin{enumerate}
	\item\label{part:schwartz-project} If $\eta$ is a Schwartz function then for any $M>0$,
	\[
		|(P_k\eta)(x)|\lesssim \frac{2^{-Mk}}{(1+|x|^M)}
	\]
	with implicit constant independent of $x$ and $k$ but depending on $\eta$ and $M$.

	\item\label{part:projection-as-convolution}
	There exist Schwartz functions $\kappa,\kappa^*$
	such that for any tempered distribution $u$,
	\begin{align*}
	(P_k u)(x) &= \int u(x + 2^{-k}y) \kappa(y)\; dy\\
	(P_{\le k} u)(x) &= \int u(x + 2^{-k}y) \kappa^*(y)\; dy,
	\end{align*}
	with convolution and scaling interpreted in the distributional sense.

	\item \label{part:maximal-function-beats-projection}
For all $u\in L^p(\Reals^n)$, $k\in\Nats$, and $x\in\Reals^n$
\[
\begin{aligned}
	|P_{k} u (x)| &\lesssim |(M u)(x)|\quad\text{and}\\
	|P_{\le k} u (x)| &\lesssim |(M u)(x)|.
\end{aligned}
\]
Here $M$ is the Hardy-Littlewood maximal operator 
and the implicit constants are independent of $u$, $k$ and $x$.

	\item\label{part:throw-away-Pk}
For any sequence $\{f_k\}_{k\in\Nats}$ of functions in $L^p(\Reals^n)$,
\[
\left|\left| \,\, \left[\sum_{k\in\Nats} |P_k f_k|^q\right]^{1/q}\,\right|\right|_{L^p(\Reals^n)} \lesssim 
\left|\left| \,\, \left[\,\,\sum_{k\in\Nats} |f_k|^q\right]^{1/q}\,\right|\right|_{L^p(\Reals^n)}.
\]
\item \label{part:sobolev-space-by-projection}
Let $s>0$ and $a\in\Nats$. For all $u\in F^{s,p}_q(\Reals^n)$,
\[
||u||_{F^{s,p}_q(\Reals^n)} \sim ||u||_{L^p(\Reals^n)} + 
\bignorm||\,\, \left[\sum_{k\ge a} |2^{sk}P_k u|^q\right]^{1/q}\,||_{L^p(\Reals^n)}.
\]
\item (Littlewood-Paley Trichotomy) \label{part:trichotomy}
Suppose $u\in L^p(\Reals^n)$ and $v\in L^t(\Reals^n)$.
Given $k,k',k''\in \Ints$,
\[
P_k( (P_{k'} u) (P_{k''}v)) = 0
\]
unless one of the following three conditions holds:
\begin{itemize}
	\item $k'\le k-4$ and $k-3\le k''\le k+3$,
	\item $k-3\le k'\le k+3$ and $k''\le k+5$,
	\item $k'\ge k+4$ and $|k'-k''|\le 2$.
\end{itemize}
\end{enumerate}
\end{proposition}

Part \eqref{part:schwartz-project} is a consequence of the definition of the projection
operator and elementary properties of the Fourier transform.
The proof of parts 
\eqref{part:projection-as-convolution}
and
\eqref{part:maximal-function-beats-projection}
can be found in the approachable lecture
notes \cite{TaoPL}, weeks 2/3.
Part \eqref{part:throw-away-Pk} in full generality follows from 
the Fefferman-Stein inequality \cite{fefferman_maximal_1971} 
together with part \eqref{part:maximal-function-beats-projection}.
We note, however, that the cases of greatest interest (Bessel potential
spaces and Sobolev-Slobodeckij spaces) only involve the cases $q=2$
and $q=p$, which do not require the full power of \cite{fefferman_maximal_1971}. 
The lecture notes \cite{TaoPL} contain a proof of part \eqref{part:throw-away-Pk}
when $q=2$, and the result when $q=p$ is an easy consequence of part \eqref{part:maximal-function-beats-projection}
and the Hardy-Littlewood maximal inequality.
Part \eqref{part:sobolev-space-by-projection} 
follows from the definition of the norm, 
part \eqref{part:projection-as-convolution} and
embedding $F^{s,p}_q(\Reals^n)\hookrightarrow F^{0,p}_2(\Reals^n)=L^p(\Reals^n)$.
Part \eqref{part:trichotomy} is just a
computation based on the supports of convolutions of 
functions $\psi_k$ used to define Littlewood-Paley projection; see \cite{TaoPL} for a related statement.
Note that there is an artificial 
asymmetry in part \eqref{part:trichotomy} between $k'$ and $k''$, and
symmetry can be restored at the expense of increasing the number
of cases.

As for Bessel potential spaces, rescaling $u\mapsto u_\rscr$ is a continuous
automorphism of any space $F^{s,p}_{q}(\Reals^n)$; for $s>0$ this
is a straightforward consequence of the Closed Graph Theorem, using the fact
that rescaling is continuous acting on $L^p(\Reals^n)$, whereas for $s<0$ and $s=0$
the result follows from duality and interpolation respectively.  We initially
require the following basic estimate for the norms of the rescaling maps.
\begin{lemma}\label{lem:basic-rescale}
Let $R>0$, and suppose $1<p,q<\infty$ and $s\in\Reals$.  Then
\[
||u_\rscr||_{F^{s,p}_q(\Reals^n)} \lesssim ||u||_{F^{s,p}_q(\Reals^n)}
\]
for all $u\in F^{s,p}_q(\Reals^n)$ and all $r\in [-R,R]$.
\end{lemma}
\begin{proof}
Recall the cutoff functions $\phi$ and $\psi$ used to define the Littlewood-Paley 
projection operators and define
\[
\psi_k(\xi) = \begin{cases} 0 & k<0\\
	\phi(\xi) & k=0\\
	\psi(2^{-k}\xi) & k>0
\end{cases}
\]
A computation shows that for $k\ge 0$, 
\[
(P_k u_\rscr)(x) = f_k(rx)
\]
where $f_k = \mathcal F^{-1} M_k \mathcal F u$ and where
\[
M_k(\xi) = \phi_k(r\xi).
\]
From the bound $|r|\le R$ we can find $J$ independent of $r$ such that
\[
M_k = M_k \sum_{j=-J}^J \psi_{k+j}
\]
with the convention that $\psi_k=0$ for $k<0$.  Hence
\[
f_k = \sum_{j=-J}^J (\mathcal F^{-1} M_k \mathcal F)[P_{k+j} u ].
\]
Using \cite{triebel_theory_2010}
Theorem 1.6.3 via the same argument as in \cite{triebel_theory_2010} Proposition 
2.3.2/1 we find
\[
\int \left(\sum |f_k|^q\right)^{p/q} \lesssim \int \left(\sum |P_k u_\rscr|^q\right)^{p/q} = ||u||^p_{F^{s,p}_q(\Reals^n)}
\]
with implicit constant independent of $r\in [-R,R]$.  Recalling that 
$(P_k u_\rscr)(x) = f_k(rx)$ the result now follows from the obvious
uniform bounds on rescaling in $L^p(\Reals^n)$ for $|r|\le R$.
\end{proof}

We now proceed with the proof of 
Proposition \ref{prop:rescale-Fsp} part \eqref{part:rescale-Fsp-generic}
is broken into three cases depending on whether $s>0$, $s<0$, or $s=0$. 
We begin with $s>0$ and first establish the
following technical lemma, which is needed to control high frequency rescaling.

\begin{lemma}\label{lem:rescale-Fsp-s-pos-baby}
Suppose $s> 0$. For all $0<r\le 1$ 
and all $u\in F^{s,p}_{q}(\Reals^n)$,
\[
\left|\left| \left[\sum_{k\ge 1}  |2^{ks}P_k u_{\{r\}}|^q\right]^{1/q} \right|\right|_{L^p(\Reals^n)}
\lesssim r^{s-\frac{n}{p}} ||u||_{F^{s,p}_q(\Reals^n)}.
\]
\end{lemma}
\begin{proof}
Suppose first that $r=2^{-j}$ for some $j\in\Ints_{\ge 0}$.
An easy computation from the definition of the Fourier transform
and change of variables shows for each $k\in \Nats$,
\[
(P_k u_{\{r\}}) = (P_{k+j} u)_{\{r\}}
\]
for all Schwartz functions $u$, and hence also for all 
$u\in F^{s,p}_q(\Reals^n)$.  But then
\begin{align*}
\left|\left|\left[\sum_{k\ge 1} |2^{ks}P_k u_{\{r\}}|^q\right]^{1/q} \right|\right|_{L^p(\Reals^n)}
&= 2^{-js}\left|\left|\left[\sum_{k\ge 1+j} |2^{ks}(P_{k} u)|^q\right]^{1/q} \right|\right|_{L^p(\Reals^n)}\\
&= 2^{-js}2^{-n/p}
\left|\left|\left[\sum_{k\ge 1+j} |2^{ks} P_{k} u|^q\right]^{1/q} \right|\right|_{L^p(\Reals^n)}\\
&\lesssim 
2^{-js}r^{-n/p} ||u||_{F^{s,p}_q(\Reals^n)}\\
& = r^{s-\frac n p} ||u||_{F^{s,p}_q(\Reals^n)}.
\end{align*}
This completes the proof in the case $r=2^{-j}$. The
general case follows from the consequence of Lemma \ref{lem:basic-rescale} that
that $u\mapsto u_{\{r\}}$
is uniformly bounded in $F^{s,p}_q(\Reals^n)$ for $1/2\le r\le 2$.
\end{proof}

Proposition \ref{prop:rescale-Fsp} part \eqref{part:rescale-Fsp-generic} when $s>0$ now follows
from the following.

\begin{proposition}\label{prop:rescale-Fsp-s-pos-generic}
Suppose $1<p,q<\infty$ and $s>0$,
and let $\chi$ be a Schwartz function.
For all $u\in F^{s,p}_q(\Reals^n)$ and $0<r\le1$,
\begin{equation}\label{eq:s_pos_scale-c}
||\chi u_{\{r\}}||_{F^{s,p}_q(\Reals^n)} \lesssim r^{\alpha}  ||u||_{F^{s,p}_q(\Reals^n)}
\end{equation}
where
\begin{equation}\label{eq:alpha_s_pos}
\alpha=\min\left( s - \frac{n}{p}, 0\right)
\end{equation}
unless $s=n/p$, in which case $\alpha$ can be any
(fixed) negative number.  The implicit constant
in \eqref{eq:s_pos_scale-c} depends on $\chi$
is independent of $u$ and $r$.
\end{proposition}
\begin{proof}
We first assume that $s\neq n/p$ and define $\alpha$
according to equation \ref{eq:alpha_s_pos}.
Since $s>0$, Proposition 
\ref{prop:littlewood-payley-review} part \eqref{part:sobolev-space-by-projection} implies
\begin{equation}\label{eq:chi_ur_to_bound-c}
||\chi u_{\{r\}}||_{F^{s,p}_q(\Reals^n)} \lesssim  || \chi u_{\{r\}} ||_{L^p(\Reals^n)}
+ \bignorm|| \left[\sum_{k\ge 10}  |P_k (\chi u_{\{r\}})|^q\right]^{\frac{1}{q}}||_{L^p(\Reals^n)}.
\end{equation}
To estimate the first term on the right-hand side of inequality
\eqref{eq:chi_ur_to_bound-c}
first consider the case $s<n/p$. Define $t$ by
\begin{equation}\label{eq:t-def-scale}
\frac{1}{t} = \frac{1}{p}-\frac{s}{n}
\end{equation}
and observe that since $1/p<1$ and since $s<n/p$, we have $0<1/t<1$.
Proposition \ref{prop:embedding-Fsp} implies
$F^{s,p}_q(\Reals^n)$ embeds in 
$F^{0,t}_2(\Reals^n) = L^t(\Reals^n)$.
From H\"older's inequality we find
\[
||\chi u_{\{r\}}||_{L^p(\Reals^n)} \le ||\chi ||_{L^\tau(\Reals^n)} ||u_{\{r\}}||_{L^t(\Reals^n)} 
\]
where $\frac{1}{\tau} = \frac1p - \frac{1}{t}$.  Hence, from Lemma 
\ref{lem:Lpscale}
\[
||u_{\{r\}}||_{L^t(\Reals^n)} = r^{-\frac{n}{t}} ||u||_{L^t(\Reals^n)} = r^{s-\frac{n}{p}} ||u||_{L^t(\Reals^n)} \lesssim
r^{\alpha} ||u||_{F^{s,p}_q(\Reals^n)}
\]
and we conclude $||\chi u_{\{r\}}||_{L^p(\Reals^n)}\lesssim r^\alpha ||u||_{F^{s,p}_q(\Reals^n)}$. 
On the other hand, if $s>n/p$ then
\[
||\chi u_{\{r\}}||_{L^p(\Reals^n)} \lesssim ||u_{\{r\}}||_{L^\infty(\Reals^n)}||\chi||_{L^p(\Reals^n)} \lesssim ||u||_{F^{s,p}_q(\Reals^n)}
= r^\alpha ||u||_{F^{s,p}_q(\Reals^n)}.
\]
Hence in both cases, $||\chi u_{\{r\}}||_{L^p(\Reals^n)} \lesssim r^\alpha ||u||_{F^{s,p}_q(\Reals^n)}$.

Turning to the second term on the right-hand side of 
inequality \eqref{eq:chi_ur_to_bound-c} we introduce the notation
\[
	\widetilde P_k = P_{k-3\le\cdot\le k+3}.
\]
The Littlewood-Paley trichotomy, Proposition \ref{prop:littlewood-payley-review} part \eqref{part:trichotomy},
implies
\begin{equation}\label{eq:tricotomy}
P_k (\chi u_{\{r\}}) = 
P_k \left(
\underbrace{ (\widetilde P_k \chi) (P_{\le k-4} u_{\{r\}})}_{\text{high-low}}
+
\underbrace{ (P_{\le k+5} \chi) \widetilde P_k u_{\{r\}})}_{\text{low-high}}
+ 
\underbrace{\sum_{k' \ge k+4 } (\widetilde P_{k'} \chi)
(P_{k'} u_{\{r\}})}_{\text{high-high}}\right)
\end{equation}
and we estimate the contributions from each of these three terms individually via the triangle inequality.

Starting with the high-low term, Proposition \ref{prop:littlewood-payley-review} parts 
\eqref{part:throw-away-Pk} and
\eqref{part:maximal-function-beats-projection} imply
\[
\begin{aligned}
\bignorm|| \left[\sum_{k\ge 10}  |2^{sk} P_k\left\{ 
\left(\widetilde P_k \chi\right)\left( P_{\le k-4} u_{\{r\}}\right)\right\}|^q\right]^{\frac 1 q} ||_{L^p(\Reals^n)}
&\lesssim
\bignorm|| \left[\sum_{k\ge 10}  |2^{sk}  
\left(\widetilde P_k \chi\right)\left( P_{\le k-4} u_{\{r\}}\right)|^q
\right]^{\frac 1 q} ||_{L^p(\Reals^n)}\\
&\lesssim
\bignorm|| |Mu_{\{r\}}| \left[\sum_{k\ge 10}  |2^{sk} 
\widetilde P_k \chi|^q
\right]^{\frac 1 q} ||_{L^p(\Reals^n)}
\end{aligned}
\]
where $M$ is the Hardy-Littlewood maximal operator.
Now suppose $s< n/p$ and pick $t$ according 
to equation \eqref{eq:t-def-scale}.  Then
H\"older's inequality and the 
triangle inequality imply
\[
\bignorm|| |Mu_{\{r\}}| \left[\sum_{k\ge 10}  |2^{sk} 
\widetilde P_{k} \chi|^q
\right]^{\frac 1 q} ||_{L^p(\Reals^n)}
\lesssim ||M u_{\{r\}}||_{L^t(\Reals^n)} ||\chi||_{F^{s,\tau}_q(\Reals^n)}
\]
where again $1/\tau=1/p-1/t$.
The Hardy-Littlewood maximal inequality and Lemma \ref{lem:Lpscale} then imply
$||Mu_{\{r\}}||_{L^t(\Reals^n)}\lesssim||u_{\{r\}}||_{L^t(\Reals^n)} \lesssim r^{-\frac{n}t}||u||_{L^t(\Reals^n)} = r^\alpha ||u||_{F^{s,p}_q(\Reals^n)}$, which yields the desired estimate
\[
\bignorm|| \left[\sum_{k\ge 10} |2^{ks} P_k \left\{
\left(\widetilde P_{k} \chi\right)\left( P_{k-4} u_{\{r\}}\right)\right\}|^q
\right]^{\frac 1 q} ||_{L^p(\Reals^n)}
\lesssim r^\alpha ||u||_{F^{s,p}_q(\Reals^n)}.
\]
Obtaining this same inequality in the case
$s>n/p$ is similar but easier, using the estimate
\[
|M u_{\{r\}}| \lesssim ||u||_{L^\infty(\Reals^n)} \lesssim ||u||_{F^{s,p}_q(\Reals^n)} = r^\alpha ||u||_{F^{s,p}_q(\Reals^n)}
\]
along with the fact $\chi \in F^{s,p}_q(\Reals^n)$.

To estimate the low-high term we
estimate $|P_{\le k+5} \chi|\lesssim ||\chi||_{L^\infty(\Reals^n)}$
and use Proposition \ref{prop:littlewood-payley-review}
part \eqref{part:throw-away-Pk} and Lemma \ref{lem:rescale-Fsp-s-pos-baby}
to conclude
\[
\begin{aligned}
\bignorm|| 
\left[\sum_{k\ge 10}  |P_k\left\{ 2^{ks} 
\left(P_{\le k+5} \chi\right)\left( \widetilde P_{k} u_{\{r\}}\right)\right\}  |^q
\right]^{\frac 1 q} ||_{L^p(\Reals^n)}
&\lesssim
\bignorm|| 
\left[\sum_{k\ge 10} |2^{ks} 
\widetilde P_{k} u_{\{r\}}  |^q
\right]^{\frac 1 q} ||_{L^p(\Reals^n)}\\
&\lesssim
\bignorm|| 
\left[\sum_{k\ge 7} |2^{ks} P_{k} u_{\{r\}}  |^q
\right]^{\frac 1 q} ||_{L^p(\Reals^n)}\\
&\lesssim r^{s-\frac n p} ||u||_{F^{s,p}_q(\Reals^n)}\\
&\lesssim r^{\alpha} ||u||_{F^{s,p}_q(\Reals^n)}.
\end{aligned}
\]

Finally, for the high-high term we 
observe from Proposition \ref{prop:littlewood-payley-review} part \eqref{part:projection-as-convolution}
that $\widetilde P_j\chi \lesssim 2^{-j}$ is uniformly bounded in $L^\infty(\Reals^n)$
independent of $j$ and hence 
Proposition \ref{prop:littlewood-payley-review} part
\eqref{part:throw-away-Pk} and
Lemma \ref{lem:rescale-Fsp-s-pos-baby} along with the triangle inequality
imply
\begin{align*}
\bignorm|| \left[ \sum_{k\ge 10}\left|2^{ks}P_{k} \left\{
\sum_{k' \ge k+4} (P_{k'}u_{\{r\}})(\widetilde P_{k'}\chi)
\right\}
\right|^q\right]^{\frac 1q}||_{L^p(\Reals^n)} 
&\lesssim
\bignorm|| \left[ \sum_{k\ge 10}\left|2^{ks}
\sum_{k' \ge k+4} (P_{k'}u_{\{r\}})(\widetilde P_{k'}\chi)
\right|^q\right]^{\frac 1q}||_{L^p(\Reals^n)} \\
&\lesssim
\sum_{a\ge 4} 2^{-as} \bignorm||\left[ \sum_{k\ge 10}\left|2^{(k+a)s}P_{k+a}u_{\{r\}}\right|^q\right]^{\frac 1q} ||_{L^p(\Reals^n)}\\
&\lesssim 
\sum_{a\ge 4} 2^{-as} \bignorm||\left[\sum_{k\ge 1}|2^{ks}P_k u_\rscr|^q
\right]^{\frac 1q}  ||_{L^p(\Reals^n)}\\
&\lesssim r^{s-\frac np}||u||_{F^{s,p}_q(\Reals^n)}.
\end{align*}
This concludes the proof unless $s= n/p$, in which 
case the result follows from interpolation.
\end{proof}

The proof of Proposition \ref{prop:rescale-Fsp} part \eqref{part:rescale-Fsp-generic} when $s<0$
uses a duality argument analogous to that used for the same step in Section \ref{secsec:rescale-Hsp}.
First, the following generalization of Lemma \ref{lem:scale-in}
follows from  \cite{triebel_theory_2010} Proposition 3.4.1/1, which is
proved similarly to Lemma \ref{lem:rescale-Fsp-s-pos-baby}.
\begin{lemma}\label{lem:scale-in-Fsp}
Suppose $1<p,q<\infty$, 
and $s>0$.  
For all $r\ge 1$ and all $u\in F^{s,p}_q(\Reals^n)$,
\begin{equation}
||u_{\{r\}}||_{F^{s,p}_q(\Reals^n)} \lesssim ||u||_{F^{s,p}_q(\Reals^n)} 
r^{s -\frac{n}{p}}.
\end{equation}
\end{lemma}
The following corollary, which completes the case of part \eqref{part:rescale-Fsp-generic},
is proved identically to Corollary \ref{cor:scale-sigma-neg} using
the duality property of Proposition \ref{prop:dual-Fsp}.
\begin{corollary}\label{cor:scale-s-neg-Fsp}
Suppose $1<p,q<\infty$ and 
$s<0$. For all $0<r\le1$ and all  $u\in F^{s,p}_q(\Reals^n)$
\[
||u_{\{r\}}||_{F^{s,p}_q(\Reals^n)} \lesssim  r^{\sigma-\frac{n}{p}}||u||_{F^{s,p}_q(\Reals^n)}.
\]
\end{corollary}
Proposition \ref{prop:rescale-Fsp} part \eqref{part:rescale-Fsp-generic}
has now been established except for the case $s=0$,
which follows from the following easy interpolation argument.
\begin{lemma}\label{lem:rescale-Fsp-s-zero} 
Suppose $1<p,q<\infty$.  For all $0<r\le 1$
and all $u\in F^{0,p}_q(\Reals^n)$,
\[
||u_{\{r\}}||_{F^{0,p}_q(\Reals^n)} \lesssim r^{-\frac np} ||u||_{F^{0,p}_q(\Reals^n)}.
\]
\end{lemma}
\begin{proof}
Pick $\sigma\in\Reals$ such that $0<\sigma<n/p$.  Then
for all Schwartz functions $u$, Corollary \ref{cor:scale-s-neg-Fsp}
and Proposition \ref{prop:rescale-Fsp-s-pos-generic} imply
\begin{align*}
||u_{\{r\}}||_{F^{\sigma,p}_q(\Reals^n)} &\lesssim r^{\sigma-\frac{n}{p}}
||u||_{F^{\sigma,p}_q(\Reals^n)}\\
||u_{\{r\}}||_{F^{-\sigma,p}_q(\Reals^n)} &\lesssim r^{-\sigma-\frac{n}{p}}
||u||_{F^{\sigma,p}_q(\Reals^n)}.
\end{align*}
for all $0<r\le 1$. The result follows from interpolation.
\end{proof}

With the proof of Proposition \ref{prop:rescale-Fsp}
part \eqref{part:rescale-Fsp-generic} now complete we turn to
part \eqref{part:rescale-Fsp-zero-center}, the improved estimate
when $u(0)=0$.  The following estimate is the key to controlling
low-frequency interactions near $x=0$.
\begin{lemma}\label{lem:scale-cut-Holder}
Suppose $u\in C^{0,\alpha}(\Reals^n)$ for some $\alpha\in (0,1]$
and that $u(0)=0$.  For all $k\in\Nats$ and all $0<r\le1$
\[
|P_{\le k} u_{\{r\}}(x)| \lesssim ||u||_{C^{0,\alpha}(\Reals^n)} \min( r^\alpha (|x|^\alpha + 1),1).
\]
\end{lemma}
\begin{proof}
Proposition \ref{prop:littlewood-payley-review} part \eqref{part:projection-as-convolution} implies there
is a Schwartz function $\kappa^*$ such that 
\[
P_{\le k} u_{\{r\}}(x) = \int u(rx + r2^{-k}y) \kappa^*(y)\; dy.
\]
We therefore have the easy estimate $||P_{\le k} u_{\{r\}}||_{L^\infty(\Reals^n)}\lesssim ||u||_{C^{0,\alpha}(\Reals^n)}$
and it suffices to show
\[
|P_{\le k} u_{\{r\}}(x)| \lesssim ||u||_{C^{0,\alpha}(\Reals^n)}r^\alpha  (|x|^\alpha + 1)
\]
as well.

Since $u(0)=0$,
on each annulus $A_j = \{y:2^{j}\le |y| \le 2^{j+1}\}$ for $j\in\Ints_{\ge 0}$ we find
\[
|u_{\{r\}}(x+2^{-k}y)| \le ||u||_{C^{0,\alpha}(\Reals^n)} r^\alpha|x+2^{-k}y|^\alpha 
\lesssim
r^\alpha ||u||_{C^{0,\alpha}(\Reals^n)} \left(|x|^\alpha+2^{\alpha j}\right)
\]
where the implicit constant is independent of $j$.
Hence
\[
\begin{aligned}
\left|\int_{A_j} u(rx + r2^{-k}y) \kappa^*(y)\; dy \right|
&\lesssim 
r^\alpha ||u||_{C^{0,\alpha}(\Reals^n)}  \left(|x|^\alpha+2^{\alpha j}\right)
\int_{A_j} |\kappa^*(y)|\; dy.
\end{aligned}
\]
Since $\kappa^*$ is a Schwartz function,
$\int_{A_j} |\kappa^*(y)|\; dy \lesssim 2^{-(1+\alpha)j}$
and hence
\begin{equation}\label{eq:annulus-bound}
\left|\int_{A_j} u(rx + r2^{-k}y) \kappa^*(y)\; dy \right|
\lesssim r^\alpha ||u||_{C^{0,\alpha}(\Reals^n)}  (|x|^\alpha + 1)2^{-j}.
\end{equation}
On the other hand,
\begin{equation}\label{eq:ball-bound}
\begin{aligned}
\left|\int_{B_1(0)} u(rx+r2^{-k}y)\kappa^*(y)\; dy\right| &\le
||u||_{C^{0,\alpha}(\Reals^n)} r^\alpha (|x|+1)^\alpha \int_{B_1} |\kappa^*(y)|\; dy\\
&\lesssim ||u||_{C^{0,\alpha}(\Reals^n)} r^\alpha (|x|^\alpha +1).
\end{aligned}
\end{equation}
The result follows from adding the contributions in inequalities
\eqref{eq:annulus-bound} and \eqref{eq:ball-bound}.
\end{proof}

Part \eqref{part:rescale-Fsp-zero-center} of Proposition \ref{prop:rescale-Fsp} 
now follows from the following.

\begin{proposition}\label{prop:rescale-Fsp-u-zero}
Suppose $1<p,q<\infty$ and $s>n/p$,
and let $\chi$ be a Schwartz function.
For all $u\in F^{s,p}_q(\Reals^n)$ with $u(0)=0$, and for
all $0<r\le 1$,
\begin{equation}\label{eq:s_pos_scale-b}
||\chi u_{\{r\}}||_{F^{s,p}_q(\Reals^n)} \lesssim r^{\alpha}  ||u||_{F^{s,p}_q(\Reals^n)}
\end{equation}
where
\begin{equation}\label{eq:alpha_s_pos_zero}
\alpha=\min\left( s - \frac{n}{p}, 1\right)
\end{equation}
unless $s=n/p+1$, in which case $\alpha$ can be any
(fixed) number less than 1.  The implicit constant
in \eqref{eq:s_pos_scale-b} depends on $\chi$
but is independent of $u$ and $r$.
\end{proposition}
\begin{proof}
We start by assuming $s\neq \frac{n}{p}+1$ and define 
$\alpha$ according to equation \eqref{eq:alpha_s_pos_zero}.
Hence $F^{s,p}_q(\Reals^n)$ embeds in $C^{0,\alpha}(\Reals^n)$.

As in the proof of Proposition \ref{prop:rescale-Fsp-s-pos-generic} 
we start from the estimate 
\begin{equation}\label{eq:chi_ur_to_bound-b}
||\chi u_{\{r\}}||_{F^{s,p}_q(\Reals^n)} \lesssim  || \chi u_{\{r\}} ||_{L^p(\Reals^n)} 
+ \bignorm|| \left[\sum_{k\ge 10} \left| 2^{ks} P_k (\chi u_{\{r\}})\right|^q\right]^{\frac{1}{q}}||_{L^p(\Reals^n)}.
\end{equation}
The second term on the right-hand side is again estimated
using the Littlewood-Paley trichotomy (equation \eqref{eq:tricotomy})
and the proof 
of Proposition \ref{prop:rescale-Fsp-s-pos-generic} 
shows that the low-high and high-high
terms of that decomposition satisfy a bound of the form
$r^{s-\frac{n}p}||u||_{F^{s,p}_q(\Reals^n)}$, regardless of the value of $u$
at zero.  Consequently, we need only estimate the 
effects of the low-frequency contributions from $u_{\{r\}}$, 
namely the high-low term, as well as $|| \chi u_{\{r\}} ||_{L^p(\Reals^n)}$.

For the low-high term, let $v_{r}(x) =  \min( r^\alpha(|x|^\alpha + 1),1)$ and
hence
Lemma \ref{lem:scale-cut-Holder} implies $|u_{\{r\}}|\lesssim ||u||_{C^{0,\alpha}(\Reals^n)} v_{r}$
for all $0<r\le 1$. Recalling the notation $\widetilde P_k = P_{k-3\le\cdot\le k+3}$,
Proposition \ref{prop:littlewood-payley-review} part \ref{part:throw-away-Pk} and Lemma \ref{lem:scale-cut-Holder} imply
\begin{equation}\label{eq:low-hi-zero}
\begin{aligned}
\bignorm|| \left[\sum_{k\ge 10}  \left| 2^{ks}P_k
\left\{ \left( \widetilde P_{k} \chi\right)\left(P_{\le k-4} u_{\{r\}}\right)\right\}\right|^q\right]^{\frac 1q}||_{L^p(\Reals^n)}
&\lesssim
\bignorm|| \left[\sum_{k\ge 10}  \left| 2^{ks}
\left( \widetilde P_{k} \chi\right)\left(P_{\le k-4} u_{\{r\}}\right)\right|^q\right]^{\frac 1q}||_{L^p(\Reals^n)}\\
&\lesssim 
||u||_{C^{0,\alpha}(\Reals^n)} ||v_{r} \eta||_{L^p(\Reals^n)}
\end{aligned}
\end{equation}
where $\eta=\left[\sum_{k\ge 10} | 2^{ks} 
\widetilde P_{k} \chi|^q\right]^{\frac 1q}$.
Since $\chi$ is a Schwartz function, Proposition \ref{prop:littlewood-payley-review} part \eqref{part:schwartz-project}
implies that given $M>0$ we can estimate
\[
|(P_j \chi) (x)| \lesssim \frac{2^{-Mj}}{(1+|x|)^M}
\]
with implicit constant independent of $j$.
As a consequence, $|\eta(x)|\lesssim 1/(1+|x|)^M$; i.e.,
$\eta$ is rapidly decreasing.

To estimate $|| v_{r} \eta ||_{L^p(\Reals^n)}$
we divide $\Reals^n$ into three regions: the ball $B_1(0)$, the annulus 
$A=B_{1/r}(0)\setminus B_1(0)$ and the exterior region $E=B_{1/r}(0)^c$.
On the unit ball $|v_{r}|\lesssim r^\alpha $ and hence
\begin{equation}\label{eq:est_B1}
||v_{r} \eta||_{L^p(B_1(0))} \lesssim r^\alpha.
\end{equation}
Outside the unit ball,
$|v_{r}(x)|\lesssim r^\alpha (|x|^\alpha+1)
\lesssim r^\alpha |x|^\alpha$ and  
$|\eta(x)|\lesssim |x|^{-(n+1)/p-\alpha}$
So
\begin{equation}\label{eq:est_A}
\begin{aligned}
\int_{A} \eta^p v_{r}^p &\lesssim  \int_1^{1/r} s^{-n-1-\alpha p} (rs)^{\alpha p} s^{n-1}\; ds
\le r^{p\alpha} \int_{1}^{1/r} s^{-2}\; ds \lesssim r^{p\alpha}
\end{aligned}
\end{equation}
Finally, for the exterior region we estimate $|v_{r}|\le 2$ and
find
\begin{equation}\label{eq:est_E}
\int_{E} \eta^p|v_{r}|^p \lesssim  
\int_{1/r}^\infty s^{-n-1-\alpha p } s^{n-1} \; ds 
\lesssim r^{1 +\alpha p}
\lesssim r^{\alpha p}.
\end{equation}
Combining inequalities \eqref{eq:est_B1}, \eqref{eq:est_A} and
\eqref{eq:est_E} we conclude
$
||v_{r} \eta||_{L^p(\Reals^n)} \lesssim r^\alpha
$
which, when combined with inequality \eqref{eq:low-hi-zero},
completes the estimate for the low-high term.

It remains to show that $||\chi u_\rscr||_{L^p(\Reals^n)}\lesssim ||u||_{F^{s,p}_q(\Reals^n)}r^\alpha$.
The argument
that showed $||v_{r} \eta||_{L^p(\Reals^n)} \lesssim r^\alpha$ only
used the fact that $\eta$ was rapidly decreasing, and hence
we also find $||v_{r}\chi||_{L^p(\Reals^n)} \lesssim r^\alpha$.
Again using the estimate
$|u_{\{r\}}|\lesssim ||u||_{C^{0,\alpha}(\Reals^n)} v_{r}$ we obtain
\[
||\chi u_{\{r\}}||_{L^p(\Reals^n)} \lesssim ||u||_{C^{0,\alpha}(\Reals^n)}||\chi v_{r}||_{L^p(\Reals^n)} 
\lesssim ||u||_{F^{s,p}_q(\Reals^n)} r^\alpha.
\]

This concludes the proof assuming $s\neq \frac{n}{p}+1$.  For the marginal case,
let $F^{s,p}_{q,0}(\Reals^n)$ denote the closed subspace of $F^{s,p}_{q}(\Reals^n$)
consisting of functions that vanish at zero, assuming of course that $s>n/p$.
The proof when $s=n/p+1$ follows from interpolation if we can show
\begin{equation}\label{eq:interp-zero-F}
[F^{s_1,p}_{q,0}(\Reals^n),F^{s_2,p}_{q,0}(\Reals^n)]_{\theta} = F^{s,p}_{q,0}(\Reals^n)
\end{equation}
assuming that $s_i>n/p$ for $i=1,2$ and that $0<\theta<1$ and $s=s_1(1-\theta) + s_2\theta$.

Let $\phi$ be 
a compactly supported smooth function with $\phi(0)=1$. For smooth compactly supported functions
$u$ define $R u = u-u(0)\phi$.  This extends to a continuous retraction 
$F^{s,p}_{q}(\Reals^n)\to F^{s,p}_{q,0}(\Reals^n)$ so long as $s>n/p$,
with the co-retraction being the natural embedding.  The interpolation 
property \eqref{eq:interp-zero-F} now follows from, e.g, \cite{triebel_spaces_1976} Lemma 6.
\end{proof}

\subsection{Interior elliptic estimates}

Elliptic operators for operators with coefficients in Triebel-Lizorkin spaces
are defined analogously to Definition \ref{def:L-elliptic-Hsp}.  This section
contains our primary elliptic regularity result, 

which relies on the following generalization
of the rescaling estimates of Proposition \ref{prop:rescale-int}.  
The proof of these estimates is a little involved, so we record the result 
for now and defer the proof to Section \ref{secsec:TL-rescaling}.

The ``regularity at a point'' result, Proposition \ref{prop:elliptic-zoom}, admits a straightforward generalization.
Note that $\tilde F^{s,p}_q(\Omega)$ denotes the closure of $\mathcal D(\Omega)$ in $F^{s,p}_q(\Reals^n)$.
\begin{proposition}\label{prop:elliptic-zoom-Fsp} 
	Let $\Omega\subset \Reals^n$ be a bounded open set.  
	Suppose $1<p,q<\infty$, $s\in\Reals$, that $d,d_0\in\Ints_{\ge 0}$ with $d_0\le d$, that $s>n/p$,
	and that the conditions of Lemma \ref{lem:S-members-Fsp} are are satisfied and hence 
	$\mathcal S_{d_0}^d(F^{s,p}_q)\neq \emptyset$.
	Suppose additionally that 
	$L=\sum_{|\alpha|\le d} a^\alpha \partial_\alpha$ 
	is a differential operator of class
	$\mathcal L_{d_0}^d(F^{s,p}_q;\Omega)$
	and that for some $x\in \Omega$ that
	\[
	L_0 = \sum_{|\alpha|=m} a^\alpha(x)\partial_\alpha
	\]
	is elliptic.  
	Given $(\sigma,a,b)\in \mathcal S_{d_0}^d(F^{s,p}_q)$
	there exists $r>0$ such that $B_r(x)\subset \Omega$
	and such that if
	\begin{align*}
	u&\in \widetilde F^{d-s,p^*}_{q^*}( B_{r}(x) )\quad \text{and}\\
	Lu&\in F^{\sigma-d,a}_b(\Omega)
	\end{align*}
	then $u\in F^{\sigma,a}_b(\Omega)$ and
	\begin{equation}\label{eq:pointwise-est-Fsp}
	||u||_{F^{\sigma,a}_b(\Omega)} \lesssim
	||L u||_{F^{\sigma-d,a}_b(\Omega)}
	+ ||u||_{F^{d-s-1,p^*}_{q^*}(\Omega)}		
	\end{equation}
	with implicit constant independent of $u$ but depending on all other parameters.
\end{proposition}
\begin{proof}
	The proof is essentially the same as the proof of Proposition \ref{prop:elliptic-zoom}, with
	the following notes:
	\begin{enumerate}
		\item The proof of the natural generalization of  Lemma \ref{lem:parametrix} regarding the parametrix goes through 
		now using \cite{triebel_theory_2010} Theorems 2.3.7 and 2.3.8 in place of the Mikhlin multiplier
		theorem to establish the desired continuity properties.
		\item We replace the use of the rescaling result Proposition \ref{prop:poor-mans-bp}
		with Proposition \ref{prop:rescale-Fsp}, which results in a little simplification because
		the replacement result is sharper.
		\item Fine parameters need tracking, but the changes are straightforward.
		When the Lebesgue parameter is $p$ the fine parameter is $q$, when the Lebesgue
		parameter is $p^*$ the fine parameter is $q^*$ and for the intermediate spaces 
		the Lebesgue parameter is $a$ and the fine parameter is $b$.
	\end{enumerate}
\end{proof}

The proof of the main interior regularity result for Bessel
potential spaces, Theorem \ref{thm:interior-reg}, carries
over to the Triebel-Lizorkin setting. The bulk of the new work consists of tracking the fine parameter.
\begin{theorem}\label{thm:interior-reg-Fsp} 
Let $\Omega$ be a bounded open set in $\Reals^n$ and suppose $s,p,q, d_0$ and $d$ are parameters
as in Lemma \ref{lem:S-members-Fsp} such that
$s>n/p$ and such that the conditions of 
Lemma \ref{lem:S-members-Fsp}
are satisfied so $\mathcal S_{d_0}^d(F^{s,p}_q)\neq \emptyset$.  
Suppose $L$ is of class $\mathcal L_{d_0}^d(F^{s,p}_q;\Omega)$
and is elliptic on $\Omega$.  If $u\in F^{d-s,p^*,q^*}(\Omega)$
and $Lu\in F^{\sigma-d,a}_b(\Omega)$ for some 
$(\sigma,a,b)\in \mathcal S_{d_0}^d(F^{s,a}_b)$
then for any open set $U$ with $\overline U\subseteq \Omega$, $u\in F^{\sigma,a}_b(U)$ and 
\begin{equation}\label{eq:int-reg-Fsp}
||u||_{F^{\sigma,a}_b(U)} \lesssim 
||Lu||_{F^{\sigma-d,a}_b(\Omega)}
+ ||u||_{F^{d-s-1,p^*}_{q^*}(\Omega)}.
\end{equation}
\end{theorem}
\begin{proof}
The proof follows that of Theorem \ref{thm:interior-reg}
with the following changes needed to manage the fine parameter.

A bootstrap step starts with knowing $u\in F^{\sigma_A,a_A}_{b_A}(\Omega_A)$ for
some open set $\Omega_A$ containing $\overline U$ and 
we wish to improve these parameters to $(\sigma_B,a_B,b_B)$ while shrinking $\Omega_A$.  We assume:
\begin{enumerate}
	\item[\optionaldesc{H1}{cond:hard-deriv-Fsp}] $\sigma_B\le \sigma$,
	\item[\optionaldesc{H2}{cond:hard-lebesgue-Fsp}] $\displaystyle \frac{1}{q}-\frac{\sigma}{n} \le \frac{1}{q_B}-\frac{\sigma_B}{n}$,
	\item[\optionaldesc{H3}{cond:soft-deriv-Fsp}]$\sigma_B\le \sigma_A + 1$,
	\item[\optionaldesc{H4}{cond:soft-lebesgue-Fsp}] $\displaystyle \frac{1}{q_A}-\frac{\sigma_A+1}{n}\le \frac{1}{q_B}-\frac{\sigma_B}{n}$,
	\item[\optionaldesc{H5}{cond:hard-deriv-strict}] if $\sigma_B = \sigma$ then $b_B\ge b$,
	\item[\optionaldesc{H6}{cond:soft-deriv-strict}] if $\sigma_B = \sigma_A+1$ then $b_B\ge b_A$.
\end{enumerate}
Hypotheses \eqref{cond:hard-deriv-Fsp}--\eqref{cond:soft-lebesgue-Fsp} are exactly those of 
Theorem \eqref{thm:interior-reg} expressed in terms of the notation of the current result.
Condition \eqref{cond:hard-deriv-strict} is needed additionally to
ensure $F^{\sigma,a}_b(\Omega)\subset F^{\sigma_B,a_B}_{b_B}(\Omega)$.
Similarly \eqref{cond:soft-deriv-strict} is the extra hypothesis needed to
ensure $F^{\sigma_A+1,a_A}_{b_A}(\Omega_A)\subset F^{\sigma_B,a_B}_{b_B}(\Omega_A)$.
If we additionally assume that $(\sigma_A,a_A,b_A)$ satisfies the conditions of Lemma \ref{lem:commutator-Fsp}
so that its commutator estimate applies, the bootstrap step argument of Theorem \ref{thm:interior-reg} then 
goes through with obvious changes and we obtain an  open set $\Omega_B$ with 
$\overline U\subset \Omega_B \subset \Omega_A$ such that $u\in F^{\sigma_B,a_B}_{b_B}(\Omega_B)$
along with the estimate
\[
||u||_{F^{\sigma_B,a_B}_{b_B}(\Omega_B)} \lesssim ||Lu||_{F^{\sigma-d,a}_b(\Omega)} + ||u||_{F^{\sigma_A,a_A}_{b_A}(\Omega_A)}.
\]

Now consider the bootstrap in the case $d_0=0$ where we pass through a sequence of
regularity parameters $(\sigma_k,a_k,b_k)$ starting from $(\sigma_,a_0,b_0)=(d-s-1,p^*,q^*)$;
we need not track the shrinking open sets. Focusing for the moment only on the parameters $\sigma_k$ and $a_k$, 
the bootstrap consists of three distinct stages:
\begin{enumerate}
	\item The initial step arriving at $(\sigma_1,a_1,b_1)=(d-s,p^*,q^*)$.
	\item A low regularity stage that either 
		\begin{itemize}
			\item preserves $\sigma_k=d-s$ while lowering $1/a_k$ by at most $1/n$ per step, or
			\item preserves the Lebesgue regularity at the low value $1/p^*-(d-s)/n$ while raising $\sigma_k$
			      by at most 1 per step.
		\end{itemize}
		At the end of the low regularity stage $\sigma_k\le \sigma$ and $a_k=a$.
	\item A derivative improving stage where $a_k=a$ is fixed and $\sigma_k$ is raised by at most 1 per step
	      until arriving at its final value. 
\end{enumerate}

We now discuss the sequence of fine parameters $b_k$, which will in fact be set to $q^*$ throughout the sequence
just described except at the last step, where it is set to its desired value.  
\begin{enumerate}
	\item The initial stage starts at $(\sigma_0,a_0,b_0)=(d-s-1,p^*,q^*)$ and wish to improve to
	$(-d+s,p^*,q^*)$.  Because $(\sigma_0+1,a_0,b_0)\in \mathcal S^d_0(F^{s,p}_q)$, the commutator result
	Lemma \ref{lem:commutator-Fsp} applies.  Hypotheses \eqref{cond:hard-deriv-Fsp}--\eqref{cond:soft-lebesgue}
	hold for the same reasons as in Theorem \ref{thm:interior-reg}.  At this step, 
	hypothesis \eqref{cond:hard-deriv-strict} reads ``if $d-s$ = $\sigma$ then $b\le q^*$'',
	which is satisfied by the definition of $\mathcal S^{d}_0(F^{s,p}_q)$. Finally, hypothesis \eqref{cond:soft-lebesgue-strict}
	holds trivially.  Thus this first bootstrap step is justified.
	\item During the low regularity stage we again preserve $b_k=q^*$ at every step.  This is justified
	as follows for the two possibilities:
	\begin{itemize}
		\item Consider a step where $\sigma_k=d-s$ and where $1/a_k$ is lowered by at most $1/n$.
		Conditions \eqref{cond:hard-deriv-Fsp}--\eqref{cond:soft-lebesgue-Fsp} are met for the
		same reasons as in Theorem \eqref{thm:interior-reg} and condition \eqref{cond:soft-lebesgue-strict}
		is met trivially.  Condition \eqref{cond:hard-lebesgue-strict} is a restriction only if $\sigma=d-s$,
		in which case it requires $b\le q^*$; this condition is met by the definition of $\mathcal S^{d}_0(F^{s,p}_q)$.
		We also need to ensure that the commutator estimate can be employed, which can be done by showing
		that $(\sigma_k,a_k,b_k)=(d-s,a_k,q^*)\in \mathcal S^d_0(F^{s,p})$.  In fact, 
		the definition of $\mathcal S^d_0(F^{s,p}_q)$ permits the fine parameter to be $q^*$ along the
		line $\sigma=d-s$, even in the marginal case $s=d-s$ where the region $S^{d}_0(F^{s,p}_q)$ collapses
		to a line segment.  The remainder of the justification of the commutator estimate is the same
		as in Theorem \ref{thm:interior-reg}.
		\item Consider a step where the Lebesgue regularity is preserved at the low value $1/p^*-(d-s)/n$.
		and where $\sigma_k$ is raised by at most 1; without loss of generality we can assume we raise 
		$\sigma_k$ by less than 1.  Hypotheses 
		\eqref{cond:hard-deriv-Fsp}--\eqref{cond:soft-lebesgue} are met for the same reasons as in Theorem
		\ref{thm:interior-reg}. Condition \eqref{cond:soft-deriv-strict} is always met because of our additional
		restriction on the step size.  Condition \eqref{cond:hard-lebesgue-strict} only comes into play
		if we are raising $\sigma_k$ to its terminal value, in which case we also set the fine parameter to its
		terminal value $b$ (and stop the bootstrap). We need
		to ensure that each non-terminal $(\sigma_k,a_k,b_k)$ lies in $S^d_0(F^{s,p}_q)$ in order to apply the commutator
		estimate, but this is ensured because a fine parameter value of $q^*$ is always permitted along this line
		of Lebesgue regularity.
	\end{itemize}
	\item On a step where we raise $\sigma_k$ and leave $a_k$ fixed at its terminal value we can 
	again assume we raise $\sigma_k$ by less than 1.  Throughout this stage we again leave $b_k$ fixed
	at $q^*$ except at the very last step.  There are no fine parameter restrictions that arise
	to allow the commutator estimate to apply, and hypotheses 
	\eqref{cond:hard-deriv-Fsp}--\eqref{cond:soft-lebesgue-Fsp} hold for the same reasons as 
	in Theorem \eqref{thm:interior-reg}. Hypothesis \eqref{cond:soft-deriv-strict} is always
	met because of our restriction on the step size and hypothesis \eqref{cond:hard-lebesgue-strict}
	only comes into play at the final step, where we meet it by setting $b_k$ to its terminal value $b$.
\end{enumerate}
At this point of the procedure we have arrived at the desired parameters $(\sigma,a,b)$,
except in the marginal case $\sigma=d-s$ in which
case we are at $(d-s,a,q^*)$.  Since $\sigma=d-s$, the definition of $\mathcal S^{d}_0(F^{s,p}_q)$
implies $b\le q^*$ and we can use this inequality to confirm that conditions 
\eqref{cond:hard-deriv-Fsp}--\eqref{cond:soft-lebesgue-strict} hold if we perform one more
bootstrap step to improve the fine parameter to its desired value $b$; note that the commutator
result Lemma \eqref{lem:commutator-Fsp} is available for this bootstrap step
since $(d-s,a,q^*)\in \mathcal S^d_0(F^{s,p}_q)$.

Now consider the case where $d_0=1$.  As in Theorem \ref{thm:interior-reg} it suffices to consider
assume $(\sigma,a,b)\in \mathcal S^{d}_1(F^{s,p}_q)$ but
$(\sigma,a,b)\not\in \mathcal S^{d}_0(F^{s,p}_q)$.  Starting from $(\sigma,a,b)$ we define $(\sigma',a',b')\in \mathcal S^{d}_0(F^{s,p}_q)$
by setting $b'=\max(q,b)$ and then applying the following rules:
\begin{enumerate}
	\item If $\sigma<s$, leave $\sigma'=\sigma$ fixed but raise $1/a$ by at most $1/n$ to $1/a'$
	such that $1/a'-\sigma/n=1/p-s/n$.
	\item If $\sigma\ge s$ and $1/p-s/n \le 1/a-\sigma/n$, lower $\sigma$ by at most 1 to $s$ while simultaneously 
	lowering $1/a$ by at most $1/n$ so that the Lebesgue regularity $1/a-\sigma/n=1/a'-\sigma'/n$ is unchanged.
	\item Otherwise, $(\sigma,a)$ satisfies $s\le \sigma\le \sigma+1$ and 
	\[
		\frac{1}{p}-\frac{s+1}{n} \le \frac{1}{a}-\frac{\sigma}{n} < \frac{1}{p}-\frac{s}{n}
	\]
	and we set $(\sigma',a')=(s,p)$.
\end{enumerate}
Note that we set $b'=\max(q,b)$ to satisfy the fine parameter restriction on $\mathcal S^d_0(F^{s,p}_q)$
when $\sigma'=s$.
In all of these three cases, 
$F^{\sigma,a}_b(\Omega)$ is contained in $F^{\sigma',a'}_{b'}(\Omega)$ and we can therefore apply the $d_0=0$
bootstrap to arrive at $(\sigma',a',b')$. Since $(\sigma',a',b')\in \mathcal S_0^d(F^{s,p}_q)$,
a computation shows that $(\sigma'+1,a',b')\in \mathcal S_1^d(F^{s,p}_q)$ and hence the commutator
result from Lemma \ref{lem:commutator-Fsp} can be applied starting from $(\sigma',a',b')$. Hence
we can apply one round of the bootstrap starting from $(\sigma_A,a_A,b_A)=(\sigma',a',b')$
to arrive at $(\sigma_B,a_B,b_B)=(\sigma,a,b)$ 
so long as hypotheses \eqref{cond:hard-deriv-Fsp}--\eqref{cond:soft-deriv-strict}
are met.  The first four are satisfied for the same reasons as in Theorem \ref{thm:interior-reg}
and \eqref{cond:hard-deriv-strict} is satisfied trivially since $b_B=b$.  Finally,
\eqref{cond:soft-deriv-strict} is a restriction only if $\sigma_B=\sigma_A+1$, in which case
$\sigma=s+1$.  But then, since $(\sigma,a,b)\in \mathcal S^d_1(F^{s,p}_q)$, we have assumed $q\le b$
and hence $b'=\max(q,b)=b$. So we are not changing the fine parameter and 
condition \eqref{cond:soft-deriv-strict} is met.
This completes the proof when $d_0=1$,
and the result for higher values of $d_0$ follows from iterating this argument.
\end{proof}

\section{Coefficients in Sobolev-Slobodeckij Spaces}
\label{s:Wsp}

Sobolev-Slobodeckij spaces of functions on an open set $\Omega\subset \Reals^n$ 
are special cases of Triebel-Lizorkin spaces as follows:
\[
W^{s,p}(\Omega) = \begin{cases} F^{s,p}_2(\Omega)& s\in\Ints\\
	F^{s,p}_p(\Omega) & s\not\in\Ints.
\end{cases}
\]
Hence the results of Section \ref{s:Fsp} specialize to statements about Sobolev-Slobodeckij spaces,
which we briefly record here.

\begin{defn}\label{def:L-Wsp}
	Suppose $d_0,d\in \Ints_{\ge 0}$ with $d_0\le d$.
	A differential operator on 
	an open set $\Omega\subseteq \Reals^n$ of the form
	\[
	\sum_{d_0\le |\alpha|\le d} a^\alpha \partial_\alpha
	\]
	is of class $\mathcal L_{d_0}^d(W^{s,p};\Omega)$ 
	for some $s\in\Reals$ and $1<p<\infty$  if each
	\[
	a^\alpha \in W^{s+|\alpha|-d,p}(\Omega ).
	\]
	\end{defn}
	
\begin{theorem}\label{thm:mult-Wsp}
 Suppose $1<p_1,p_2,p<\infty$ and $s_1,s_2,s\in\Reals$.  Let $r_1,r_2$ and $r$ be defined by
\[
\frac{1}{r_1} = \frac{1}{p_1} - \frac{s_1}{n},\qquad
\frac{1}{r_2} = \frac{1}{p_2} - \frac{s_2}{n},\quad\text{\rm and}\quad
\frac{1}{r} = \frac{1}{p} - \frac{s}{n}.
\]
Pointwise multiplication of $C^\infty(\overline \Omega)$ 
functions extends to a continuous bilinear map 
$W^{s_1,p_1}(\Omega)\times W^{s_2,p_2}(\Omega)
\rightarrow W^{s,p}(\Omega)$ so long as
\begin{align}
s_1+s_2 &\ge 0\\
\min({s_1,s_2}) \ge \sigma \\
\max\left(\frac{1}{r_1},\frac{1}{r_2}\right) & \le  \frac{1}{r}\\
\frac{1}{r_1} + \frac{1}{r_2} & \le 1\\
\label{eq:SS-lebesgue-mult}
\frac{1}{r_1} + \frac{1}{r_2} & \le \frac{1}{r}
\end{align}
with the the following caveats:
\begin{itemize}
	\item Inequality \eqref{eq:SS-lebesgue-mult} is 
	strict if $\min(1/r_1,1/r_2,1-1/r)=0$.
	\item If $s=s_i\not\in\Ints$, then $p=p_i$, $i=1,2$.
	\item If $s_1,s_2\not\in \Ints$ and $s_1+s_2=0$ then $\frac{1}{p_1}+\frac{1}{p_2}=1$.
\end{itemize}
\end{theorem}

\begin{proposition}\label{prop:mapping-Wsp}
Let $\Omega$ be a bounded open subset of $\Reals^n$.
Suppose $1<p,q<\infty$, 
$s>n/p$, $\sigma\in\Reals$ and 
$d,d_0\in\Ints_{\ge 0}$ with $d\ge d_0$.  
An operator of class $\mathcal L_{d_0}^d(W^{s,p};\Omega)$
extends from a map 
$C^\infty(\overline \Omega)\mapsto \mathcal D'(\Omega)$ to 
a continuous linear map 
$W^{\sigma,q}(\Omega)\mapsto W^{\sigma-d,q}(\Omega)$
so long as
\begin{equation}\label{eq:S-conds-Wsp}
\begin{aligned}
\sigma&\in [d-s,s+d_0]\\
\frac{1}{q} - \frac{\sigma}{n} &\in \left[
\frac{1}{p} - \frac{s+d_0}{n}, \frac{1}{p^*}-\frac{d-s}{n}
\right]
\end{aligned}
\end{equation}
and so long as:
\begin{itemize}
	\item If $s\not\in\Ints$ and $\sigma=s+d_0$ then $q=p$.
	\item If $s\not\in\Ints$ and $\sigma=d-s$ then $q=p^*$.
\end{itemize}
Moreover, operators in $\mathcal L_{d_0}^d(W^{s,p};\Omega)$ 
depend continuously
on their coefficients $a^\alpha\in W^{s,p+|\alpha|-d}(\Omega)$.
\end{proposition}

\begin{defn}\label{def:scriptS-Wsp}
	Suppose $1<p<\infty$, $s\in\Reals$ and $d,d_0\in\Ints_{\ge 0}$ with $d\ge d_0$.
	The \textbf{compatible Sobolev indices} 
	for an operator of class $\mathcal L_{d_0}^d(W^{s,p};\Omega)$ is the set
	\[
	\mathcal S_{d_0}^d(W^{s,p}) \subseteq \Reals\times (1,\infty)
	\]
	of tuples $(\sigma,q)$ 
	satisfying \eqref{eq:S-conds-Wsp} along with the additional 
	conditions at the end of Proposition \ref{prop:mapping-Wsp}
	when $\sigma=s+d_0$ or $\sigma=d-s$.
\end{defn}
	
\begin{lemma}\label{lem:S-members-Wsp}
	Suppose $1<p<\infty$, $s\in\Reals$ and $d,d_0\in \Ints_{\ge 0}$
	with $d\ge d_0$.  Then $\mathcal S_{d_0}^d(W^{s,p})$ 
	is nonempty if and only if
	\begin{align}
	\label{eq:S-p-1-Wsp}
	s &\ge (d-d_0)/2\text{, and}\\
	\label{eq:S-p-2-Wsp}
	\frac{1}{p} - \frac{s}{n} &\le \frac{1}{2}-\frac{(d-d_0)/2}{n}
	\end{align}
	with the additional condition when $s\not\in\Ints$ that $p = 2$ in the marginal case
	$s=(d-d_0)/2$.
	If $S_{d_0}^d(W^{s,p})$ is non-empty then it contains $(s+d_0,p)$, $(d-s,p^*)$,
	and $((d+d_0)/2,2)$.  Moreover, if $(\sigma,q)\in \mathcal S_{d_0}^d(W^{s,p})$, then
	we have the continuous inclusions of Fr\'echet spaces
	\begin{equation}\label{eq:S-include-Wsp}
	W^{s+d_0,p}_{\rm loc}(\Reals^n) \subseteq W^{\sigma,q}_{\rm loc}(\Reals^n)
	\subseteq W^{d-s,p^*}_{\rm loc}(\Reals^n).
	\end{equation}
	\end{lemma}
	
\begin{theorem}\label{thm:interior-reg-Wsp} 
	Let $\Omega$ be a bounded open set in $\Reals^n$ and suppose $s,p, d_0$ and $d$ are parameters
	as in Lemma \ref{lem:S-members-Wsp} such that
	$s>n/p$ and such that the conditions of 
	Lemma \ref{lem:S-members-Wsp}
	are satisfied so $\mathcal S_{d_0}^d(W^{s,p})\neq \emptyset$.
	Suppose $L$ is of class $\mathcal L_{d_0}^d(W^{s,p};\Omega)$
	and is elliptic on $\Omega$.  If $u\in W^{d-s,p^*}(\Omega)$
	and $Lu\in W^{\sigma-d,q}(\Omega)$ for some 
	$(\sigma,q)\in \mathcal S_{d_0}^d(W^{s,p})$
	then for any open set $U$ with $\overline U\subseteq \Omega$, $u\in W^{\sigma,q}(U)$ and 
	\begin{equation}\label{eq:int-reg-Wsp}
	||u||_{W^{\sigma,q}(U)} \lesssim 
	||Lu||_{W^{\sigma-d,q}(\Omega)}
	+ ||u||_{W^{d-s-1,p^*}(\Omega)}.
	\end{equation}
	\end{theorem}

\section{Coefficients in Besov Spaces}
\label{s:Bsp}

We establish results for operators with coefficients in Besov spaces, mirroring the developments of the preceding sections.
Recall from Section \ref{s:Fsp} the Littlewood-Paley projectors $P_k$ and $P_{\le k}$.
Let $1<p,q<\infty$ and $s\in\Reals$.
A tempered distribution $u$ belongs to the Besov
space $\Bv[s,p,q](\Reals^n)$ if
\begin{equation}\label{e:Besov-norm}
||u||_{\Bv[s,p,q](\Reals^n)} =
\left\|P_{\le 0} u\right\|_{L^p(\Reals^n)}
+ \left( \sum_{k\ge 1} 2^{sqk}\|P_ku\|_{L^p(\Reals^n)}^q\right)^{\frac1q} 
< \infty.
\end{equation}
On an open set 
$\Omega\subseteq \Reals^n$, the space $\Bv[s,p,q](\Omega)$
consists of the restrictions of distributions in
$\Bv[s,p,q](\Reals^n)$ to $\Omega$ and is given the quotient norm.

Embedding properties of Besov spaces can be found in 
\cite{triebel_theory_2010} Proposition 2.3.2/2, and Theorems
2.7.1 and 3.3.1 along with \cite{triebel_interpolation_1978} Theorems 2.8.1 and 4.6.1; we summarize these 
in the following proposition.  The important
distinction here from Triebel-Lizorkin spaces is that when performing Sobolev embedding,
the fine parameter cannot be improved if the Lebesgue regularity $1/p-s/n$ stays fixed.
This phenomenon is the source of many of the additional fine parameter restrictions in this
section beyond those of Section \ref{s:Fsp}.

\begin{proposition}\label{prop:embedding-Bsp}
Assume $1<p,p_1,p_2,q,q_1,q_2<\infty$
and $s,s_1,s_2\in \Reals$, and suppose $\Omega$ is a bounded open set in  $\Reals^n$.
\begin{enumerate}
	\item 
	If $s_1>s_2$ then
$\Bv[s_1,p,q_1](\Reals^n)\hookrightarrow \Bv[s_2,p,q_2](\Reals^n)$ and $\Bv[s_1,p,q_1](\Omega)\hookrightarrow \Bv[s_2,p,q_2](\Omega)$.
	\item 
	If $q_1\le q_2$ then
$\Bv[s,p,q_1](\Reals^n)\hookrightarrow \Bv[s,p,q_2](\Reals^n)$ and $\Bv[s,p,q_1](\Omega)\hookrightarrow \Bv[s,p,q_2](\Omega)$.
	\item\label{part:embedding-Bsp-bounded-same-s}
	If $p_1\ge p_2$ then
$\Bv[s,p_1,q](\Omega)\hookrightarrow \Bv[s,p_2,q](\Omega)$.
	\item
	If $s_1>s_2$ and
	$\frac{1}{p_1}-\frac{s_1}{n} = \frac{1}{p_2}-\frac{s_2}{n}$ then
$\Bv[s_1,p_1,q](\Reals^n)\hookrightarrow \Bv[s_2,p_2,q](\Reals^n)$ and $\Bv[s_1,p_1,q](\Omega)\hookrightarrow \Bv[s_2,p_2,q](\Omega)$.
\item
If $s_1>s_2$ and
$\frac{1}{p_1}-\frac{s_1}{n} < \frac{1}{p_2}-\frac{s_2}{n}$ then
$\Bv[s_1,p_1,q_1](\Omega)\hookrightarrow \Bv[s_2,p_2,q_2](\Omega)$.
\item If $0<\alpha<1$ then $\Bv[\frac{n}p+\alpha,p,q](\Reals^n)\hookrightarrow C^{0,\alpha}(\Reals^n)$
and $\Bv[\frac{n}p+\alpha,p,q](\Omega)\hookrightarrow C^{0,\alpha}(\Omega)$.
\end{enumerate}
\end{proposition}
As noted previously following Proposition \ref{prop:embedding-Fsp}, although \cite{triebel_theory_2010} 
and \cite{triebel_interpolation_1978} only
prove the embedding results above for bounded domains when the boundary is smooth, the result for arbitrary
bounded open sets is an easy corollary.

Complex interpolation of Besov spaces (\cite{triebel_theory_2010} Theorems 2.4.7 and 3.3.6) 
follows the same pattern as for Triebel-Lizorkin spaces.
\begin{proposition}
	Assume $1<p_1,p_2,q_1,q_2< \infty$
	and $s_1,s_2\in \Reals$, and suppose $\Omega$ is either $\Reals^n$ or is
	a bounded $C^\infty$ domain in $\Reals^n$.
	For $0<\theta<1$,
	\[
		[\Bv[s_1,p_1,q_1](\Omega),\Bv[s_2,p_2,q_2](\Omega)]_\theta = \Bv[s,p,q](\Omega)
	\]
	where
	\[
		s=(1-\theta)s_1+\theta s_2,\quad \frac{1}{p} = (1-\theta)\frac{1}{p_1}+\theta \frac{1}{p_2},\quad
		(1-\theta)\frac{1}{q_1}+\theta \frac{1}{q_2}.
	\]
\end{proposition}

Duality for Besov spaces of functions on $\Reals^n$ is analogous to that for Triebel-Lizorkin spaces; see
\cite{triebel_theory_2010} Theorem 2.11.2.
\begin{proposition}\label{prop:dual-}
	Assume $1<p,q<\infty$
	and $s\in \Reals$.
	The bilinear map $\mathcal D(\Reals^n)\times \mathcal D(\Reals^n)\to \Reals$
	given by $\left<f,g\right> := \int_\Omega fg$ extends to a continuous bilinear map
	$\Bv[s,p,q](\Reals^n)\times \Bv[-s,p^*,q^*](\Reals^n)\to \Reals$.
	Moreover, 
	$f\mapsto \left<f,\cdot\right>$ is a continuous identification of $\Bv[s,p,q](\Reals^n)$
	with $(\Bv[-s,p^*,q^*](\Reals^n))^*$.
\end{proposition}

\subsection{Mapping properties}

As for the other function spaces, mapping properties of differential operators
depend on the rules for multiplication in Besov spaces.
We recall the relevant result here,
and give a self-contained proof in Appendix \ref{app:mult-Bsp}.

\begin{theorem}\label{thm:mult-Besov}
Let $\Omega$ be a bounded open subset of $\Reals^n$.
	Suppose $1<p_1,p_2,p,q_1,q_2,q<\infty$ and $s_1,s_2,s\in\Reals$.  Let $r_1,r_2$ and $r$ be defined by
\[
\frac{1}{r_1} = \frac{1}{p_1} - \frac{s_1}{n},\qquad
\frac{1}{r_2} = \frac{1}{p_2} - \frac{s_2}{n},\quad\text{\rm and}\quad
\frac{1}{r} = \frac{1}{p} - \frac{s}{n}.
\]
Pointwise multiplication of $C^\infty(\overline \Omega)$ 
functions extends to a continuous bilinear map 
$\Bv[s_1,p_1,q_1](\Omega)\times \Bv[s_2,p_2,q_2](\Omega)
\to \Bv[s,p,q](\Omega)$ so long as
\begin{align}
\min(s_1,s_2)&\ge s\\
s_1+s_2&\ge 0\\
\max\left(\frac 1{r_1},\frac 1{r_2}\right) & \le  \frac 1 r\\
\label{eq:r1-r2-vs-1-Besov-baby-5}
\frac{1}{r_1} + \frac{1}{r_2} & \le 1\\
\label{eq:r1-r2-vs-r-Besov-baby-5}
\frac{1}{r_1} + \frac{1}{r_2} & \le \frac{1}{r}
\end{align}
with the following caveats:
\begin{itemize}
	\item If $s=s_i$ or $1/r=1/r_i$ for some $i$ then $1/q\le 1/q_i$.
	\item If $s_1+s_2=0$ or $1/r_1+1/r_2=1$ then $1/q_1+1/q_2\ge1$.
	\item If equality holds in \eqref{eq:r1-r2-vs-r-Besov-baby-5} then
	\begin{itemize}
	\item $\min(1/r_1,1/r_2,1-1/r)\neq0$.	
	\item If $\min(s_1,s_2)\le0$ then $1/q_1+1/q_2\ge1$ and $1/q \le 1/r$.
    \item If $s_i$ has the same sign as $\min(s_1,s_2)$ for some $i$ then $1/q\le 1/q_i$.
    \item If $s_i$ has the same sign as $\max(s_1,s_2)$ for some $i$ then $1/r_i\le 1/q_i$.
	\item If $s=0$ then $1/q\le 1/q_i$ and $1/r_i\le 1/q_i$ for both $i=1,2$.
	\end{itemize}
	\item If $s_1=s_2=s=0$ then $\displaystyle \frac 1q\le \min\left(\frac 12,\frac 1r\right)$ and 
	$\displaystyle \max\left(\frac 12,\frac 1{r_i}\right)\le \frac 1{q_i}$ for both $i=1,2$.
\end{itemize}
\end{theorem}
The list of caveats above is extensive in comparison with Theorem \ref{thm:mult-Fsp},
but the bulk of these occur when inequality \eqref{eq:r1-r2-vs-r-Besov-baby-5} 
is not strict, and we do not encounter this edge case in our applications. 

Operators with coefficients in Besov spaces
are defined analogously to those of Definition \ref{def:L-Hsp}.
\begin{defn}\label{def:L-Bsp}
Suppose $d_0,d\in \Ints_{\ge 0}$ with $d_0\le d$.
A differential operator on 
an open set $\Omega\subseteq \Reals^n$ of the form
\[
\sum_{d_0\le |\alpha|\le d} a^\alpha \partial_\alpha
\]
is of class $\mathcal L_{d_0}^d(\Bv[s,p,q];\Omega)$ 
for some $s\in\Reals$ and $1<p,q<\infty$ if each
\[
a^\alpha \in \Bv[s+|\alpha|-d,p,q](\Omega).
\]
\end{defn}

Theorem \ref{thm:mult-Besov} implies the following.  Notably,
in the computations that lead to this result, the caveats
of Theorem \ref{thm:mult-Besov} concerning the edge cases
$1/r_1+1/r_2=1/r$ and $s_1=s_2=s=0$ never occur.
\begin{proposition}\label{prop:mapping-Bsp}
Let $\Omega$ be a bounded open subset of $\Reals^n$.
	Suppose $1<p,q,a,b<\infty$, 
$s>n/p$, $\sigma\in\Reals$ and 
$d,d_0\in\Ints_{\ge 0}$ with $d\ge d_0$.  
An operator of class $\mathcal L_{d_0}^d(\Bv[s,p,q];\Omega)$
extends from a map 
$C^\infty(\overline \Omega)\mapsto \mathcal D'(\Omega)$ to 
a continuous linear map 
$\Bv[\sigma,a,b](\Omega)\mapsto \Bv[\sigma-d,a,b](\Omega)$
so long as
\begin{equation}\label{eq:S-conds-Bsp}
\begin{aligned}
\sigma&\in [d-s,s+d_0]\\
\frac{1}{a} - \frac{\sigma}{n} &\in \left[
\frac{1}{p} - \frac{s+d_0}{n}, \frac{1}{p^*}-\frac{d-s}{n}
\right]
\end{aligned}
\end{equation}
and so long as:
\begin{itemize}
	\item If $\sigma=s+d_0$ or $\frac{1}{a} - \frac{\sigma}{n}=\frac{1}{p} - \frac{s+d_0}{n}$ then $\frac{1}{b} \le \frac{1}{q}$.
	\item If $\sigma=d-s$ or $\frac{1}{a} - \frac{\sigma}{n}=\frac{1}{p^*} - \frac{d-s}{n}$ then $\frac{1}{b} \ge \frac{1}{q^*}$. 
\end{itemize}
Moreover, operators in $\mathcal L_{d_0}^d(\Bv[s,p,q];\Omega)$ 
depend continuously
on their coefficients $a^\alpha \in \Bv[s-d+|\alpha|,p,q](\Omega)$.
\end{proposition}

We have the following
generalization of Definition \ref{def:scriptS-Hsp}.
\begin{defn}\label{def:scriptS-Bsp}
Suppose $1<p,q<\infty)$, $s\in\Reals$ and $d,d_0\in\Ints_{\ge 0}$ with $d\ge d_0$.
The \textbf{compatible Sobolev indices} 
for an operator of class $\mathcal L_{d_0}^d(\Bv[s,p,q];\Omega)$ is the set
\[
\mathcal S_{d_0}^d(\Bv[s,p,q]) \subseteq \Reals\times (1,\infty)\times(1,\infty)
\]
of tuples $(\sigma,a,b)$ 
satisfying \eqref{eq:S-conds-Bsp} along with the additional 
conditions at the end of Proposition \ref{prop:mapping-Bsp}
in any of the boundary cases $\sigma=s+d_0$, $\sigma=d-s$, 
$1/a-\sigma/n = 1/p-s/n$ or $1/a-\sigma/n = 1/p^*-(d-s)/n$.
\end{defn}

The analogue of Lemma \ref{lem:S-members} in the Besov context is the following.

\begin{lemma}\label{lem:S-members-Bsp}
Suppose $1<p,q<\infty$, $s\in\Reals$ and $d,d_0\in \Ints_{\ge 0}$
with $d\ge d_0$.  Then $\mathcal S_{d_0}^d(\Bv[s,p,q])$ 
is nonempty if and only if
\begin{align}
\label{eq:S-p-1-Bsp}
s &\ge (d-d_0)/2\text{, and}\\
\label{eq:S-p-2-Bsp}
\frac{1}{p} - \frac{s}{n} &\le \frac{1}{2}-\frac{(d-d_0)/2}{n}
\end{align}
with the additional condition $q\le 2$ in each of the marginal cases
$s=(d-d_0)/2$ and $\frac{1}{p} - \frac{s}{n} = \frac{1}{2}-\frac{(d-d_0)/2}{n}$.
If $S_{d_0}^d(\Bv[s,p,q])$ is non-empty then it contains $(s+d_0,p,q)$, $(d-s,p^*,q^*)$,
and $((d+d_0)/2,2,2)$.  Moreover, if $(\sigma,a,b)\in \mathcal S_{d_0}^d(\Bv[s,p,q])$, then
we have the continuous inclusions of Fr\'echet spaces
\begin{equation}\label{eq:S-include-Bsp}
\Bvloc[s+d_0,p,q](\Reals^n) \subseteq \Bvloc[\sigma,a,b](\Reals^n)
\subseteq \Bvloc[d-s,p^*,q^*](\Reals^n).
\end{equation}
\end{lemma}

The following commutator result is the analogue of Lemma \ref{lem:commutator-Fsp}
\begin{lemma} \label{lem:commutator-Bsp}
Suppose $1<p,q,a,b<\infty$, $s>n/p$, $\sigma\in\Reals$ and 
$d,d_0\in\Ints_{\ge 0}$ with $d\ge d_0$. Let $\Omega$ be a bounded
open subset of $\Reals^n$ and let
$L$ be an operator of class $\mathcal L_{d_0}^d(\Bv[s,p,q];\Omega)$.
If $\phi\in\mathcal D(\Omega)$ then $[L,\phi]$ 
extends from a map 
$C^\infty(\overline \Omega)\mapsto \mathcal D'(\Omega)$ to 
a continuous linear map 
$\Bv[\sigma,a,b](\Omega)\mapsto \Bv[\sigma-d+1,a,b](\Omega)$
so long as  $(\sigma + 1,a,b) \in \mathcal S^{d}_{d_0}(\Bv[s,p,q])$. Moreover,
if $d_0=0$, the same result holds if $(\sigma,a,b)\in \mathcal S^{d}_{0}(\Bv[s,p,q])$.
\end{lemma}
\begin{proof}
The proof is a computation using Theorem \ref{thm:mult-Besov} that parallels that of 
Lemma \ref{lem:commutator-Fsp}.   The only difference is that there are three additional 
fine parameter restrictions that arise. In the notation of Theorem \ref{thm:mult-Besov},
these occur when $r=r_i$, when $1/r_1+1/r_2=1$ and when $1/r_1+1/r_2=1/r$. Under
the given hypotheses, the last of these conditions never occurs, and in remaining two cases
the fine parameter restrictions are met for the same reasons as the analogous 
restrictions are met when $s=s_i$ and when $s_1+s_2=0$
in Lemma \ref{lem:commutator-Fsp}.
\end{proof}

\subsection{Rescaling estimates}

The same argument used for Triebel-Lizorkin spaces shows that rescaling $u\mapsto u_\rscr$ is
a continuous automorphism of Besov spaces, and we wish to generalize the associated estimates
of Theorem \ref{prop:rescale-Fsp} to the Besov setting. This could be accomplished by suitably modifying
the arguments of Section \ref{secsec:TL-rescaling}, but we can prove the desired results as
a corollary of the Triebel-Lizorkin estimates using the following real interpolation property, 
which follows from \cite{triebel_theory_2010} Theorem 2.4.2.

\begin{proposition}
	Assume $1<p, q_1, q_2, q < \infty$
	and $s_1,s_2\in \Reals$ with $s_1\neq s_2$.
	Suppose $0<\theta<1$ and let $s=(1-\theta)s_1+\theta s_2$. Then
	\[
		[F^{s_1,p}_{q_1}(\Reals^n),F^{s_2,p}_{q_2}(\Reals^n)]_{\theta,q} = \Bv[s,p,q](\Reals^n).
	\]
\end{proposition}

\begin{proposition}\label{prop:rescale-Bsp}
	Suppose $1<p,q<\infty$,
	 $s\in\Reals$
	and that $\chi$ is a Schwartz function on $\Reals^n$.
	There exists a constant $\alpha\in\Reals$ such that
	for all $0<r\le 1$ and all $u\in \Bv[s,p,q](\Reals^n)$
	\begin{equation}\label{eq:Bsp-scale-b}
	||\chi u_{\{r\}}||_{\Bv[s,p,q](\Reals^n)} \lesssim r^{\alpha}||u||_{\Bv[s,p,q](\Reals^n)}.
	\end{equation}
	Specifically:
	\begin{enumerate}
		\item\label{part:rescale-Bsp-generic-b}
	Inequality \eqref{eq:Bsp-scale-b} holds with
	\[
	\alpha = \min\left(s-\frac{n}{p},0\right)
	\]
	unless $s-n/p=0$, in which case it holds for any choice
	of $\alpha<0$, with implicit constant depending on 
	$\alpha$.
		\item\label{part:rescale-Bsp-zero-cente-b}
		If $s>n/p$ (in which case functions in $\Bv[s,p,q](\Reals^n)$
		are H\"older continuous) and if $u\in \Bv[s,p,q](\Reals^n)$
		with $u(0)=0$, then inequality holds with
		\[
	\alpha = \min\left(s-\frac{n}{p},1\right)
		\]
		unless $s-n/p=1$, in which case it holds for any choice of
		$\alpha<1$, with implicit constant depending on $\alpha$.
	\end{enumerate}
	\end{proposition}
\begin{proof}
Suppose $s<n/p$. Pick $\epsilon>0$ such that $s_1=s+\epsilon< n/p$ as well, and let $s_2=s-\epsilon$.
From real interpolation with endpoints $F^{s_1,p}_2(\Reals^n)$ and $F^{s_2,p}_2(\Reals^n)$
along with Proposition \ref{prop:rescale-Fsp} we find
\[
||u_{r}||_{\Bv[s,p,q](\Reals^n)} \lesssim  r^{\frac 12\left(s_1-\frac np\right)+\frac12 \left(s_2-\frac np\right)}||u||_{\Bv[s,p,q](\Reals^n)}
= r^{s-\frac np} ||u||_{\Bv[s,p,q](\Reals^n)}.
\]
A similar and easier proof works when $s>n/p$ and the marginal case $s=n/p$ can be handled 
by interpolation between endpoints with differentiability $s_1$ and $s_2$ taken arbitrarily close
to $s$, as in the proof of Proposition \ref{prop:poor-mans-bp}.

For the improved estimate in the H\"older continuous case, we let
$F^{s,p}_{q,0}(\Reals^n)$ be the closed subspace of $F^{s,p}_q(\Reals^n)$
of functions that vanish at $0$, assuming of course that $s>n/p$.  The spaces
$\Bvz[s,p,q]$ are defined similarly. The same argument as at the end of 
the proof of Proposition \ref{prop:rescale-Fsp-u-zero} shows that
\begin{equation}\label{eq:interp-zero}
[F^{s_1,p}_{q_1,0}(\Reals^n),F^{s_2,p}_{q_2,0}(\Reals^n)]_{\theta,q} = 
\Bvz[s,p,q](\Reals^n)
\end{equation}
assuming that $s_i>n/p$ for $i=1,2$ and that 
$0<\theta<1$ and $s=(1-\theta)s_1+\theta s_2$.  Using this interpolation
property, the proof of the improved estimate now follows exactly 
as in the generic case.
\end{proof}

\subsection{Interior elliptic estimates}

Elliptic operators with Besov space coefficients are defined analogously as 
in Definition \ref{def:L-elliptic-Hsp}.
The following ``regularity at a point'' result is proved identically as for Proposition \ref{prop:elliptic-zoom-Fsp},
using the fact that the parametrix result Lemma \ref{lem:parametrix} is proved identically for Besov spaces. Note that 
$\tildeBv[s,p,q](\Omega)$ denotes the closure of $\mathcal D(\Omega)$ in $\Bv[s,p,q](\Reals^n)$.

\begin{proposition}\label{prop:elliptic-zoom-Bsp} 
	Let $\Omega\subset \Reals^n$ be a bounded open set.  
	Suppose $s\in\Reals$, $1<p,q<\infty$, that $d,d_0\in\Ints_{\ge 0}$ with $d_0\le d$, that $s>n/p$,
	and that the conditions of Lemma \ref{lem:S-members-Bsp} are are satisfied and hence 
	$\mathcal S_{d_0}^d(\Bv[s,p,q])\neq \emptyset$.
	Suppose additionally that 
	$L=\sum_{|\alpha|\le d} a^\alpha \partial_\alpha$ 
	is a differential operator of class 
	$\mathcal L_{d_0}^d(\Bv[s,p,q];\Omega)$
	and that for some $x\in \Omega$ that
	\[
	L_0 = \sum_{|\alpha|=m} a^\alpha(x)\partial_\alpha
	\]
	is elliptic.  
	Given $(\sigma,a,b)\in \mathcal S_{d_0}^d(F^{s,p}_q)$
	there exists $r>0$ such that $B_r(x)\subset \Omega$
	and such that if
	\begin{align*}
	u&\in \tildeBv[d-s,p^*,q^*]( B_{r}(x) )\quad \text{and}\\
	Lu&\in \Bv[\sigma-d,a,b](\Omega)
	\end{align*}
	then $u\in \Bv[\sigma,a,b](\Omega)$ and
	\begin{equation}\label{eq:pointwise-est-Bsp}
	||u||_{\Bv[\sigma,a,b](\Omega)} \lesssim
	||L u||_{\Bv[\sigma-d,a,b](\Omega)}
	+ ||u||_{\Bv[d-s-1,p^*,q^*](\Omega)}		
	\end{equation}
	with implicit constant independent of $u$ but depending on all other parameters.
\end{proposition}

With the previous proposition established, local elliptic regularity is proved using the same
techniques as in Theorem \ref{thm:interior-reg-Fsp}, taking into account extra fine-parameter
restrictions that arise for Besov spaces.

\begin{theorem}\label{thm:interior-reg-Bsp} 
	Let $\Omega$ be a bounded open set in $\Reals^n$ and suppose $s,p, d_0$ and $d$ are parameters
	as in Lemma \ref{lem:S-members-Bsp} such that
	$s>n/p$ and such that the conditions of 
	Lemma \ref{lem:S-members-Bsp}
	are satisfied so $\mathcal S_{d_0}^d(\Bv[s,p,q])\neq \emptyset$.  
	Suppose $L$ is of class $\mathcal L_{d_0}^d(\Bv[s,p,q];\Omega)$
	and is elliptic on $\Omega$.  If $u\in \Bv[d-s,p^*,q^*](\Omega)$
	and $Lu\in \Bv[\sigma-d,a,b](\Omega)$ for some 
	$(\sigma,a,b)\in \mathcal S_{d_0}^d(\Bv[s,p,q])$
	then for any open set $U$ with $\overline U\subseteq \Omega$, $u\in \Bv[\sigma,a,b](U)$ and 
	\begin{equation}\label{eq:int-reg-Bsp}
	||u||_{\Bv[\sigma,a,b](U)} \lesssim 
	||Lu||_{\Bv[\sigma-d,a,b](\Omega)}
	+ ||u||_{\Bv[d-s-1,p^*,q^*](\Omega)}.
	\end{equation}
	\end{theorem}
	\begin{proof}
	The proof very closely follows that of Theorem \ref{thm:interior-reg-Fsp}, and we list here
	the additional steps needed to further manage the fine parameter.

	In addition to conditions \eqref{cond:hard-deriv-Fsp}--\eqref{cond:soft-deriv-strict}
	from that proof, the bootstrap step requires two more hypotheses to ensure
	the needed embeddings:
	\begin{enumerate}
		\item[\optionaldesc{H7}{cond:hard-lebesgue-strict}] if $\frac 1{a_B}-\frac{\sigma_B}{n} = \frac 1{a}-\frac{\sigma}{n}$ then $b_B\ge b$,
		\item[\optionaldesc{H8}{cond:soft-lebesgue-strict}] if $\frac 1{a_B}-\frac{\sigma_B}{n} = \frac 1{a_A}-\frac{\sigma_A+1}{n}$ then $b_B\ge b_A$.
	\end{enumerate}

	For the stages of the main bootstrap for $d_0=0$ from Theorem \ref{thm:interior-reg-Fsp} we
	make the following adjustments
	\begin{enumerate}
		\item In the initial step from $(\sigma_0,a_0,b_0)=(d-s-1,p^*,q^*)$ to 
		$(\sigma_1,a_1,b_1)=(d-s,p^*,q^*)$ hypothesis \eqref{cond:soft-lebesgue-strict} is met trivially
		and hypothesis \eqref{cond:hard-lebesgue-strict} only yields a restriction if
		$\frac 1{p^*}-\frac{d-s}{n} = \frac 1{a}-\frac{\sigma}{n}$, in which case it requires $b\le q^*$. 
		But this fine parameter restriction is met because $(\sigma,a,b)=(d-s,a,b)\in \mathcal S^{d}_0(\Bv[s,p,q])$.
		\item During the low regularity stage we again preserve $b_k=q^*$. 
		\begin{itemize}
			\item If we preserve $\sigma_k=d-s$ and lower $1/a_k$ we can arrange to do so by less than $1/n$.
			      In this case hypothesis \eqref{cond:soft-lebesgue-strict} is met automatically.
				  Hypothesis \eqref{cond:hard-lebesgue-strict} only applies if $\sigma=d-s$ and only on
				  the last step where we lower $1/a_k$, in which case we set the fine parameter to its 
				  final value (satisfying \eqref{cond:hard-lebesgue-strict}) and the bootstrap stops.
				  All other aspects of this stage are justified identically to Theorem \ref{thm:interior-reg-Fsp}.
			\item If we preserve the Lebesgue regularity $1/p^*-(d-s)/n$ and raise $\sigma_k$ then hypothesis
			\eqref{cond:soft-lebesgue-strict} is met automatically.  Hypothesis
			\eqref{cond:hard-lebesgue-strict} provides a restriction only if $1/a-\sigma/n=1/p^*-(d-s)/n$,
			in which case we have assumed $b\le q^*$ from the definition of $\mathcal S^{d}_0(\Bv[s,p,q])$
			and condition \eqref{cond:hard-lebesgue-strict} is met. We need to ensure that
			the iterates $(\sigma_k,a_k,b_k)$ all remain in $S^{d}_0(\Bv[s,p,q])$ in order
			for the commutator result Lemma \ref{lem:commutator-Bsp} to apply.  Indeed, the line
			of Lebesgue regularity $1/p^*-(d-s)/n$ is associated with a fine parameter restriction, but
			it is always met by keeping $b_k=q^*$.  
		\end{itemize}
		\item In the final stage we raise $\sigma_k$ (by less than 1 per iteration) while keeping $a_k=a$.  
		Again we preserve $b_k=q^*$ except at the final step.  Hypothesis \eqref{cond:soft-lebesgue-strict}
		is met because of the step size restriction and hypothesis \eqref{cond:hard-lebesgue-strict} is
		a restriction only at the last step, at which point it is satisfied by setting the fine parameter
		to its terminal value $b$.
	\end{enumerate}
	At the end of this procedure, the bootstrap has stopped at its desired value except in the two marginal
	cases $\sigma=d-s$ and $1/a-\sigma/n=1/p^*-(d-s)/n$, in which case we have arrived at $(\sigma,a,q^*)$.
	In both these cases the definition of $\mathcal S^d_0(\Bv[s,p,q])$ ensures that $b<q^*$. Using this inequality,
	one readily verifies that hypotheses \eqref{cond:hard-deriv-Fsp}--\eqref{cond:soft-lebesgue-strict}
	hold if we perform one more bootstrap step to improve the fine parameter to its final value $b$. As
	in the proof of Theorem \ref{thm:interior-reg-Fsp}, one needs to verify that the starting
	regularity $(\sigma,a,q^*)$ for the bootstrap lies in $\mathcal S^{d}_0(\Bv[s,p,q])$ so as to use the commutator result
	Lemma \ref{lem:commutator-Bsp}, but this is always met for the fine parameter $q^*$ on
	the two marginal lines $\sigma=d-s$ and $1/a-\sigma/n=1/p^*-(d-s)/n$. This completes the proof when $d_0=0$.

	Now consider the case $d_0=1$.  Arguing as in Theorem \ref{thm:interior-reg-Fsp},
	we can apply the $d_0=0$ bootstrap to arrive at some $(\sigma',a',b')\in \mathcal S^{d}_0(\Bv[s,p,q])$ 
	where $b'=\max(q,b)$ and where $(\sigma',a')$ is obtained from $(\sigma,a)$ as follows:
	\begin{enumerate}
		\item\label{case:lower-1/a} If $\sigma<s$, leave $\sigma'=\sigma$ fixed but raise $1/a$ by at most $1/n$ to $1/a'$
		such that $1/a'-\sigma/n=1/p-s/n$.
		\item If $\sigma\ge s$ and $1/p-s/n \le 1/a-\sigma/n$, lower $\sigma$ by at most 1 to $s$ while simultaneously 
		lowering $1/a$ by at most $1/n$ so that the Lebesgue regularity $1/a-\sigma/n=1/a'-\sigma'/n$ is unchanged.
		\item \label{case:move-other} Otherwise, $(\sigma,a)$ satisfies $s\le \sigma\le \sigma+1$ and 
		\[
			\frac{1}{p}-\frac{s+1}{n} \le \frac{1}{a}-\frac{\sigma}{n} < \frac{1}{p}-\frac{s}{n}
		\]
		and we set $(\sigma',a')=(s,p)$.
	\end{enumerate}
	Note that we set $b'=\max(q,b)$ to ensure that the fine parameter restrictions of $\mathcal S^{d}_0(\Bv[s,p,q])$
	along the lines $\sigma'=s$ and $1/a'-\sigma'/n=1/p-s/n$ are met.  Since $\Bv[\sigma,a,b](\Omega)$
	embeds in $\Bv[\sigma',a',b'](\Omega)$ we can initially apply the $d_0=0$ bootstrap to arrive at 
	$(\sigma',a',b')$.  At this point we wish to apply a single bootstrap
	step to terminate at $(\sigma,a,b)$.  Since $(\sigma',a',b')\in \mathcal S^{d}_0(\Bv[s,p,q])$ a computation
	shows $(\sigma'+1,a',b')\in \mathcal S^d_1(\Bv[s,p,q])$ and hence the commutator Lemma \ref{lem:commutator-Bsp}
	is available starting from $(\sigma',a',b')$. Hence we can perform the desired bootstrap step
	if we show that hypotheses 
	\eqref{cond:hard-deriv-Fsp}--\eqref{cond:soft-lebesgue-strict} hold with 
	$(\sigma_A,a_A,b_A)=(\sigma',a',b')$ and $(\sigma_B,a_B,b_B)=(\sigma,a,b)$.
	
	Conditions \eqref{cond:hard-deriv-Fsp}--\eqref{cond:soft-lebesgue-Fsp} hold for the same 
	reasons as in Theorem \ref{thm:interior-reg} and conditions \eqref{cond:hard-deriv-strict}
	and \eqref{cond:hard-lebesgue-strict} hold trivially since we are setting the fine parameter to $b$.
	Condition \eqref{cond:soft-deriv-strict} holds for exactly the same 
	reason as in Theorem \ref{thm:interior-reg-Fsp}.
	Finally, 
	condition \eqref{cond:soft-lebesgue-strict} implies a restriction only in cases 
	\eqref{case:lower-1/a} and \eqref{case:move-other} and only when $1/a-\sigma/n=1/p-(s+1)/n$.
    But then
	$(\sigma,a,b)\in \mathcal S^d_1(\Bv[s,p,q])$ is a spot where the fine parameter restriction $q\le b$
	holds and hence $b'=\max(q,b)=b$ already. Hence condition \eqref{cond:soft-lebesgue-strict} is met.  
	
	This concludes the proof when $d_0=1$
	and the result holds for higher values of $d_0$ by iterating this argument.
\end{proof}

\appendix
\section{Multiplication in Triebel-Lizorkin Spaces}\label{app:mult-Fsp}

We prove the multiplication rules
for Triebel-Lizorkin spaces, Theorem \ref{thm:mult-Fsp}, 
which we restate here for convenience.
\begin{reptheorem}{thm:mult-Fsp}
	Let $\Omega$ be a bounded open subset of $\Reals^n$.
	Suppose $1<p_1,p_2,p,q_1,q_2,q<\infty$ and $s_1,s_2,s\in\Reals$.  Let $r_1,r_2$ and $r$ be defined by
   \[
   \frac{1}{r_1} = \frac{1}{p_1} - \frac{s_1}{n},\qquad
   \frac{1}{r_2} = \frac{1}{p_2} - \frac{s_2}{n},\quad\text{\rm and}\quad
   \frac{1}{r} = \frac{1}{p} - \frac{s}{n}.
   \]
   Pointwise multiplication of $C^\infty(\Omega)$ 
   functions extends to a continuous bilinear map 
   $F^{s_1,p_1}_{q_1}(\Omega)\times F^{s_2,p_2}_{q_2}(\Omega)
   \rightarrow F^{s,p}_{q}(\Omega)$ so long as
   \begin{align}
   s_1+s_2\ge 0\\
   \min({s_1,s_2})\ge s\\
   \max\left(\frac{1}{r_1},\frac{1}{r_2}\right) & \le  \frac{1}{r}\\
   \frac{1}{r_1} + \frac{1}{r_2} & \le 1\\
   \label{eq:Fsp-lebesgue-mult-b}
   \frac{1}{r_1} + \frac{1}{r_2} & \le \frac{1}{r}
   \end{align}
   with the the following caveats:
   \begin{itemize}
	   \item Inequality \eqref{eq:Fsp-lebesgue-mult-b} is 
	   strict if $\min(1/r_1,1/r_2,1-1/r)=0$.
	   \item If $s_i=s$ for some $i$ then $1/q\le 1/q_i$.
	   \item If $s_1+s_2=0$ then $\displaystyle\frac{1}{q_1}+\frac{1}{q_2}\ge 1$.
	   \item If $s_1=s_2=s=0$ then $\displaystyle \frac1q\le \frac{1}{2}\le \frac1{q_i}$ for $i=1,2$.
	  \end{itemize}
   \end{reptheorem}

The proof is broken into a number of cases depending on 
the value of $\min(s_1,s_2)$.  Propositions \ref{prop:mult-Fsp-s-pos} and
\ref{prop:mult-Fsp-pos-pos-neg} cover the case $\min(s_1,s_2)>0$,
Proposition \ref{prop:mult-Fsp-s2-neg} is the case $\min(s_1,s_2) <0$
and the remaining case $\min(s_1,s_2)=0$ is the content of Proposition
\ref{prop:mult-Fsp-s2-zero}.

These results all build on the following two lemmas, which
concern multiplication of spaces
having the same number $s>0$
of derivatives and the same fine parameter, 
but where the Lebesgue parameters vary.  The lemmas employ
the same elementary Littlewood-Paley techniques used in Section
\ref{secsec:TL-rescaling}.

\begin{lemma}\label{lem:mult-Fsp-same-s-high}
Suppose $1<p_2\le p_1<\infty$, 
$1<q<\infty$, $s\in\Reals$ and $s>n/p_1$.
Given $u\in F^{s,p_1}_{q}(\Reals^n)$ and $v\in F^{s,p_2}_{q}(\Reals^n)$,
both supported in $B_R(0)$ for some $R>0$, $uv\in F^{s,p_2}_{q}(\Reals^n)$
and 
\[
||uv||_{F^{s,p_2}_{q}(\Reals^n)}\lesssim ||u||_{F^{s,p_1}_{q}}(\Reals^n)||v||_{F^{s,p_2}_{q}(\Reals^n)}.
\]
The implicit constant depends on $s$, $p_1$, $p_2$ $q$ and $R$ but
is independent of $u$ and $v$.
\end{lemma}

\begin{proof}
Since $s> 0$, Proposition \ref{prop:littlewood-payley-review}\eqref{part:sobolev-space-by-projection}
implies
\[
||uv||_{F^{s,p_2}_q(\Reals^n)} \lesssim \
\left|\left| uv\right|\right|_{L^{p_2}(\Reals^n)}
+\left|\left| 
	\left[\sum_{k\ge 10} |2^{sk}P_k(uv)|^q\right]^{1/q} \right|\right|_{L^{p_2}(\Reals^n)}.
\]
For the low frequency component we use the fact 
that $s>n/p$ and Proposition \ref{prop:embedding-Fsp}
twice to conclude 
\[
\left|\left| uv \right|\right|_{L^{p_2}(\Reals^n)}
\lesssim ||u||_{L^\infty(\Reals^n)} ||v||_{L^{p_2}(\Reals^n)} \lesssim 
||u||_{F^{s,p_1}_{q}(\Reals^n)} ||v||_{F^{s,p_2}_q(\Reals^n)}.
\]
Turning to the high-frequency component we use the Littlewood-Paley
trichotomy 
Proposition \ref{prop:littlewood-payley-review}\eqref{part:trichotomy},
to conclude for any $k\ge 10$
\begin{equation}\label{eq:trichotomy-mult}
P_k(uv) = P_k\left(
\underbrace{(P_{\le k-4} u) (\widetilde P_{k}v)}_{\text{low-high}}
+ \underbrace{(\widetilde P_{k}u) (P_{\le k+5}v)}_{\text{high-low}}
+ \underbrace{\sum_{k'\ge k+4}
(P_{k'} u) (\widetilde P_{k'} v)}_{\text{high-high}}\right)
\end{equation}
where $\widetilde P_k = P_{k-3\le \cdot\le k+3}$.
For the low-high contributions we use Proposition \ref{prop:littlewood-payley-review}
parts \eqref{part:throw-away-Pk} and 
\eqref{part:maximal-function-beats-projection} 
to conclude
\begin{align*}
\left|\left| 
	\left[\sum_{k\ge 10} |2^{sk}P_k(
(P_{\le k-4} u)(\widetilde P_{k}v)
	)|^q\right]^{1/q} \right|\right|_{L^{p_2}(\Reals^n)}
&\lesssim 
\left|\left| 
	\left[\sum_{k\ge 10} |2^{sk}
(P_{\le k-4} u)(\widetilde P_{k}v)
	|^q\right]^{1/q} \right|\right|_{L^{p_2}(\Reals^n)} \\
&\lesssim ||Mu||_{L^\infty(\Reals^n)} 
\left|\left|
	\left[\sum_{k\ge 10} |2^{sk}
\widetilde P_{k} v |^q\right]^{1/q} \right|\right|_{L^{p_2}(\Reals^n)}.
\end{align*}
Since $s>n/p$, $||Mu||_{L^\infty(\Reals^n)}\lesssim ||u||_{L^\infty(\Reals^n)}\lesssim ||u||_{F^{s,p_1}_q(\Reals^n)}$.  
Moreover the triangle inequality  implies
\[
\left|\left|
	\left[\sum_{k\ge 10} |2^{sk}
\widetilde P_{k} v |^q\right]^{1/q} \right|\right|_{L^{p_2}(\Reals^n)}
\lesssim ||v||_{F^{s,p_2}_q(\Reals^n)}
\]
and we conclude that the low-high interactions are controlled by
$||u||_{F^{s,p_1}_q(\Reals^n)}||v||_{F^{s,p_2}_q(\Reals^n)}$.

To analyze the high-low term we set
\[
	\frac1t=\frac1{p_2}-\frac1{p_1}
	\] and observe that since $p_1\ge p_2$,
$1<t\le \infty$. The estimate now proceeds similarly to the low-high case: using
Proposition \ref{prop:littlewood-payley-review}
parts \eqref{part:throw-away-Pk} and 
\eqref{part:maximal-function-beats-projection}
along with H\"older's inequality
and the Hardy-Littlewood maximal inequality we find
\begin{align*}
\left|\left| 
	\left[\sum_{k\ge 10} |2^{sk}P_k(
(\widetilde P_{k} u)(P_{\le k+5}v)
	)|^2\right]^{1/2} \right|\right|_{L^{p_2}(\Reals^n)}
&\lesssim 
\left|\left| 
	\left[\sum_{k\ge 10} |2^{sk}
(\widetilde P_{k} u)
	|^q\right]^{1/q} |M v| \right|\right|_{L^{p_2}(\Reals^n)}\\
&\lesssim 
\left|\left| 
	\left[\sum_{k\ge 10} |2^{sk}
(\widetilde P_{k} u)
	|^q\right]^{1/q} \right|\right|_{L^{p_1}(\Reals^n)}
||Mv||_{L^t(\Reals^n)}\\
&\lesssim
||u||_{F^{s,p_1}_q(\Reals^n)}||v||_{L^t(\Reals^n)}.
\end{align*}
Since $s>n/p_1$ we have
\[
\frac{1}{t} = \frac{1}{p_2}-\frac{1}{p_1} > \frac{1}{p_2}-\frac{s}{n}.
\]
So Sobolev embedding and the bounded support of $v$ imply
\[
||v||_{L^t(\Reals^n)} \lesssim ||v||_{F^{s,p_2}_q(\Reals^n)}
\]
with implicit constant depending on the radius $R$ of support.
We conclude that the low-high terms are controlled by $||u||_{F^{s,p_1}_q(\Reals^n)}||v||_{F^{s,p_2}_q(\Reals^n)}$.

Finally, for the high-high contributions we start by
applying Proposition \ref{prop:littlewood-payley-review} parts
\eqref{part:throw-away-Pk} and \eqref{part:maximal-function-beats-projection} to obtain
\begin{align*}
\left|\left| \left[\sum_{k\ge 10} 
\left|2^{sk}
P_k \left( 
\sum_{k'\ge k+4}
(P_{k'} u) (\widetilde P_{k'} v\right)\right|^q\right]^{1/q}\right|\right|_{L^{p_2}(\Reals^n)}
&\lesssim
\left|\left|
|Mu| 
\left[\sum_{k\ge 10} 
\left|2^{sk}
\left( 
\sum_{k'\ge k+4}
 (\widetilde P_{k'} v\right)\right|^q\right]^{1/q}\right|\right|_{L^{p_2}(\Reals^n)}\\
&\lesssim
||Mu||_{L^\infty(\Reals^n)}
\sum_{a\ge 4}
2^{-sa}
\left|\left|
\left[\sum_{k\ge 10} 
\left|2^{s(k+a)}
\left( 
 \widetilde P_{k+a} v\right)\right|^q\right]^{1/q}\right|\right|_{L^{p_2}(\Reals^n)}\\
&\lesssim
||Mu||_{L^\infty(\Reals^n)} ||v||_{F^{s,p_2}_q(\Reals^n)}
\sum_{a\ge 4}
2^{-sa}.
\end{align*}
Since $s>0$ the sum is finite, and we have already seen
that $||Mu||_{L^\infty(\Reals^n)}\lesssim ||u||_{F^{s,p_1}_q(\Reals^n)}$.
So the high-high terms are also controlled by 
$||u||_{F^{s,p_1}_q(\Reals^n)}||v||_{F^{s,p_2}_q(\Reals^n)}$ as needed.
\end{proof}

\begin{lemma}\label{lem:mult-Fsp-same-s-low}
Suppose $s>0$, $1<p\le p_2\le p_1<\infty$, $1<q<\infty$ and $s<n/p_1$.
Suppose moreover that 
\begin{equation}\label{eq:q_vs_ps-H}
\frac{1}{p} \ge  \frac{1}{p_1} + \frac{1}{p_2} -\frac{s}{n}
\end{equation}
Given $u\in F^{s,p_1}_q(\Reals^n)$ and $v\in H^{s,p_2}_q(\Reals^n)$,
both supported in $B_R(0)$ for some $R>0$, $uv\in F^{s,p}_q(\Reals^n)$
and 
\[
||uv||_{F^{s,p}_q(\Reals^n)}\lesssim ||u||_{F^{s,p_1}_q(\Reals^n)}||v||_{F^{s,p_2}_q(\Reals^n)}.
\]
The implicit constant depends on $s$, $p_1$, $p_2$, $q$ and $R$ but
is independent of $u$ and $v$.
\end{lemma}
\begin{proof}
The proof follows the pattern of Lemma \ref{lem:mult-Fsp-same-s-high}.
First apply Proposition \ref{prop:littlewood-payley-review}\eqref{part:sobolev-space-by-projection} to obtain
\[
||uv||_{F^{s,p}_q(\Reals^n)} \lesssim \
\left|\left| uv\right|\right|_{L^p(\Reals^n)}
+\left|\left| 
	\left[\sum_{k\ge 10} |2^{sk}P_k(uv)|^q\right]^{1/q} \right|\right|_{L^p(\Reals^n)}.
\]

To estimate the low frequency term define
\begin{equation}\label{eq:t-def}
\frac{1}{t} = \frac{1}{p} - \frac{1}{p_2}
\end{equation}
and observe 
$1>1/p>1/t \ge (1/p_1)-(s/n)>0$ by inequality \eqref{eq:q_vs_ps-H}
and the hypothesis $s<n/p_1$.  
Sobolev embedding and the fact that $u$ is supported 
on $B_R(0)$ then imply
\begin{equation}\label{eq:t-from-sp}
||u||_{L^t(\Reals^n)} \lesssim ||u||_{F^{s,p_1}_q(\Reals^n)} 
\end{equation}
and we conclude from H\"older's inequality
\[
\left|\left| uv\right|\right|_{L^p(\Reals^n)}
\lesssim  ||u||_{L^{t}(\Reals^n)} ||v||_{L^{p_2}(\Reals^n)}
\lesssim ||u||_{F^{s,p_1}_q(\Reals^n)} ||v||_{F^{s,p_2}_q(\Reals^n)}.
\]

As in the proof of Lemma \ref{lem:mult-Fsp-same-s-high} the high-frequency term is split
into three terms using
the Littlewood-Paley trichotomy, Proposition \ref{prop:littlewood-payley-review}\eqref{part:trichotomy};
see equation \eqref{eq:trichotomy-mult}.
For the low-high contributions we 
use Proposition \ref{prop:littlewood-payley-review}
parts \eqref{part:throw-away-Pk} and \eqref{part:maximal-function-beats-projection}
together with the Hardy-Littlewood maximal inequality
to obtain
\begin{align*}
\left|\left| 
	\left[\sum_{k\ge 10} |2^{sk}P_k(
(P_{\le k-4} u)(\widetilde P_{k}v)
	)|^q\right]^{1/q} \right|\right|_{L^{p}(\Reals^n)}
&\lesssim 
\left|\left| 
	\left[\sum_{k\ge 10} |2^{sk}
(P_{\le k-4} u)(\widetilde P_{k}v)
	|^q\right]^{1/q} \right|\right|_{L^p(\Reals^n)} \\
&\lesssim  
\left|\left| |Mu|
	\left[\sum_{k\ge 10} |2^{sk}
\widetilde P_{k} v |^q\right]^{1/q} \right|\right|_{L^p(\Reals^n)}\\
&\lesssim 
||Mu||_{L^{t}(\Reals^n)} ||v||_{F^{s,p_2}_q(\Reals^n)}\\
&\lesssim
||u||_{L^{t}(\Reals^n)} ||v||_{F^{s,p_2}_q(\Reals^n)}
\end{align*}
where $t$ is again defined as in \eqref{eq:t-def}.
From inequality \eqref{eq:t-from-sp}
we conclude that the low-high interactions are controlled by
$||u||_{F^{s,p_1}_q(\Reals^n)}||v||_{F^{s,p_2}_q(\Reals^n)}$.

The estimate for the high-low contributions proceeds identically to that
for the low-high contributions with the minor change that 
in equation \eqref{eq:t-def} we replace $p_2$ with $p_1$.

Turning now to the high-high 
contributions we define $t$ as in equation in \eqref{eq:t-def}
and apply
 Proposition \ref{prop:littlewood-payley-review} parts
\eqref{part:throw-away-Pk} and \eqref{part:maximal-function-beats-projection} 
to obtain
\begin{align*}
\left|\left| \left[\sum_{k\ge 10} 
\left|2^{sk}
P_k \left( 
\sum_{k'\ge k+4}
(P_{k'} u) (\widetilde P_{k'} v)\right)\right|^q\right]^{1/q}\right|\right|_{L^p(\Reals^n)}
&\lesssim
\left| \left|
|Mu|
\left[\sum_{k\ge 10} 
\left|2^{sk}
\left( 
\sum_{k'\ge k+4}
 (\widetilde P_{k'} v\right)\right|^q\right]^{1/q}\right|\right|_{L^p(\Reals^n)}\\
&\lesssim
||Mu||_{L^t(\Reals^n)}
\sum_{a\ge 4}
2^{-sa}
\left|\left|
\left[\sum_{k\ge 10} 
\left|2^{s(k+a)}
\left( 
 \widetilde P_{k+a} v\right)\right|^q\right]^{1/q}\right|\right|_{L^{p_2}(\Reals^n)}\\
&\lesssim
||u||_{L^t(\Reals^n)} ||v||_{F^{s,p_2}_q(\Reals^n)}
\sum_{a\ge 4}
2^{-sa}.
\end{align*}
Since $s>0$ the sum is finite, and we have already seen
that $||u||_{L^t(\Reals^n)}\lesssim ||u||_{F^{s,p_1}_q(\Reals^n)}$.
So the high-high terms are also controlled by 
$||u||_{F^{s,p_1}_q(\Reals^n)}||v||_{F^{s,p_2}_q(\Reals^n)}$ as needed.
\end{proof}

Having established the technical core of the theory, 
it remains for us to build a conveniently accessible interface
in the form of Theorem \ref{thm:mult-Fsp}.
We now turn to a sequence of results that prove the theorem
by cases based on the sign of $\min(s_1,s_2)$, and
for brevity we establish the following standing hypothesis.

\begin{assumption}\label{as:mult-hyp}
Suppose
\begin{itemize}
\item $\Omega\subseteq \Reals^n$ is a bounded $C^\infty$ domain,
\item $s_1,s_2,s\in\Reals$, 
\item $1<p_1,p_2,p<\infty$,
\item $1<q_1,q_2,q<\infty$. 
\end{itemize} 
Moreover, define
\[
\frac{1}{r_1} = \frac{1}{p_1} - \frac{s}{n},\qquad
\frac{1}{r_2} = \frac{1}{p_2} - \frac{s}{n},\quad\text{\rm and}\quad
\frac{1}{r} = \frac{1}{p} - \frac{s}{n}.
\]
\end{assumption}
Note, in particular, that Assumption \ref{as:mult-hyp} supposes that
$\Omega$ has a smooth boundary.  This hypothesis is made for convenience,
and the proof of Theorem \ref{thm:mult-Fsp} follows for general bounded 
domains from the smooth case using a straightforward extension argument based on the quotient
space definition of the relevant spaces.

The following result is mostly a translation of 
Lemmas \ref{lem:mult-Fsp-same-s-high} and \ref{lem:mult-Fsp-same-s-low}
into the hypotheses of Theorem \ref{thm:mult-Fsp}.

\begin{proposition}\label{prop:mult-same-s}
Assume the multiplication hypothesis \ref{as:mult-hyp}
and that $s_1=s_2=s>0$ and $q_1=q_2=q$. 
Pointwise multiplication of $C^\infty(\overline \Omega)$ 
functions extends to a continuous bilinear map 
$F^{s,p_1}_q(\Omega)\times F^{s,p_2}_q(\Omega)
\rightarrow F^{s,p}_q(\Omega)$
so long as
\begin{align}
\label{eq:same-s-r1-r2-vs-r}
\max\left(\frac 1{r_1},\frac 1{r_2}\right) &\le \frac{1}{r}\\
\label{eq:same-s-r1-plus-r2-vs-r}
\frac 1{r_1}+\frac 1{r_2} &\le \frac{1}{r}
\end{align}
with the final inequality strict if $\min(1/r_1,1/r_2)=0$.
\end{proposition}

\begin{proof}
Using an elementary extension and cutoff function argument, it suffices
to show 
\begin{equation}\label{eq:instead-of-domain}
||uv||_{F^{s,p}_q(\Reals^n)} \lesssim ||u||_{F^{s,p_1}_q(\Reals^n)}||v||_{F^{s,p_2}_q(\Reals^n)}
\end{equation}
whenever $u\in H^{s,p_1}(\Reals^n)$ and $v\in H^{s,p_2}(\Reals^n)$ are supported in $B_R(0)$ for some $R>0$.

Suppose $\min(1/r_1,1/r_2)\neq 0$.  
Without loss of generality, we can assume 
$1/r_1=\min(1/r_1,1/r_2)$.  Suppose
$1/r_1<0$.  Since $1/r_1\le 1/r_2\le 1/r$ it follows that
$1/p_1\le 1/p_2 \le 1/p$ and 
Lemma \ref{lem:mult-Fsp-same-s-high} together with 
Proposition \ref{prop:embedding-Fsp} implies
\[
||uv||_{F^{s,p}_q(\Reals^n)}\lesssim ||uv||_{F^{s,p_2}_q(\Reals^n)} 
\lesssim ||u||_{F^{s,p_1}_q(\Reals^n)}||v||_{H^{s,p_2}_q(\Reals^n)}.
\]

Next, assume $1/r_1>0$.  Since $1/r_1+1/r_2\le 1/r$, inequality 
\eqref{eq:q_vs_ps-H} holds and the result follows from 
Lemma \ref{lem:mult-Fsp-same-s-low}.

Finally, consider the threshold case $1/r_1=0$, so $\min(1/r_1,1/r_2)=0$.
We have therefore also assumed 
$1/r_1+1/r_2<1/r$ and can lower $p_1$ and $p_2$ slightly
to $\hat p_1$ and $\hat p_2$ so that
\begin{equation}\label{eq:hat-r-bound}
1/\hat r_1+1/\hat r_2<1/r
\end{equation}
remains true. 
Notice that $\min(1/\hat r_1,1/\hat r_2)>0$, and hence
inequality \eqref{eq:hat-r-bound} also implies 
 $\max(1/\hat r_1,1/\hat r_2)<1/r$.  We have
 therefore already established continuity
 of multiplication 
$F^{s,\hat p_1}_q(\Omega)\times F^{s,\hat p_2}_q(\Omega)\to
F^{s,p}_q(\Omega)$ and the result follows from 
the continuous embedding
$
F^{s,p_1}_q(\Omega)\times F^{s,p_2}_q(\Omega)\hookrightarrow
F^{s,\hat p_1}_q(\Omega)\times F^{s,\hat p_2}_q(\Omega).
$
\end{proof}

Using Sobolev embedding we can now relax the requirement
$s_1=s_2=s$ and $q_1=q_2=q$, noting that a fine parameter
restriction arises if $s=s_i$ for some $i$.
\begin{proposition}\label{prop:mult-Fsp-s-pos}
Assume the multiplication hypothesis \ref{as:mult-hyp}
and that $s_1,s_2,s>0$. 
Multiplication is continuous 
$F^{s_1,p_1}_{q_1}(\Omega)\times F^{s_2,p_2}_{q_2}(\Omega)
\to F^{s,p}_q(\Omega)$ so long as
\begin{align}
s&\le \min({s_1,s_2})\\
\max\left(\frac{1}{r_1},\frac{1}{r_2}\right) & \le  \frac{1}{r}\\
\label{eq:r1-r2-vs-r-Fsp-baby}
\frac{1}{r_1} + \frac{1}{r_2} & \le \frac{1}{r}
\end{align}
with the following caveats:
\begin{itemize}
	\item Inequality \eqref{eq:r1-r2-vs-r-Fsp-baby}
	is strict if $\min(1/r_1,1/r_2)=0$.
	\item If $s_i=s$ for some $i$ then $q\ge q_i$.
\end{itemize}
\end{proposition}
\begin{proof}
First, suppose $\min(1/r_1, 1/r_2) > 0$.  Define $t_1, t_2$ by
\[
\frac{1}{t_i} -\frac{s}{n} = \frac{1}{r_i} = \frac{1}{p_i} - \frac{s_i}{n}.
\]
Since $s_i\ge s \ge 0$,
\[
0< \frac{1}{r_i} \le \frac{1}{t_i} \le \frac{1}{p_i} < 1
\]
so by Sobolev embedding (Proposition \ref{prop:embedding-Fsp})
we have 
\[
F^{s_1,p_1}_{q_1}(\Omega)\times F^{s_2,p_2}_{q_2}(\Omega) \hookrightarrow
F^{s,t_1}_{q}(\Omega)\times F^{s,t_2}_q(\Omega).
\]
Note that this is the point where we used the caveat $q_i\le q$ if $s_i=s$.  
Because the value of $r_i$ is preserved in the transition from $p_i$ to $t_i$,
the hypotheses of Proposition \ref{prop:mult-same-s} are satisfied 
and we have continuity of multiplication 
$F^{s,t_1}_q(\Omega)\times F^{s,t_2}_q(\Omega)
\to F^{s,p}_{q}(\Omega)$.

Now suppose $\min(1/r_1,1/r_2)<0$; without loss of generality 
we can assume $1/r_1\le 1/r_2$ and therefore $1/r_1<0$. 
Using the hypotheses $s\le s_2$, $1/r\ge 1/r_2$ and $q_i\le q$ if $s=s_2$
we know from Sobolev embedding that 
$F^{s_2,p_2}_{q_2}(\Omega) \hookrightarrow F^{s,p}_q(\Omega)$.
Hence it suffices to prove that multiplication is continuous
\[
F^{s_1,p_1}_{q_1}(\Omega)\times F^{s,p}_{q}(\Omega) \to
F^{s,p}_{q}(\Omega).
\]
Proposition \ref{prop:mult-same-s} implies 
multiplication is continuous
\[
F^{s,t}_q(\Omega) \times F^{s,p}_{q}(\Omega) \to
F^{s,p}_{q}(\Omega)
\]
if $t\in(1,\infty)$ satisfies
\begin{equation}\label{eq:mult-into-same}
\begin{aligned}
\frac{1}{t}-\frac{s}n &\le \frac{1}p-\frac{s}n\\
\frac{1}{t}-\frac{s}n &< 0.
\end{aligned}
\end{equation}
Hence we need only show that $F^{s_1,p_1}_{q_1}(\Omega)$
embeds into $F^{s,t}_q(\Omega)$ for some $t$ satisfying
conditions \eqref{eq:mult-into-same}.

There are two cases depending on the value of 
\[
\frac{1}{\hat t} = \frac{1}{p_1} - \frac{s_1}{n} + \frac{s}{n}.
\]
If $\hat t>0$ we take $t=\hat t$ and observe that since $s_1\ge s$, $t>1$. Sobolev embedding 
(using the hypothesis $q_1\le q$ if $s_1=s$) implies 
$F^{s_1,p_1}_{q_1}(\Omega)\hookrightarrow F^{s,t}_q(\Omega)$.
Inequalities \eqref{eq:mult-into-same} follow from the
the observations $1/t-s/n = 1/p_1-s_1/n < 0$
and 
\[
\frac{1}{t} - \frac{s} n = 
\frac{1}{p_1} - \frac{s_1} n = \frac{1}{r_1} \le \frac{1}{r} = 
\frac{1}{p}-\frac{s}{n}.
\]
Suppose instead $\hat t\le 0$.
Now we simply choose any $t>1$ satisfying
conditions \eqref{eq:mult-into-same}.  
Sobolev embedding
$F^{s_1,p_1}_{q_1}(\Omega)
\hookrightarrow 
F^{s,t}_{q}(\Omega)$
now follows from the inequality $\hat t\le 0$ (noting that
this can only happen if $s_1>s$ and hence the fine
parameter plays no role).

The proposition is now proved except in the marginal case 
$\min(1/r_1,1/r_2)=0$.  In this case we have assumed
$1/r_1+1/r_2 < 1/r$ and consequently $1/r_i<1/r$, $i=1,2$.
Just as in Proposition \ref{prop:mult-same-s}
we can lower $p_1$ and $p_2$ slightly while maintaining
these strict inequalities, and the result follows from
our previous work.
\end{proof}

We now extend Proposition \ref{prop:mult-Fsp-s-pos} to 
the case $s\le 0$ while still assuming $\min(s_1,s_2)>0$.
The proof relies on the embeddings
\begin{itemize}
	\item $L^1(\Omega)\hookrightarrow F^{s,p}_q(\Omega)$ if $s<0$ and  $1/p^*+s/n<0$
	\item $L^a(\Omega)\hookrightarrow F^{s,p}_q(\Omega)$ for any $a>1$ if $s<0$ and  $1/p^*+s/n=0$
\end{itemize}
which are proved by extending functions on $\Omega$ by zero to all of $\Reals^n$
and applying duality on $\Reals^n$, Proposition \ref{prop:elliptic-zoom-Fsp}, 
along with Sobolev embedding on $\Reals^n$.

\begin{proposition}\label{prop:mult-Fsp-pos-pos-neg}
Assume the multiplication hypothesis \ref{as:mult-hyp}
and that $s_1,s_2>0$ and $s\le 0$.
Multiplication of $C^\infty(\Omega)$ 
functions extends to a continuous bilinear map 
$F^{s_1,p_1}_{q_1}(\Omega)\times F^{s_2,p_2}_{q_2}(\Omega)
\to F^{s,p}_q(\Omega)$ so long as
\begin{align}
\max\left(\frac{1}{r_1},\frac{1}{r_2}\right) & \le  \frac{1}{r}\\
\label{eq:r1-r2-vs-r-Fsp-baby-3}
\frac{1}{r_1} + \frac{1}{r_2} & \le \frac{1}{r}\\
\frac{1}{r_1} + \frac{1}{r_2} & \le 1
\end{align}
with the following caveat:
\begin{itemize}
	\item Inequality \eqref{eq:r1-r2-vs-r-Fsp-baby-3}
	is strict if $\min(1/r_1,1/r_2,1-1/r)=0$.
\end{itemize}
\end{proposition}
\begin{proof}
First, suppose $1/r<1$.  Choose $1/t\in(0,1)$ and 
$\sigma>0$ such that $\sigma<\min(s_1,s_2)$
and such that $(1/t)-(\sigma/n)=1/r$.  This is possible because
of the strict inequality $1/r<1$.  Proposition \ref{prop:mult-Fsp-s-pos}
and Sobolev embedding then ensure 
the continuity of multiplication 
\[
	F^{s_1,p_1}_{q_1}(\Omega)\times F^{s_2,p_2}_{q_2}(\Omega)
\to F^{\sigma,t}_q(\Omega) \hookrightarrow F^{s,p}_{q}(\Omega).
\]

Now suppose $1/r>1$.
Since $1/p^*+s/n = 1-1/r < 0$, from the comments before the start of the proposition it 
suffices to show show that product embeds continuously in $L^1(\Omega)$,
and the hard case occurs when $\min(1/r_1,1/r_2)>0$. But then
Sobolev embedding implies
$F^{s_i,p_i}_{q_i}(\Omega) \hookrightarrow L^{r^i}(\Omega)$.
Since $1/r_1+1/r_2\le 1/r< 1$, the result follows from H\"older's inequality.

Finally, suppose $1/r=1$. Now Sobolev embedding
implies $F^{-s,p^*}_{q^*}(\Omega)\hookrightarrow L^{t}(\Omega)$ 
for all $t<\infty$.
So it suffices to show that the product 
embeds continuously in $L^a(\Omega)$ for some $a>1$
and again the hard case
occurs when $1/r_1,1/r_2>0$.  Since $1-1/r=0$,
$\min(1/r_1,1/r_2,1-1/r)=0$ and we have hence assumed
$1/r_1+1/r_2 < 1/r =1$.   Arguing as in the case $1/r<1$,
the product lies in $L^a(\Omega)$ with $1/a = 1/r_1+1/r_2$.
The proof is complete, noting that the assumption $1/r_1+1/r_2<1$ implies 
$a>1$ as required.
\end{proof}

The previous two propositions establish Theorem \ref{thm:mult-Fsp}
if $\min(s_1,s_2)>0$.  The case $\min(s_1,s_2)<0$ follows
from a duality argument.

\begin{proposition}\label{prop:mult-Fsp-s2-neg}
Assume the multiplication hypothesis \ref{as:mult-hyp}
and that $\min(s_1,s_2)<0$.
Multiplication of $C^\infty(\Omega)$ 
functions extends to a continuous bilinear map 
$F^{s_1,p_1}_{q_1}(\Omega)\times F^{s_2,p_2}_{q_2}(\Omega)
\to F^{s,p}_q(\Omega)$ so long as
\begin{align}
s_1+s_2&\ge 0\\
s&\le \min(s_1,s_2)\\
\max\left(\frac{1}{r_1},\frac{1}{r_2}\right) & \le  \frac{1}{r}\\
\label{eq:r1-r2-vs-r-Fsp-baby-4}
\frac{1}{r_1} + \frac{1}{r_2} & \le \frac{1}{r}\\
\frac{1}{r_1} + \frac{1}{r_2} & \le 1
\end{align}
with the following caveats:
\begin{itemize}
	\item Inequality \eqref{eq:r1-r2-vs-r-Fsp-baby-4}
	is strict if $\min(1/r_1,1/r_2,1-1/r)=0$.
	\item If  $s=s_i$ for some $i$, then $q_i\le q$.
	\item If $s_1+s_2=0$ then $\frac{1}{q_1}+\frac{1}{q_2}\ge 1$.
\end{itemize}
\end{proposition}

\begin{proof}
Without loss of generality we can assume $s_1<0$, in which case
$s_2>0$.  Note moreover that $s<0$ since $s\le s_1$.

We first show continuity of multiplication
\[
F^{s_2,p_2}_{q_2}(\Omega)\times F^{-s,p^*}_{q^*}(\Omega)\to F^{-s_1,p_1^*}_{q_1^*}(\Omega).
\]
Since $s_2,-s,-s_1>0$ we need only verify the conditions 
of Proposition \ref{prop:mult-Fsp-s-pos}. 
These read
\begin{align*}
-s_1\le s_2\\
-s_1\le -s\\
\frac{1}{p_2}-\frac{s_2}{n} \le \frac{1}{p_1^*}+\frac{s_1}{n}\\
\frac{1}{p^*}+\frac{s}{n} \le \frac{1}{p_1^*}+\frac{s_1}{n} \\
\frac{1}{p_2}-\frac{s_2}{n} + \frac{1}{p^*}-\frac{-s}{n} \le
\frac{1}{p_1^*}-\frac{-s_1}{n}
\end{align*}
with the final inequality strict if $\min(\frac{1}{p_2}-\frac{s_2}{n},
\frac{1}{p^*}-\frac{-s}{n})=0$ and additionally
\begin{align*}
	q^*&\le q_1^*\quad\text{if $-s=-s_1$},\\
	q_2&\le q_1^* \quad\text{if $s_2=-s_1$.}
\end{align*}
These conditions can be rewritten
\begin{align*}
0 \le s_1+s_2\\
s\le s_1 = \min(s_1,s_2)\\
\frac{1}{r_2} \le 1 - \frac{1}{r}\\
1-\frac{1}{r} \le 1 - \frac{1}{r_1}\\
\frac{1}{r_2} + 1 - \frac{1}{r} \le 1-\frac{1}{r_1}.
\end{align*}
with the final inequality strict if $\min(1/r_2,1-1/r)=0$
and additionally
\begin{align*}
q&\ge q_1 \quad\text{if $s=s_1$},\\
\frac{1}{q_1}+\frac{1}{q_2} &\ge 1\quad\text{if $s_1+s_2=0$.}
\end{align*}
Note that since $1/r_1>0$, $\min(1/r_2,1-1/r)=0$
if and only if $\min(1/r_1,1/r_2,1-1/r)=0$. 
Hence we have assumed all of these conditions.

The result now follows from a duality argument.  Suppose $u_i\in F^{s_i,p_i}_{q_i}(\Omega)$,
$i=1,2$. We define an element $z\in F^{s,p}_q(\Omega)$ as follows.
Let $\tilde u_i$ be an extension to $\Reals^n$ that has support in some large
ball $\overline{B_R(0)}$ independent of the functions $u_i$, such that 
$||\tilde u_i||_{F^{s_i,p_i}_{q_i}(\Reals^n)}
\lesssim ||u_i||_{F^{s_i,p_i}_{q_i}(\Omega)}$. 
Given $w\in F^{-s,p^*}_{q^*}(\Reals^n)$, the product 
$\tilde u_2 w$ is an element of $F^{-s_1,p_1^*}_{q_1^*}(B_{2R(0)})$ by the argument above,
and because of the support of $\tilde u_2$ in $\overline{B_R(0)}$, 
it extends continuously by zero to an element of $F^{-s_1,p_1^*}_{q_1^*}(\Reals^n)$.
Using duality pairing on $\Reals^n$ we define $f(w) = \left<\tilde u_2 w, \tilde u_1\right>$
and one readily verifies  the estimate
\begin{equation}\label{eq:duality-grind}
	|f(w)|\lesssim ||u_1||_{F^{s_1,p_1}_{q_1}(\Omega)}||u_2||_{F^{s_2,p_2}_{q_2}(\Omega)}||w||_{F^{-s,p^*}_{q^*}(\Reals^n)}
\end{equation}
and hence $f$ determines an element of $F^{s,p}_{q}(\Reals^n)$. Let $z$ be its restriction to $\Omega$.
A routine computation shows that $z$ is independent of the choice of extensions, 
depends bilinearly (and via \eqref{eq:duality-grind} continuously) on the factors $u_i$,
and that if $u_1$ and $u_2$ are smooth, then $z$ is simply the product $u_1 u_2$.
\end{proof}

It remains to establish 
Theorem \ref{thm:mult-Fsp} when $\min(s_1,s_2)=0$. 

\begin{proposition}\label{prop:mult-Fsp-s2-zero}
Assume the multiplication hypothesis \ref{as:mult-hyp}
that $\min(s_1,s_2)=0$.  
Multiplication is continuous 
$F^{s_1,p_1}_{q_1}(\Omega)
\times F^{s_2,p_2}_{q_2}(\Omega) \to F^{s,p}_q(\Omega)$
so long as
\begin{align}
\max\left(\frac{1}{r_1},\frac{1}{r_2}\right) & \le  \frac{1}{r}\\
\label{eq:r1-r2-vs-r-Fsp-baby-9}
\frac{1}{r_1} + \frac{1}{r_2} & \le \frac{1}{r}
\end{align}
with the following caveats:
\begin{itemize}
	\item Inequality \eqref{eq:r1-r2-vs-r-Fsp-baby-9}
	is strict if $\min(1/r_1,1/r_2,1-1/r)=0$.
	\item If $s_i=s$ for some $i$ then $q\ge q_i$.
	\item If $s_1+s_2=0$ then $\frac1{q_1}+\frac1{q_2} \ge 1$.
	\item If $s_1=s_2=s=0$ then $q_1,q_2\le 2$ and $q\ge 2$.
\end{itemize}
\end{proposition}
\begin{proof}
Without loss of generality we can assume $s_1 \ge s_2 =0$.  

Suppose first that $s_1>0$ and $s<0$.  By a duality argument analogous
to the one at the end of Proposition \ref{prop:mult-Fsp-s2-neg}
it suffices to 
show that multiplication is continuous 
$F^{s_1,p_1}_{q_1}(\Omega)\times F^{-s,p^*}_{q^*}(\Omega)\to
F^{0,p_2^*}_{q_2^*}(\Omega)$.
Proposition \ref{prop:mult-Fsp-pos-pos-neg} ensures this
is possible so long as 
\[
\frac{1}{r_1} \le \frac{1}{p_2^*},\quad
\frac{1}{p^*}+\frac{s}{n}\le \frac{1}{p_2^*},\quad
\frac{1}{r_1} + \frac{1}{p^*}+\frac{s}{n},\quad
\frac{1}{r_1} + \frac{1}{p^*}+\frac{s}{n} \le \frac{1}{p_2^*}
\]
with the final inequality strict if $\min(1/r_1,1/p^*+s/n,1-1/p_2^*)=0$.
But these are equivalent to
\[
\frac{1}{r_1} +\frac{1}{r_2}\le 1,\quad
\frac{1}{r_2}\le \frac{1}{r},\quad
\frac{1}{r_1}\le \frac{1}{r},\quad
\frac{1}{r_1} +\frac{1}{r_2}\le \frac 1r
\]
with the final inequality strict if $\min(1/r_1,1-1/r,1/r_2)=0$,
which were all assumed.

Now consider the case $s_1>0$ but $s=0$. 
Pick $\epsilon>0$ so that $\epsilon < s_1$ and so that
\begin{alignat*}{2}
\frac{1}{t_1} &:=\frac{1}{r_2} +\frac \epsilon n < 1&\qquad
\frac{1}{t_2} &:= \frac{1}{r_2} -\frac \epsilon n > 0\\
\frac{1}{\tau_1} &:= \frac{1}{r} +\frac \epsilon n < 1&\qquad
\frac{1}{\tau_2} &:= \frac{1}{r} -\frac \epsilon n > 0
\end{alignat*}
This collection of inequalities can be satisfied because 
$1/r_2 = 1/p_2 \in (0,1)$ and because
$0< 1/r_2\le 1/r=1/p<1$.
An easy computation shows that 
Proposition \ref{prop:mult-Fsp-s-pos} ensures continuity of multiplication
\[
F^{s_1,p_1}_{q_1}(\Omega)\times F^{\epsilon,t_1}_{q_2}(\Omega)
\to F^{\epsilon,\tau_1}_{q}(\Omega);
\]
note that this uses the hypothesis $q\ge q_2$ 
which we have assumed since $s_2=s=0$.  
Similarly, Proposition \ref{prop:mult-Fsp-s2-neg} ensures
continuity of multiplication 
\[
F^{s_1,p_1}_{q_1}(\Omega)\times F^{-\epsilon,t_2}_{q_2}(\Omega)
\to F^{-\epsilon,\tau_2}_{q}(\Omega).
\]
The result now follows from interpolation, noting that
$\frac{1}{2}(\frac{1}{t_1}+\frac{1}{t_2})=\frac{1}{r_2}=\frac{1}{p_2}$
and $\frac{1}{2}(\frac{1}{\tau_1}+\frac{1}{\tau_2})=\frac{1}{r}=\frac{1}{p}$.

Finally suppose $s_1=0$ and $s<0$. 
By duality it suffices to prove multiplication is continuous
\[
F^{0,p_1}_{q_1}(\Omega) \times F^{-s,p^*}_{q^*}(\Omega)\to F^{0,p_2^*}_{q_2^*}(\Omega).
\]
This follows from the case just considered, noting that we 
pick up the requirement, $q_1\le q_2^*$, which is equivalent to
$1/q_1+1/q_2\ge 1$, which we have assumed since $s_1+s_2=0$.

All that remains is the case $s_1=s_2=s=0$.
We have the obvious consequence
of H\"older's inequality:
\[
F^{0,p_1}_{q_1}(\Omega)\times F^{0,p_2}_{q_2}(\Omega)
\hookrightarrow L^{p_1}(\Omega)\times L^{p_2}(\Omega)
\to L^{p}(\Omega) \hookrightarrow F^{0,p}_{q}(\Omega).
\]
if $q_1, q_2\le 2$, $q\ge 2$, and $1/p_1+1/p_2\le 1/p$.
It is perhaps surprising that this cannot be improved 
(\cite{sickel_holder_1995} Corollary 4.3.1(ii)).
\end{proof}

%%%%%%%%%%%%%%%%%%%%%%%%%%%%%%%%%%%%%%%%%%%%%%%%%%%%%%%%%%%%%%%%
%%%%%%%%%%%%%%%%%%%%%%%%%%%%%%%%%%%%%%%%%%%%%%%%%%%%%%%%%%%%%%%%
\section{Multiplication in Besov spaces}
\label{app:mult-Bsp}
%%%%%%%%%%%%%%%%%%%%%%%%%%%%%%%%%%%%%%%%%%%%%%%%%%%%%%%%%%%%%%%%
%%%%%%%%%%%%%%%%%%%%%%%%%%%%%%%%%%%%%%%%%%%%%%%%%%%%%%%%%%%%%%%%
We prove the following multiplication theorem for Besov spaces
following a strategy similar to that of Appendix \ref{app:mult-Fsp}.

\begin{reptheorem}{thm:mult-Besov}
	Let $\Omega$ be a bounded open subset of $\Reals^n$.
	Suppose $1<p_1,p_2,p,q_1,q_2,q<\infty$ and $s_1,s_2,s\in\Reals$.  Let $r_1,r_2$ and $r$ be defined by
\[
\frac{1}{r_1} = \frac{1}{p_1} - \frac{s_1}{n},\qquad
\frac{1}{r_2} = \frac{1}{p_2} - \frac{s_2}{n},\quad\text{\rm and}\quad
\frac{1}{r} = \frac{1}{p} - \frac{s}{n}.
\]
Pointwise multiplication of $C^\infty(\overline \Omega)$ 
functions extends to a continuous bilinear map
$\Bv[s_1,p_1,q_1](\Omega)\times \Bv[s_2,p_2,q_2](\Omega)
\to \Bv[s,p,q](\Omega)$ so long as
\begin{align}
\min(s_1,s_2)&\ge s\\
s_1+s_2&\ge 0\\
\max\left(\frac 1{r_1},\frac 1{r_2}\right) & \le  \frac 1 r\\
\label{eq:r1-r2-vs-1-Besov-baby-5-app}
\frac{1}{r_1} + \frac{1}{r_2} & \le 1\\
\label{eq:r1-r2-vs-r-Besov-baby-5-app}
\frac{1}{r_1} + \frac{1}{r_2} & \le \frac{1}{r}
\end{align}
with the following caveats:
\begin{itemize}
	\item If $s=s_i$ or $1/r=1/r_i$ for some $i$ then $1/q\le 1/q_i$.
	\item If $s_1+s_2=0$ or $1/r_1+1/r_2=1$ then $1/q_1+1/q_2\ge1$.
	\item If equality holds in \eqref{eq:r1-r2-vs-r-Besov-baby-5-app} then
	\begin{itemize}
	\item $\min(1/r_1,1/r_2,1-1/r)\neq0$.	
	\item If $\min(s_1,s_2)\le0$ then $1/q_1+1/q_2\ge1$ and $1/q \le 1/r$.
    \item If $s_i$ has the same sign as $\min(s_1,s_2)$ for some $i$ then $1/q\le 1/q_i$.
    \item If $s_i$ has the same sign as $\max(s_1,s_2)$ for some $i$ then $1/r_i\le 1/q_i$.
	\item If $s=0$ then $1/q\le 1/q_i$ and $1/r_i\le 1/q_i$ for both $i=1,2$.
	\end{itemize}
	\item If $s_1=s_2=s=0$ then $\displaystyle \frac 1q\le \min\left(\frac 12,\frac 1r\right)$ and 
	$\displaystyle \max\left(\frac 12,\frac 1{r_i}\right)\le \frac 1{q_i}$ for both $i=1,2$.
\end{itemize}
\end{reptheorem}

The large number of new caveats in the edge cases compared to
those of Triebel-Lizorkin multiplication 
is a consequence of two phenomena 
related to Besov space embedding on a bounded domain $\Omega$.
First, recall from Proposition \ref{prop:embedding-Bsp}
that if $\frac1{p_1}-\frac{s_1}n\le \frac1{p_2}-\frac{s_2}n$ and $s_1\ge s_2$ then
\begin{equation*}
\Bv[s_1,p_1,q_1](\Omega)\hookrightarrow \Bv[s_2,p_2,q_2](\Omega)
\end{equation*}
just as for Triebel-Lizorkin embedding, 
except the marginal case $\frac1{p_1}-\frac{s_1}n = \frac1{p_2}-\frac{s_2}n$
requires additionally $q_1\leq q_2$. 
Second, we require embeddings of Besov spaces into Lebesgue spaces $L^p(\Omega)$,
which are less straightforward than the Triebel-Lizorkin setting because
Lebesgue spaces are not generically Besov spaces.  
\begin{proposition}\label{prop:B-L-embedding}
Let $\Omega$ be an open set in $\Reals^n$.
Suppose $1<p,q,r<\infty$ and $s>0$.
\begin{enumerate}
\item\label{part:B-into-L} If  $\displaystyle \frac1p-\frac sn = \frac 1r$ and if $\displaystyle \frac1q\ge \frac1r$ then $\Bv[s,p,q](\Omega) \hookrightarrow L^r(\Omega)$.
\item\label{part:B0-into-L} If $q\le 2$ and $q\le p$ then $\Bv[0,p,q](\Omega)\hookrightarrow L^p(\Omega)$.
\item\label{part:L-into-B0} If $q\ge 2$ and $q\ge p$ then $L^p(\Omega)\hookrightarrow \Bv[0,p,q](\Omega)$.
\end{enumerate}
\end{proposition}
If $\Omega=\Reals^n$ then part \ref{part:B-into-L} follows from
\cite{franke_spaces_1986} Theorem 1
and the remaining parts are a consequence of \cite{triebel_theory_2010} Proposition 2.3.2/2.
The same facts remain true for arbitrary open sets by the usual extension/restriction argument.

We now proceed with the sequence of results that prove Theorem \ref{thm:mult-Besov}.
The following two lemmas are analogs of Lemmas \ref{lem:mult-Fsp-same-s-high} and \ref{lem:mult-Fsp-same-s-low}
and are the technical foundation of the remainder of the appendix.
\begin{lemma}\label{lem:mult-Besov-same-s-high}
Suppose $1<p\le p_1<\infty$, 
$1<q<\infty$ and $s>n/p_1$.
If $u\in \Bv[s,p_1,q](\Reals^n)$ and $v\in \Bv[s,p,q](\Reals^n)$ are
both supported in $B_R(0)$ for some $R>0$ then $uv\in \Bv[s,p,q](\Reals^n)$
and 
\[
||uv||_{\Bv[s,p,q]}\lesssim ||u||_{\Bv[s,p_1,q]}||v||_{\Bv[s,p,q]}.
\]
The implicit constant depends on $s$, $p_1$, $p$, $q$ and $R$ but
is independent of $u$ and $v$.
\end{lemma}

\begin{proof}
By an obvious modification of Proposition \ref{prop:littlewood-payley-review}\eqref{part:sobolev-space-by-projection}
\[
||uv||_{\Bv[s,p,q]} \lesssim \
\left\|uv\right\|_{L^p(\Reals^n)}
+ \left( \sum_{k\ge 10} 2^{sqk}\|P_k(uv)\|_{L^p(\Reals^n)}^q\right)^{\frac1q}
\]
and we follow the pattern of Lemma \ref{lem:mult-Fsp-same-s-high} to bound the right-hand side of this inequality.

Since $s>n/p_1$,
the low frequency part admits the bound
\[
\|uv\|_{L^p(\Reals^n)}
\lesssim \|u\|_{L^\infty(\Reals^n)} \|v\|_{L^p(\Reals^n)} 
\lesssim \|u\|_{\Bv[s,p_1,q](\Reals^n)} \|v\|_{\Bv[s,p,q](\Reals^n)}.
\]
For the high-frequency part we define $\widetilde P_{k} = P_{k-3\le\cdot\le k+3}$ 
and observe that the Littlewood-Paley trichotomy of
Theorem \ref{prop:littlewood-payley-review}\eqref{part:trichotomy} implies
\[
P_k(uv) = P_k\left(
\underbrace{(P_{\le k-4} u) (\widetilde P_{k}v)}_{\text{low-high}}
+ \underbrace{(\widetilde P_{k}u) (P_{\le k+5}v)}_{\text{high-low}}
+ \underbrace{ 
	\sum_{k'\ge k+4}
(P_{k'} u) (\widetilde P_{k'} v)}_{\text{high-high}}\right).
\]

For the low-high term, we use Proposition \ref{prop:littlewood-payley-review} parts 
\eqref{part:throw-away-Pk} and \eqref{part:maximal-function-beats-projection} 
along with the Hardy-Littlewood maximal inequality to compute
\begin{equation*}
\begin{split}
\|P_k\left ( (P_{\le k-4} u) (\widetilde P_{k}v)\right)\|_{L^p(\Reals^n)}
\lesssim
\|(P_{\le k-4} u) (\widetilde P_{k}v)\|_{L^p(\Reals^n)}
&\le
\|P_{\le k-4} u\|_{L^\infty(\Reals^n)} \|\widetilde P_{k}v\|_{L^p(\Reals^n)}\\
&\lesssim
\|Mu\|_{L^\infty(\Reals^n)} \|\widetilde P_{k}v\|_{L^p(\Reals^n)}\\
&\lesssim
\|u\|_{L^\infty(\Reals^n)} \|\widetilde P_{k}v\|_{L^p(\Reals^n)}\\
&\lesssim
\|u\|_{\Bv[s,p_1,q](\Reals^n)} \|\widetilde P_{k}v\|_{L^p(\Reals^n)}
\end{split}
\end{equation*}
where $M$ is, as usual, the maximal operator.
Summing over $k$, we get the desired bound
\begin{equation*}
\begin{split}
\sum_{k\ge 10} 2^{sqk}\|P_k\left( (P_{\le k+5} u) (\widetilde P_{k}v)\right)\|_{L^p(\Reals^n)}^q
&\lesssim
\|u\|_{\Bv[s,p_1,q](\Reals^n)}^q \sum_{k\ge 10} 2^{sqk} \|\widetilde P_{k}v\|_{L^p(\Reals^n)}^q
\lesssim
\|u\|_{\Bv[s,p_1,q](\Reals^n)}^q \|v\|_{\Bv[s,p,q](\Reals^n)}^q.
\end{split}
\end{equation*}
Similarly, for the high-low term, we have
\begin{equation*}
\begin{split}
\|P_k\left((\widetilde P_{k}u)(P_{\le k+5} v)\right)\|_{L^p(\Reals^n)}
\lesssim
\|(\widetilde P_{k}u)(P_{\le k+5} v)\|_{L^p(\Reals^n)}
&\le
\|\widetilde P_{k}u\|_{L^{p_1}(\Reals^n)} \|P_{\le k+5}v\|_{L^{t}(\Reals^n)}\\
&\lesssim
\|\widetilde P_{k}u\|_{L^{p_1}(\Reals^n)} \|Mv\|_{L^{t}(\Reals^n)}\\
&\lesssim
\|\widetilde P_{k}u\|_{L^{p_1}(\Reals^n)} \|v\|_{L^{t}(\Reals^n)}
\end{split}
\end{equation*}
where $t>1$ is defined by
\begin{equation*}
\frac1{p_1}+\frac1t=\frac1p .
\end{equation*}
Then summing over $k$ gives
\begin{equation}\label{eq:besov-high-low}
\begin{split}
\sum_{k\ge 10} 2^{sqk} \|P_k\left((\widetilde P_{k}u)(P_{\le k+5} v\right)\|_{L^p(\Reals^n)}^q
&\lesssim
\|v\|_{L^t(\Reals^n)}^q \sum_{k\ge 10} 2^{sqk} \|\widetilde P_{k}u\|_{L^{p_1}(\Reals^n)}^q
\lesssim
\|v\|_{L^t(\Reals^n)}^q \|u\|_{\Bv[s,p_1,q](\Reals^n)}^q .
\end{split}
\end{equation}
Using the fact that $v$ has support on a ball of fixed radius we control $\|v\|_{L^t(\Reals^n)}$.
If $s>n/p$
\begin{equation*}
\|v\|_{L^t(\Reals^n)} 
\lesssim
\|v\|_{L^\infty(\Reals^n)} 
\lesssim
\|v\|_{\Bv[s,p_1,q](\Reals^n)} ,
\end{equation*}
and otherwise, since
\begin{equation*}
\frac1t=\frac1p-\frac1{p_1}>\frac1p-\frac{s}n\ge 0,
\end{equation*}
Proposition \ref{prop:B-L-embedding} along with Sobolev embedding 
for functions of bounded support to
lower $p$ and suitably adjust $q$ implies
\begin{equation*}
\|v\|_{L^t(\Reals^n)} 
\lesssim
\|v\|_{\Bv[s,p,q](\Reals^n)}.
\end{equation*}
In either case find that the left-hand side of inequality \eqref{eq:besov-high-low} is bounded by 
$||u||_{B^{s,p_1}_q(\Reals^n)}^q||v||_{B^{s,p}_q(\Reals^n)}^q$

Finally, turning to the high-high term, we start with
\begin{equation*}
\begin{split}
\left\|P_k\sum_{k'\ge k+4}( P_{k'} u) (\widetilde P_{k'} v)\right\|_{L^p(\Reals^n)}
\lesssim
\left\|\sum_{k'\ge k+4}(P_{k'} u) (\widetilde P_{k'} v)\right\|_{L^p(\Reals^n)}
&\le
\sum_{k'\ge k+4}\| P_{k'}u\|_{L^\infty(\Reals^n)} \|\widetilde P_{k'}v\|_{L^{p}(\Reals^n)} \\
&\lesssim
\sum_{k'\ge k+4}\|Mu\|_{L^\infty(\Reals^n)} \|\widetilde P_{k'}v\|_{L^{p}(\Reals^n)} \\
&\lesssim
\|u\|_{\Bv[s,p_1,q](\Reals^n)} \sum_{k'\ge k+4} \|\widetilde P_{k'}v\|_{L^{p}(\Reals^n)}.
\end{split}
\end{equation*}
Now sum over $k$ to get
\begin{equation*}
\begin{split}
\sum_{k\ge 10} 2^{sqk} \left\|P_k\sum_{k'\ge k+4}(P_{k'} u) (\widetilde P_{k'} v)\right\|_{L^p(\Reals^n)}^q
&\lesssim
\|u\|_{\Bv[s,p_1,q](\Reals^n)}^q \sum_{k\ge 10} \left( \sum_{k'\ge k+4} 2^{sk}\|\widetilde P_{k'}v\|_{L^{p}(\Reals^n)} \right)^q.
\end{split}
\end{equation*}
and observe that the desired estimate holds if we can show
\begin{equation}\label{eq:pre-Jensen}
\sum_{k\ge 10} \left( \sum_{k'\ge k+4} 2^{sk}\|\widetilde P_{k'}v\|_{L^{p}(\Reals^n)} \right)^q\lesssim ||v||_{B^{s,p}_q(\Reals^n)}^q.
\end{equation}
However, setting $b_k = 2^{sk}\|\widetilde P_{k}v\|_{L^{p}(\Reals^n)}$ for $k\ge 1$ we have
\[
\sum_{k\ge 10} \left( \sum_{k'\ge k+4} 2^{sk}\|\widetilde P_{k'}v\|_{L^{p}(\Reals^n)} \right)^q
= 
\sum_{k\ge 10} \left( \sum_{a \ge 4} 2^{-sa} 2^{s(k+a)}\|\widetilde P_{k+a}v\|_{L^{p}(\Reals^n)} \right)^q
=\sum_{k\ge 10} \left( \sum_{a \ge 4} 2^{-sa} b_{k+a} \right)^q.
\]
Since $\sum_{a\ge 4} 2^{-sa}$ is finite we can apply Jensen's inequality to conclude
\[
\sum_{k\ge 10} \left( \sum_{a \ge 4} 2^{-sa} b_{k+a} \right)^q \lesssim 
\sum_{k\ge 10} \sum_{a \ge 4} 2^{-sa} b_{k+a}^q  = 
\sum_{a\ge 4} \sum_{k\ge 10} 2^{-sa} b_{k+a}^q \le ||b||_{\ell_q}^q \sum_{a\ge 4} 2^{-sa}.
\]
This final sum is finite and $||b||_{\ell_q}\lesssim ||v||_{B^{s,p_2}_q(\Reals^n)}$, 
which leads to inequality \eqref{eq:pre-Jensen}.
\end{proof}

\begin{lemma}\label{lem:mult-Besov-same-s-low}
Suppose $1<p<p_2\le p_1<\infty$, 
$1<q_1,q_2<\infty$ and $0<s<n/p_1$.
Suppose moreover that 
\begin{equation}\label{eq:q_vs_ps}
\frac{1}{p} \ge \frac{1}{p_1} + \frac{1}{p_2} -\frac{s}{n} ,
\end{equation}
and if \eqref{eq:q_vs_ps} is an equality assume additionally
\begin{equation}
\frac{1}{q_1} \ge \frac{1}{p_1} - \frac{s}{n} ,
\qquad\textrm{and}\qquad
\frac{1}{q_2} \ge \frac{1}{p_2} - \frac{s}{n} .
\end{equation}
If $u\in \Bv[s,p_1,q_1](\Reals^n)$ and $v\in \Bv[s,p_2,q_2](\Reals^n)$ are
both supported in $B_R(0)$ for some $R>0$ then $uv\in \Bv[s,p,q](\Reals^n)$ with $q=\max\{q_1,q_2\}$,
and 
\[
||uv||_{\Bv[s,p,q]}\lesssim ||u||_{\Bv[s,p_1,q_1]}||v||_{\Bv[s,p_2,q_2]} .
\]
The implicit constant depends on $s$, $p_1$, $p_2$, $p$, $q_1$, $q_2$ and $R$ but
is independent of $u$ and $v$.
\end{lemma}

\begin{proof}
We follow the familiar pattern.
For convenience, recall the bound
\[
||uv||_{\Bv[s,p,q](\Reals^n)} \lesssim
\left\|uv\right\|_{L^p(\Reals^n)}
+ \left( \sum_{k\ge 10} 2^{sqk}\|P_k(uv)\|_{L^p(\Reals^n)}^q\right)^{\frac1q} ,
\]
and the Littlewood-Paley trichotomy
\[
P_k(uv) = P_k\left(
\underbrace{(P_{\le k-4} u) (\widetilde P_{k}v)}_{\text{low-high}}
+ \underbrace{(\widetilde P_{k}u) (P_{\le k+5}v)}_{\text{high-low}}
+ \underbrace{\sum_{k'\ge k+4}
(P_{k'} u) (\widetilde P_{k'} v)}_{\text{high-high}}\right)
\]
where $\widetilde P_k = P_{k-3\le\cdot\le k+3}$.
For the low frequency part, H\"older's inequality 
followed by Sobolev embedding (using the bounded support of $u$) 
and Proposition \ref{prop:B-L-embedding} yields
\[
\|uv\|_{L^p(\Reals^n)}
\lesssim \|u\|_{L^{t_1}(\Reals^n)} \|v\|_{L^{p_2}(\Reals^n)} 
\lesssim \|u\|_{\Bv[s,p_1,q_1](\Reals^n)} \|v\|_{\Bv[s,p_2,q](\Reals^n)} ,
\]
where $t_1\in(1,\infty)$ is defined by
\begin{equation*}
\frac1{t_1}=\frac1p-\frac1{p_2}\ge\frac1{p_1}-\frac{s}{n} ,
\end{equation*}
with the inequality strict if \eqref{eq:q_vs_ps} is strict.
If \eqref{eq:q_vs_ps} is an equality, then we need $q_1\leq t_1$,
since the embedding $\Bv[s,p_1,q_1](\Reals^n)\hookrightarrow L^{t_1}(\Reals^n)$ becomes borderline.

For the low-high term, applying now-familiar facts from Proposition \ref{prop:littlewood-payley-review} we have
\begin{equation*}
\begin{split}
\|P_k\left((P_{\le k-4} u) (\widetilde P_{k}v)\right)\|_{L^p(\Reals^n)}
\lesssim
\|(P_{\le k-4} u) (\widetilde P_{k}v)\|_{L^p(\Reals^n)}
&\le
\|P_{\le k-4} u\|_{L^{t_1}(\Reals^n)} \|\widetilde P_{k}v\|_{L^{p_2}(\Reals^n)}\\
&\lesssim
\|Mu\|_{L^{t_1}(\Reals^n)} \|\widetilde P_{k}v\|_{L^{p_2}(\Reals^n)}\\
&\lesssim
\|u\|_{L^{t_1}(\Reals^n)} \|\widetilde P_{k}v\|_{L^{p_2}(\Reals^n)} \\
&\lesssim
\|u\|_{\Bv[s,p_1,q_1](\Reals^n)} \|\widetilde P_{k}v\|_{L^{p_2}(\Reals^n)} ,
\end{split}
\end{equation*}
and summing over $k$ gives
\begin{equation*}
\begin{split}
\sum_{k\ge 10} 2^{sqk}\|P_k(P_{\le k+5} u) (\widetilde P_{k}v)\|_{L^p(\Reals^n)}^q
&\lesssim
\|u\|_{\Bv[s,p_1,q_1](\Reals^n)}^q \sum_{k\ge 0} 2^{sqk} \|\widetilde P_{k}v\|_{L^{p_2}(\Reals^n)}^q
\lesssim
\|u\|_{\Bv[s,p_1,q_1](\Reals^n)}^q \|v\|_{\Bv[s,p_2,q](\Reals^n)}^q .
\end{split}
\end{equation*}
Similarly, for the high-low term, we have
\begin{equation*}
\begin{split}
\|P_k\left((\widetilde P_{k}u)(P_{\le k+5} v)\right)\|_{L^p(\Reals^n)}
\lesssim
\|(\widetilde P_{k}u)(P_{\le k+5} v)\|_{L^p(\Reals^n)}
&\le
\|\widetilde P_{k}u\|_{L^{p_1}(\Reals^n)} \|P_{\le k+5}v\|_{L^{t_2}(\Reals^n)}\\
&\lesssim
\|\widetilde P_{k}u\|_{L^{p_1}(\Reals^n)} \|Mv\|_{L^{t_2}(\Reals^n)}\\
&\lesssim
\|\widetilde P_{k}u\|_{L^{p_1}(\Reals^n)} \|v\|_{L^{t_2}(\Reals^n)} \\
&\lesssim
\|\widetilde P_{k}u\|_{L^{p_1}(\Reals^n)} \|v\|_{\Bv[s,p_2,q_2](\Reals^n)} ,
\end{split}
\end{equation*}
where $t_2\in (1,\infty)$ is defined by
\begin{equation*}
\frac1{t_2}=\frac1p-\frac1{p_1}\ge\frac1{p_2}-\frac{s}{n} ,
\end{equation*}
with the inequality strict if \eqref{eq:q_vs_ps} is strict.
If \eqref{eq:q_vs_ps} is an equality, then we need $q_2\leq t_2$,
since the embedding $\Bv[s,p_2,q_2](\Reals^n)\hookrightarrow L^{t_2}(\Reals^n)$ again becomes borderline.
Summing over $k$ yields
\begin{equation*}
\begin{split}
\sum_{k\ge 10} 2^{sqk} \|P_k\left((\widetilde P_{k}u)(P_{\le k+5} v)\right)\|_{L^p(\Reals^n)}^q
&\lesssim
\|v\|_{\Bv[s,p_2,r_2](\Reals^n)}^q \sum_{k\ge 10} 2^{sqk} \|\widetilde P_{k}u\|_{L^{p_1}(\Reals^n)}^q
\lesssim
\|v\|_{\Bv[s,p_2,r_2](\Reals^n)}^q \|u\|_{\Bv[s,p_1,q](\Reals^n)}^q .
\end{split}
\end{equation*}

Finally, turning to the high-high term, start with
\begin{equation*}
\begin{split}
\left\|P_k\sum_{k'\ge k+4}(P_{k'} u) (\widetilde P_{k'} v)\right\|_{L^p(\Reals^n)}
\lesssim
\left\|\sum_{k'\ge k+4}(P_{k'} u) (\widetilde P_{k'} v)\right\|_{L^p(\Reals^n)}
&\le
\sum_{k'\ge k+4}\| P_{k'}u\|_{L^{t_1}(\Reals^n)} \|\widetilde P_{k'}v\|_{L^{p_2}(\Reals^n)}\\ 
&\lesssim
\sum_{k'\ge k+4}\|Mu\|_{L^{t_1}(\Reals^n)} \|\widetilde P_{k'}v\|_{L^{p_2}(\Reals^n)} \\
&\lesssim
\|u\|_{\Bv[s,p_1,q_1](\Reals^n)} \sum_{k'\ge k+4} \|\widetilde P_{k'}v\|_{L^{p_2}(\Reals^n)} ,
\end{split}
\end{equation*}
and sum over $k$ to get
\begin{equation}\label{eq:pre-Jensen-v2}
\begin{split}
\sum_{k\ge 10} 2^{sqk} \left\|P_k\sum_{k'\ge k+4}( P_{k'} u) (\widetilde P_{k'} v)\right\|_{L^p(\Reals^n)}^q
\lesssim
\|u\|_{\Bv[s,p_1,q_1](\Reals^n)}^q \sum_{k\ge 10} 2^{sqk} \left( \sum_{k'\ge k+4} \|\widetilde P_{k'}v\|_{L^{p_2}(\Reals^n)} \right)^q.
\end{split}
\end{equation}
Moreover, the same argument as at the end of Lemma \ref{lem:mult-Besov-same-s-high} shows 
\[
\sum_{k\ge 10} 2^{sqk} \left( \sum_{k'\ge k+4} \|\widetilde P_{k'}v\|_{L^{p_2}(\Reals^n)} \right)^q	\lesssim ||v||_{\Bv[s,p_2,q](\Reals^n)}
\]
and it follows that the left-hand side of inequality \eqref{eq:pre-Jensen-v2} is controlled by 
$||u||_{\Bv[s,p_1,q](\Reals^n)}^q||v||_{\Bv[s,p_2,q](\Reals^n)}^q$.

At this point we have shown
\begin{equation*}
\|uv\|_{\Bv[s,p,q](\Reals^n)}
\lesssim
\|u\|_{\Bv[s,p_1,q]v} \|v\|_{\Bv[s,p_2,q](\Reals^n)} 
\end{equation*}
which, combined with the embeddings 
$\Bv[s,p_1,q_1](\Reals^n)\hookrightarrow \Bv[s,p_1,q](\Reals^n)$
and $\Bv[s,p_2,q_2](\Reals^n)\hookrightarrow \Bv[s,p_2,q](\Reals^n)$,
establishes the proof.
\end{proof}

The following result consolidates the previous two lemmas and 
applies to bounded smooth domains $\Omega$ rather than $\Reals^n$.

\begin{proposition}\label{prop:mult-Besov-same-s}
Assume the multiplication hypothesis \ref{as:mult-hyp}
and that $s_1=s_2=s>0$.
Pointwise multiplication of $C^\infty(\overline \Omega)$ 
functions extends to a continuous bilinear map 
$\Bv[s,p_1,q_1](\Omega)\times \Bv[s,p_2,q_2](\Omega)
\rightarrow \Bv[s,p,q](\Omega)$
so long as
\begin{align}
\label{eq:same-s-r1-r2-vs-r-Besov}
\max\left(\frac 1{r_1},\frac 1{r_2}\right) &\le \frac{1}{r}\\
\label{eq:same-s-r1-plus-r2-vs-r-Besov}
\frac 1{r_1}+\frac 1{r_2} &\le \frac{1}{r}\\
\label{eq:same-s-q1q2-vs-q-Besov}
\min\left(\frac 1 {q_1},\frac 1 {q_2}\right)&\ge \frac{1}{q}
\end{align}
with the following caveats:
\begin{itemize}
	\item Inequality \eqref{eq:same-s-r1-plus-r2-vs-r-Besov} is strict if $\min(1/r_1,1/r_2)=0$.
	\item If \eqref{eq:same-s-r1-plus-r2-vs-r-Besov} is an equality, then $1/r_1\le 1/q_1$ and $1/r_2\le 1/ q_2$.
\end{itemize}
\end{proposition}

\begin{proof}
Using the argument at the start of Proposition \ref{prop:mult-same-s} 
it suffices to show
\begin{equation}\label{eq:same-s-est-Besov}
	||uv||_{\Bv[s,p,q](\Reals^n)} \lesssim 
	||u||_{\Bv[s_1,p_1,q_1](\Reals^n)}
	||v||_{\Bv[s_2,p_2,q_2](\Reals^n)}
\end{equation}
whenever $u\in \Bv[s_1,p_1,q](\Reals^n)$ and $v\in \Bv[s_2,p_2,q](\Reals^n)$ are supported in
some ball of radius $R$ large enough to contain $\overline \Omega$.
Since the conditions are symmetric with respect to the indices 1 and 2, without loss of generality, assume that 
$p_2\leq p_1$.
We split the proof into 3 cases.

First, consider the case $s>n/p_1$, that is, $1/r_1<0$.
Thus \eqref{eq:same-s-r1-r2-vs-r-Besov} becomes $p\leq p_2$.
Inequality \eqref{eq:same-s-est-Besov} then follows from Lemma \ref{lem:mult-Besov-same-s-high}
, the fact that $\|v\|_{\Bv[s_2,p,q](\Reals^n)}\lesssim \|v\|_{\Bv[s_2,p_1,q](\Reals^n)}$
with implicit constant depending on $R$, and the embeddings
$\Bv[s,p_i,q_i](\Reals^n)\hookrightarrow \Bv[s,p_i,q](\Reals^n)$ for $i=1,2$.

Now suppose that $0<s<n/p_1$, that is, $1/r_1>0$.
In this case, \eqref{eq:same-s-r1-plus-r2-vs-r-Besov} becomes 
\begin{equation*}
\frac{1}{p} \ge \frac{1}{p_1} + \frac{1}{p_2} -\frac{s}{n} ,
\end{equation*}
and we have assumed additionally that if this is an equality then
$q_i\le r_i$ for $i=1,2$.  These are exactly the hypotheses of
Lemma \ref{lem:mult-Besov-same-s-low} and estimate
\eqref{eq:same-s-est-Besov} follows.

Finally, we look at the case $s=n/p_1$, that is, $1/r_1=0$.
Then $\min(1/r_1,1/r_2)=0$ since $1/r_2\ge 1/r_1$ and we have therefore 
assumed \eqref{eq:same-s-r1-plus-r2-vs-r-Besov} is strict.  Hence we can pick
$\eta \in (1,p_1)$ such that
\begin{equation*}
\frac{1}{p} > \frac{1}{\eta} + \frac{1}{p_2} -\frac{s}{n} .
\end{equation*}
as well. Since $s<n/\eta$, estimate \eqref{eq:same-s-est-Besov} now follows 
from Lemma \ref{lem:mult-Besov-same-s-low} and the continuous embedding 
$\Bv[s,p_1,q_1](\Reals^n)\hookrightarrow \Bv[s,\eta,q_1](\Reals^n)$.
\end{proof}

The restriction $s=s_1=s_2$ of the previous result can easily be relaxed 
with the help of Sobolev embeddings.

\begin{proposition}\label{prop:mult-Besov-s-pos}
Assume the multiplication hypothesis \ref{as:mult-hyp}
and that $\min(s_1,s_2,s)>0$. 
Pointwise multiplication of $C^\infty(\overline \Omega)$ 
functions extends to a continuous bilinear map 
$\Bv[s_1,p_1,q_1](\Omega)\times \Bv[s_2,p_2,q_2](\Omega)
\to \Bv[s,p,q](\Omega)$ so long as
\begin{align}
\min({s_1,s_2})&\ge s\\
\label{eq:max-r1-r2-vs-r-Besov-baby}
\max\left(\frac{1}{r_1},\frac{1}{r_2}\right) & \le  \frac{1}{r}\\
\label{eq:r1-r2-vs-r-Besov-baby}
\frac{1}{r_1} + \frac{1}{r_2} & \le \frac{1}{r}
\end{align}
with the following caveats:
\begin{itemize}
	\item Inequality \eqref{eq:r1-r2-vs-r-Besov-baby}
	is strict if $\min(1/r_1,1/r_2)=0$.
	\item If $s=s_i$ or $1/r=1/r_i$ for some $i$ then $1/q\le 1/q_i$.
	\item If $1/r_1+1/r_2=1/r$ then $1/r_i\le 1/q_i$ for $i=1,2$ and $1/q\le\min(1/q_1,1/q_2)$.
\end{itemize}
\end{proposition}

\begin{proof}
The proof follows the outline of Proposition \ref{prop:mult-Fsp-s-pos}, with some extra care
with respect to the fine parameter.

\textbf{Case: $\min(1/r_1,1/r_2)>0.$}\\
Suppose first that $1/r_1+1/r_2 < 1/r$. Since each
$1/r_i>0$ we have $\max(1/r_1,1/r_2)<1/r$ as well.  For any $s_i>s$ we can then lower $s_i$
slightly while preserving these strict inequalities \eqref{eq:max-r1-r2-vs-r-Besov-baby}--\eqref{eq:r1-r2-vs-r-Besov-baby}
to set $q_i$ to any desired value.  Hence, without loss of generality, if $s_i>s$ we can assume $q_i\le q$.
Otherwise, if $s_i=s$ we have assumed $q_i\le q$.  Following the technique of the corresponding case of
Proposition \ref{prop:mult-Fsp-s-pos} we then embed $\Bv[s_i,p_i,q_i](\Omega)\hookrightarrow \Bv[s,t_i,q_i]$
where $1/t_i-s/n = 1/r_i$ and apply Proposition \eqref{prop:mult-Besov-s-pos} to obtain
continuous multiplication 
$\Bv[s,t_1,q_1](\Omega)\times\Bv[s,t_2,q_2](\Omega)\hookrightarrow \Bv[s,p,q](\Omega)$, noting that 
inequalities \eqref{eq:same-s-r1-r2-vs-r-Besov}--\eqref{eq:same-s-q1q2-vs-q-Besov} all hold
and that the caveats are irrelevant.

If instead $1/r_1+1/r_2 = 1/r$ the same technique applies, except we cannot adjust any $q_i$ and must
assume $q_i\le q$ and $q_i\le r_i$ in advance, as we have done.

\textbf{Case: $\min(1/r_1,1/r_2)=0.$}\\
We have assumed that inequality \eqref{eq:r1-r2-vs-r-Besov-baby}
is strict.  Since each $1/r_i>0$ we also know that inequality \eqref{eq:max-r1-r2-vs-r-Besov-baby} is also strict.
Hence we can lower each $p_i$ while leaving $s_i$ fixed and maintain these strict inequalities. The result now
follows from the case $\min(1/r_1,1/r_2)>0$.

\textbf{Case: $\min(1/r_1,1/r_2)<0.$}\\
Without loss of generality we can assume $1/r_1\le 1/r_2$ and hence
$1/r_1<0$.  Since $s_2\ge s$ and $1/r_2\le 1/r$ and since we have additionally assumed that $q_2\le q$
if either $s=s_2$ or $r=r_2$ we know $\Bv[s_2,p_2,q_2](\Omega)\hookrightarrow \Bv[s,p,q](\Omega)$.
Hence we need only demonstrate continuity of multiplication 
$\Bv[s_1,p_1,q_1](\Omega)\times \Bv[s,p,q](\Omega)\to \Bv[s,p,q](\Omega)$.
Proposition \ref{prop:mult-Besov-same-s} implies multiplication $B^{s,t}_{q}\times 
\Bv[s,p,q](\Omega)\to \Bv[s,p,q](\Omega)$ is continuous so long as $t\ge p$ and $1/t-s/n<0$.  
Hence we are done if we can show that $\Bv[s_1,p_1,q_1](\Omega)$ embeds into 
some $\Bv[s,t,q](\Omega)$ satisfying these two conditions.  The proof now follows the corresponding
case of Proposition \ref{prop:mult-Fsp-s-pos} except we must now assume (as we have done) 
that $q_1\le q$ if $r_1=r$ in addition to assuming $q_1\le q$ if $s_1=s$ in order for the
requisite Sobolev embeddings to be valid. 
\end{proof}

What remains is the case where one or more of the indices $s_1$, $s_2$ and $s$ is nonpositive.
The following 3 propositions then explore the full range of the parameters $s_1$, $s_2$ and $s$,
and correspond to the cases where $\min(s_1,s_2)$ is negative, positive, and zero.
The main tool we employ is duality.

\begin{proposition}\label{prop:mult-Besov-neg}
Assume the multiplication hypothesis \ref{as:mult-hyp}
and that $\min(s_1,s_2)<0$.
Pointwise multiplication of $C^\infty(\overline \Omega)$ 
functions extends to a continuous bilinear map 
$\Bv[s_1,p_1,q_1](\Omega)\times \Bv[s_2,p_2,q_2](\Omega)
\to \Bv[s,p,q](\Omega)$ so long as
\begin{align}
\min(s_1,s_2)&\ge s\\
s_1+s_2&\ge 0\\
\max\left(\frac 1{r_1},\frac 1{r_2}\right) & \le  \frac 1r\\
\label{eq:r1-r2-vs-1-Besov-baby-4}
\frac{1}{r_1} + \frac{1}{r_2} & \le 1\\
\label{eq:r1-r2-vs-r-Besov-baby-4}
\frac{1}{r_1} + \frac{1}{r_2} & \le \frac{1}{r}
\end{align}
with the following caveats:
\begin{itemize}
	\item If $s=s_i$ or $1/r=1/r_i$ for some $i$ then $1/q\le 1/q_i$.
	\item If $s_1+s_2=0$ or $1/r_1+1/r_2=1$ then $1/q_1+1/q_2\ge1$.
	\item If $1/r_1+1/r_2=1/r$ then:
	\begin{itemize}
	\item $\min(1/r_1,1/r_2,1-1/r)\neq0$.	
	\item $1/q_1+1/q_2\ge1$ and $1/q \le 1/r$.
    \item If $s_i<0$ for some $i$ then $1/q\le 1/q_i$.
    \item If $s_i>0$ for some $i$ then $1/r_i \le 1/q_i$.
	\end{itemize}
\end{itemize}
\end{proposition}

\begin{proof}
Without loss of generality, assume that $s_1>0$ and $s_2<0$.
Following the duality technique of Proposition \ref{prop:mult-Fsp-s2-neg},
continuity of multiplication 
$\Bv[s_1,p_1,q_1](\Omega)\times \Bv[s_2,p_2,q_2](\Omega)\to \Bv[s,p,q](\Omega)$ 
follows from continuity of multiplication 
$\Bv[s_1,p_1,q_1](\Omega)\times \Bv[s_2^*,p_2^*,q_2^*](\Omega)\to \Bv[s^*,p^*,q^*](\Omega)$,
with 
\begin{equation*}
s_2^*=-s, \quad \frac1{p_2^*}=1-\frac1{p},\quad \frac1{q_2^*}=1-\frac1{q},
\quad
s^*=-s_2, \quad \frac1{p^*}=1-\frac1{p_2},\quad \frac1{q^*}=1-\frac1{q_2} ,
\end{equation*} 
and hence the problem falls under Proposition \ref{prop:mult-Besov-s-pos}.
Note that
\begin{equation*}
\frac1{r_2^*}=1-\frac1{r},\qquad \textrm{and} \qquad
\frac1{r^*}=1-\frac1{r_2} .
\end{equation*} 
For clarity, let us display here what Proposition \ref{prop:mult-Besov-s-pos} becomes when we write the conditions in terms of the un-starred parameters:
multiplication $\Bv[s_1,p_1,q_1](\Omega)\times \Bv[s_2^*,p_2^*,q_2^*](\Omega)\to \Bv[s^*,p^*,q^*](\Omega)$ is continuous so long as
\begin{align}
-s_2&\le \min({s_1,-s})\\
\max\left(\frac{1}{r_1},1-\frac{1}{r}\right) & \le  1-\frac{1}{r_2}\\
\label{eq:r1-r2-vs-r-Besov-baby-dual}
\frac{1}{r_1} + 1 - \frac{1}{r} & \le 1- \frac{1}{r_2}
\end{align}
with the following caveats:
\begin{itemize}
	\item Inequality \eqref{eq:r1-r2-vs-r-Besov-baby-dual}
	is strict if $\min(1/r_1,1-1/r)=0$.
	\item If $-s_2=s_1$ or $1-1/r_2=1/r_1$ then $1-1/q_2\le 1/q_1$.
	\item If $-s_2=-s$ or $1-1/r_2=1-1/r$ then $1-1/q_2\le 1-1/q$.
	\item If $1/r_1+1-1/r=1-1/r_2$ then $q_1\le r_1$, $1-1/q\ge 1-1/r$, and $1-1/q_2\le\min(1/q_1,1-1/q)$.
\end{itemize}
Noting that $\min(1/r_1,1/r_2,1-1/r)=0$ if and only if $\min(1/r_1,1-1/r)=0$, all these conditions are
implied by the hypotheses of the proposition.
\end{proof}

Now we treat the case $\min(s_1,s_2)>0$.

\begin{proposition}\label{prop:mult-Besov-pos}
Assume the multiplication hypothesis \ref{as:mult-hyp}
and that $\min(s_1,s_2)>0$.
Pointwise multiplication of $C^\infty(\overline \Omega)$ 
functions extends to a continuous bilinear map 
$\Bv[s_1,p_1,q_1](\Omega)\times \Bv[s_2,p_2,q_2](\Omega)
\to \Bv[s,p,q](\Omega)$ so long as
\begin{align}
\min(s_1,s_2)&\ge s\\
\label{eq:ri-vs-r-Besov-baby-5-b}
\max\left(\frac{1}{r_1},\frac{1}{r_2}\right) & \le  \frac{1}{r}\\
\label{eq:r1-r2-vs-1-Besov-baby-5-b}
\frac{1}{r_1} + \frac{1}{r_2} & \le 1\\
\label{eq:r1-r2-vs-r-Besov-baby-5-b}
\frac{1}{r_1} + \frac{1}{r_2} & \le \frac{1}{r}
\end{align}
with the following caveats:
\begin{itemize}
	\item If $s=s_i$ or $1/r=1/r_i$ for some $i$ then $1/q\le 1/q_i$.
	\item If $1/r_1+1/r_2=1$ then $1/q_1+1/q_2\ge1$.
	\item If $1/r_1+1/r_2=1/r$ then:
	\begin{itemize}
		\item $\min(1/r_1,1/r_2,1-1/r)\neq0$.	
		\item $1/q\le 1/q_i$ and $1/r_i \le 1/q_i$ for $i=1,2$.
		\end{itemize}
\end{itemize}
\end{proposition}

\begin{proof}
Suppose $s>0$. Then $1/r<1$, so $\min(1/r_1,1/r_2,1-1/r)=0$ if and only if $\min(1/r_1,1/r_2)=0$.
Using this observation it is easy to see that the hypotheses of the current result imply the hypotheses
of Proposition \ref{prop:mult-Besov-s-pos} and the desired continuity of multiplication when $s>0$ follows.

We split the remaining case $s\le0$ into the following 4 subcases.
\begin{itemize}
\item
If $1/r<1$, then let $0<\sigma<\min(s_1,s_2)$ be small enough so that
$\sigma/n+1/r < 1$ and define $\eta\in $ by $1/\eta-\sigma/n=1/r$.
Observe that $0<1/\eta<1$ since $1/r$ and $\sigma$ are both positive and 
since $\sigma$ is sufficiently small.
Using the fact that $\min(1/r_1,1/r_2)=0$ if and only if
$\min(1/r_1,1/r_2,1-1/r)=0$ when $1/r<1$ one  
then verifies that the hypotheses of Proposition \ref{prop:mult-Besov-s-pos}
are met to ensure multiplication is continuous $\Bv[s_1,p_1,q_1](\Omega)\times \Bv[s_2,p_2,q_2](\Omega)\rightarrow \Bv[\sigma,\eta,q](\Omega)$.
The result then follows from the embedding $\Bv[\sigma,\eta,q](\Omega)\hookrightarrow \Bv[s,p,q](\Omega)$.
\item
If $1/r=1$, we observe that inequality \eqref{eq:r1-r2-vs-r-Besov-baby-5-b} is strict. Indeed,
if $\min(1/r_1,1/r_2)\ge 0$ then $\min(1/r_1,1/r_2,1-1/r)=0$ and this holds by hypothesis,
and if $\min(1/r_1,1/r_2)< 0$ it is an easy consequence of inequalities 
\eqref{eq:ri-vs-r-Besov-baby-5-b} and \eqref{eq:r1-r2-vs-r-Besov-baby-5-b} and the fact that $1/r=1$.
Moreover each $1/r_i < 1/p_i < 1= 1/r$.  Hence we can pick $r'$ with $0<1/r'<1$ such that $1/r_i < 1/r'$
for $i=1,2$ and such that $1/r_1+1/r_2 < 1/r'$.  Now pick $0<\sigma<\min(s_1,s_2)$ such that
$\sigma$ is small enough so that $\eta$ defined by $1/\eta = 1/r'+\sigma/n$ lies in $(0,1)$.  The proof now
proceeds as in the previous subcase, verifying that the hypotheses of Proposition \ref{prop:mult-Besov-s-pos}
are met to get continuity of 
$\Bv[s_1,p_1,q_1](\Omega)\times \Bv[s_2,p_2,q_2](\Omega)\rightarrow \Bv[\sigma,\eta,q](\Omega)\hookrightarrow \Bv[s,p,q](\Omega)$.
Note that this verification benefits from the observation $1/r_1+1/r_2<1/r'$ strictly.

\item
If $1/r>1$ and $1/r_1+1/r_2<1$ we can again pick $r'>1$ with  $1/r_i<1/r'<1$ for $i=1,2$ and with
$1/r_1+1/r_1<1/r'<1$.  The proof now proceeds exactly as in the previous subcase.
\item
If $1/r>1$ and $1/r_1+1/r_2=1$, then we necessarily have the strict inequality $s<0$.
Without loss of generality we can assume $1/r_1\le 1/r_2$, and 
choose $\sigma$ such that $0>\sigma>\max(s,-s_1)$ and 
such that $1/\eta:=1/r_2+\sigma/n>0$. Observe $1/\eta<1$ as well.  One now verifies that
the hypotheses of Proposition \ref{prop:mult-Besov-neg} are met
to obtain continuity of multiplication 
$\Bv[s_1,p_1,q_1](\Omega)\times\Bv[\sigma,\eta,q_2](\Omega)\to \Bv[s,p,q](\Omega)$,
and indeed the check of its caveats is straightforward because we know $1/r_1+1/r_2<1/r$.
\end{itemize}
The proof is complete.
\end{proof}

Finally, we have the case $\min(s_1,s_2)=0$.

\begin{proposition}\label{prop:mult-Besov-zero}
Assume the multiplication hypothesis \ref{as:mult-hyp}
and that $\min(s_1,s_2)=0$.
Pointwise multiplication of $C^\infty(\overline \Omega)$ 
functions extends to a continuous bilinear map 
$\Bv[s_1,p_1,q_1](\Omega)\times \Bv[s_2,p_2,q_2](\Omega)
\to \Bv[s,p,q](\Omega)$ so long as
\begin{align}
s&\le 0\\
\max\left(\frac{1}{r_1},\frac{1}{r_2}\right) & \le  \frac{1}{r}\\
\label{eq:r1-r2-vs-1-Besov-baby-6}
\frac{1}{r_1} + \frac{1}{r_2} & \le 1\\
\label{eq:r1-r2-vs-r-Besov-baby-6}
\frac{1}{r_1} + \frac{1}{r_2} & \le \frac{1}{r}
\end{align}
with the following caveats:
\begin{itemize}
	\item If $s=s_i$ or $1/r=1/r_i$ for some $i$ then $1/q\le 1/q_i$.
	\item If $s_1+s_2=0$ or $1/r_1+1/r_2=1$ then $1/q_1+1/q_2\ge1$.
	\item If equality holds in \eqref{eq:r1-r2-vs-r-Besov-baby-6} then
	\begin{itemize}
	\item $\min(1/r_1,1/r_2,1-1/r)\neq0$.	
	\item $1/q_1+1/q_2\ge1$ and $1/q \le 1/r$.
    \item If $s_i=0$ for some $i$ then $1/q\le 1/q_i$.
    \item If $s_i$ has the same sign as $\max(s_1,s_2)$ for some $i$ then $1/r_i\le 1/q_i$.
	\item If $s=0$ then $1/q\le 1/q_i$ and $1/r_i\le 1/q_i$ for both $i=1,2$.
	\end{itemize}
	\item If $s_1=s_2=s=0$ then $\displaystyle \frac 1q\le \min\left(\frac 12,\frac 1r\right)$ and 
	$\displaystyle \max\left(\frac 12,\frac 1{r_i}\right)\le \frac 1{q_i}$ for both $i=1,2$.
\end{itemize}
\end{proposition}

\begin{proof}
First, consider the case $s_1>s_2=0>s$.
Pick $\sigma$ such that $0>\sigma>\max(s,-s_1)$ and such that $\sigma$ is close enough to zero
such that $1/\eta:=1/r_2+\sigma/n>0$. This is possible since $1/r_2=1/p_2>0$. 
Observe that $1/\eta<1$ as well since $1/r_2=1/p_2<1$ and since $\sigma<0$. 
By Sobolev embedding, $\Bv[s_2,p_2,q_2](\Omega)\hookrightarrow \Bv[\sigma,\eta,q_2](\Omega)$.
One now verifies that the hypotheses of Proposition \ref{prop:mult-Besov-neg}
are met to imply continuity of multiplication $\Bv[s_1,p_1,q_1](\Omega)\times \Bv[\sigma,\eta,q_2](\Omega)\rightarrow \Bv[s,p,q](\Omega)$.
The only interesting point in the verification is the fact that if \eqref{eq:r1-r2-vs-r-Besov-baby-6} is an equality 
then we have assumed $1/r_1\le 1/q_1$ and $1/q\le 1/q_i$, which are fine parameter requirements
needed to use Proposition \ref{prop:mult-Besov-neg} with $s_1>0$ and $\sigma<0$.

Next suppose that $s_1>s_2=s=0$.
Let $0<\sigma<s_1$ be small enough so that $\eta$, $\theta$, $\eta_2$ and $\theta_2$ defined by
$$
\frac1{\eta_2}-\frac\sigma{n} = \frac1{\theta_2}+\frac\sigma{n} = \frac1{r_2} ,
\qquad
\frac1{\eta}-\frac\sigma{n} = \frac1{\theta}+\frac\sigma{n} = \frac1{r}
$$
all lie in $(1,\infty)$. This is possible since $1/r=1/p\in (1,\infty)$ and similarly for $1/r_2$.
A computation verifies that multiplication is continuous
$\Bv[s_1,p_1,q_1](\Omega)\times \Bv[\sigma,\eta_2,q_2](\Omega)\rightarrow \Bv[\sigma,\eta,q](\Omega)$ due to Proposition \ref{prop:mult-Besov-pos},
and
$\Bv[s_1,p_1,q_1](\Omega)\times \Bv[-\sigma,\theta_2,q_2](\Omega)\rightarrow \Bv[-\sigma,\theta,q](\Omega)$ due to Proposition \ref{prop:mult-Besov-neg}.
Note that Proposition \ref{prop:mult-Besov-pos} requires $1/r_i\le 1/q_i$ and $1/q\le q_i$ for $i=1,2$ 
if \eqref{eq:r1-r2-vs-r-Besov-baby-6} is an equality
and that we have assumed this since $s=0$.
At this point, continuity of $\Bv[s_1,p_1,q_1](\Omega)\times \Bv[0,p_2,q_2](\Omega)\rightarrow \Bv[0,p,q](\Omega)$ is guaranteed by complex interpolation.

The next case we consider is $s_1=s_2=0>s$, which follows from the preceding case by the duality
argument of Proposition \ref{prop:mult-Fsp-s2-neg}.
Namely, continuity of $\Bv[0,p_1,q_1](\Omega)\times \Bv[0,p_2,q_2](\Omega)\rightarrow \Bv[s,p,q](\Omega)$ 
is implied by continuity of $\Bv[-s,p^*,q^*](\Omega)\times \Bv[0,p_2,q_2](\Omega)\rightarrow \Bv[0,p_1^*,q_1^*](\Omega)$,
with 
\begin{equation*}
\frac1{p_1^*}=1-\frac1{p},\quad \frac1{q_1^*}=1-\frac1{q},
\quad
\frac1{p^*}=1-\frac1{p_1},\quad \frac1{q^*}=1-\frac1{q_1} .
\end{equation*} 
A laborious but straightforward computation with these new parameters verifies that the interpolation 
technique of the previous case again applies.  One finds again that  
$1/r_i\le 1/q_i$ and $1/q\le q_i$ for $i=1,2$ are all required if \eqref{eq:r1-r2-vs-r-Besov-baby-6}
is an equality, and these are assumed since $s=0$.

Finally, suppose that $s_1=s_2=s=0$. 
We have assumed $\max(1/r_i,2)\le 1/q_i$ for $i=1,2$ and hence Proposition
\ref{prop:B-L-embedding} implies 
$\Bv[0,p_i,q_i](\Omega)\hookrightarrow L^{r_i}(\Omega)=L^{p_i}(\Omega)$.
We have also assumed $1/q\le \min(1/r,1/2)$ and hence
$L^p(\Omega)=L^r(\Omega)\hookrightarrow \Bv[0,p,q](\Omega)$. Since 
$1/p_1+1/p_2 = 1/r_1+1/r_2 \le 1/r = 1/p$ the continuous multiplication
is a consequence of H\"older's inequality.
\end{proof}

A routine verification shows that the 
hypotheses of Theorem \ref{thm:mult-Besov} imply the hypotheses of
Propositions \ref{prop:mult-Besov-neg} \ref{prop:mult-Besov-pos} and 
\ref{prop:mult-Besov-zero} in each of these special cases, which proves
Theorem \ref{thm:mult-Besov} in the event that $\Omega$ is a bounded smooth
domain.  As discussed following the statement of Assumption \ref{as:mult-hyp},
an extension/restriction argument then proves the result for an arbitrary
bounded domain.

\section*{Acknowledgements}
MH was supported in part by NSF Award DMS-2012857. 
DM was supported in part by NSF Award DMS-1263544. 

\bibliographystyle{amsalpha-abbrv}
\bibliography{InteriorEstimates,DM-auto-refs}

\end{document}